\documentclass[11pt,amssymb]{amsart}

\usepackage{txfonts}

\usepackage{mathtools} 

\usepackage{amsmath}
\usepackage{amssymb}
\usepackage{amsthm}

\usepackage{newlfont}

\usepackage{graphicx}

\numberwithin{equation}{section}

\newtheorem{thm}{Theorem}[section]
\newtheorem{cor}[thm]{Corollary}
\newtheorem{lem}[thm]{Lemma}
\newtheorem{prop}[thm]{Proposition}

\theoremstyle{definition}

\theoremstyle{remark}
\newtheorem{rem}[thm]{Remark}
\newtheorem*{rem*}{Remark}

\newcommand{\ed}{\end {document}}
\newcommand{\tjs}{{\tilde J}^s}

\newcommand{\lsm}{\lesssim}

\newcommand{\bmo}{\operatorname{BMO}}
\newcommand{\boo}{{\dot B^0}_{\infty,\infty}}

\title[Kato-Ponce and Leibniz]
{On Kato-Ponce and fractional Leibniz}

\author[D. Li]{Dong Li}
\address[madli@ust.hk]
{{\sc D. Li}: Department of Mathematics, The Hong Kong University of Science and
Technology, Clear Water Bay, Kowloon, Hong Kong}


\begin{document}
\begin{abstract}
We show that in the Kato-Ponce inequality $\|J^s(fg)-fJ^s g\|_p
\lesssim \| \partial f \|_{\infty} \| J^{s-1} g \|_p + \| J^s f \|_p
\|g\|_{\infty}$, the $J^s f$ term on the RHS can be replaced by
$J^{s-1} \partial f$. This solves a question 
raised in Kato-Ponce \cite{KP88}. We
propose 
a new fractional Leibniz rule for $D^s=(-\Delta)^{s/2}$ and similar
operators, generalizing the Kenig-Ponce-Vega estimate \cite{KPV93} to all $s>0$.
We also prove a family of generalized and refined Kato-Ponce type inequalities which include
many commutator estimates as special cases. To showcase the sharpness of the estimates at
various endpoint cases, we construct several counterexamples.
In particular, we
show that in the original Kato-Ponce inequality,
the $L^{\infty}$-norm on the RHS cannot be replaced by the weaker BMO norm. Some divergence-free
counterexamples are also included.
\end{abstract}

\maketitle

\tableofcontents 

\section{Introduction}
Let $J^s=(1-\Delta)^{s/2}$, $s\in \mathbb R$. In \cite{KP88},
Kato and Ponce proved the following fundamental commutator estimate:
\begin{align}
 &\| J^s(fg) - fJ^s g \|_p \lesssim_{s,p,d}
  \| J^s f \|_p \| g\|_{\infty}+\| \partial f \|_{\infty} \| J^{s-1} g \|_p,
  \label{a1}
\end{align}
where $s>0$, $1<p<\infty$, $\partial=(\partial_1,\cdots,\partial_d)$ (occasionally we also denote it
as $\nabla$)
and $f$, $g \in \mathcal S(\mathbb R^d)$. On page 892 of \cite{KP88} (see Remark 1.1(c) therein),
they conjectured that the $J^s f$ term on the RHS can be replaced by $J^{s-1} \partial f$.
The first purpose of this paper is to confirm that this is indeed the case.

\begin{thm} \label{thm1}
Let $s>0$, $1<p<\infty$. Then for any $f$, $g\in \mathcal S(\mathbb R^d)$,
\begin{align}
\| J^s(fg) - f J^s g\|_p \lesssim_{s,p,d}
\| J^{s-1} \partial f \|_p \| g\|_{\infty}+\| \partial f \|_{\infty} \| J^{s-1} g\|_p.
\label{a2}
\end{align}
Furthermore for $0<s\le 1$,
\begin{align*}
\| J^s(fg) - fJ^s g \|_p \lesssim_{s,p,d} \| J^{s-1} \partial f \|_p \| g\|_{\infty}.
\end{align*}
\end{thm}

More general results are available. See Theorem \ref{thm3} in the later part of this introduction.

Denote $D^s=(-\Delta)^{s/2}$. In \cite{KPV93}, Kenig, Ponce and Vega (KPV) proved the fundamental
estimate:
\begin{align}
\| D^s(fg) -f D^s g -g D^s f \|_p \lesssim_{s,s_1,s_2,p,p_1,p_2,d} \| D^{s_1} f\|_{p_1}
\| D^{s_2} g \|_{p_2}, \label{a4}
\end{align}
where $s=s_1+s_2$, $0<s,s_1,s_2<1$, $\frac 1p = \frac 1 {p_1} + \frac 1{p_2}$
and $1<p,p_1,p_2<\infty$. A natural question is to investigate what is the natural
formulation/generalisation of the KPV estimate when $s\ge 1$.
The second purpose of this paper is to solve this problem. Our theorem below establishes a new
fractional Leibniz rule for any $D^s$, $s>0$. It includes various end-point situations.

\begin{thm} \label{thm2}
\underline{Case 1}: $1<p<\infty$.

Let $s>0$ and $1<p<\infty$.  Then for any $s_1,s_2\ge 0$ with $s_1+s_2=s$, and
 any $f, g \in \mathcal S(\mathbb R^d)$, the following hold:

\begin{enumerate}

\item If  $1<p_1,p_2<\infty$ with $\frac 1p= \frac 1 {p_1} + \frac 1{p_2}$, then
\begin{align}
\| D^s(fg) &- \sum_{|\alpha|\le s_1} \frac 1 {\alpha !} \partial^{\alpha} f D^{s,\alpha} g
- \sum_{|\beta|\le s_2} \frac 1 {\beta!} \partial^{\beta} g D^{s,\beta} f \|_p \notag \\
& \lesssim_{s,s_1,s_2,p,p_1,p_2,d} \| D^{s_1} f \|_{p_1}
\| D^{s_2} g \|_{p_2}. \label{a5}
\end{align}

\item If $p_1=p$, $p_2=\infty$, then
\begin{align}
 \| D^s(fg)- & \sum_{|\alpha|<s_1}
 \frac 1 {\alpha!} \partial^{\alpha} f D^{s,\alpha} g - \sum_{|\beta| \le s_2}
 \frac 1 {\beta!} \partial^{\beta} g D^{s,\beta} f \|_p \notag \\
 & \lesssim_{s,s_1,s_2,p,d} \| D^{s_1} f \|_p \| D^{s_2} g \|_{\operatorname{BMO}}.  \label{a5c}
 \end{align}

 \item If $p_1=\infty$, $p_2=p$, then
 \begin{align}
 \| D^s(fg) &-\sum_{|\alpha|\le s_1} \frac 1 {\alpha!} \partial^{\alpha} f D^{s,\alpha} g
 - \sum_{|\beta|<s_2} \frac 1 {\beta!} \partial^{\beta} g D^{s,\beta} f \|_p \notag \\
& \lesssim_{s,s_1,s_2,p,d} \| D^{s_1} f \|_{\operatorname{BMO}} \| D^{s_2} g\|_p. \label{a5e}
\end{align}
\end{enumerate}

In the above we adopt the usual multi-index notation, namely $\alpha=(\alpha_1,\cdots,\alpha_d)$,
$\partial^{\alpha} = \partial_x^{\alpha} = \partial_{x_1}^{\alpha_1} \cdots \partial_{x_d}^{\alpha_d}$,
$|\alpha|=\sum_{j=1}^d \alpha_j$ and $\alpha ! =\alpha_1 ! \cdots \alpha_d !$. The operator $D^{s,\alpha}$ is
defined via Fourier transform\footnote{
The precise form of Fourier transform does not matter. But see \eqref{Fourier_def} for the definition used
in this paper.}
as
\begin{align*}
&\widehat{D^{s,\alpha} g} (\xi) = \widehat{ D^{s,\alpha}}(\xi) \hat g(\xi), \\
&\widehat{D^{s,\alpha}} (\xi) = i^{-|\alpha|} \partial_{\xi}^{\alpha}(|\xi|^s).
\end{align*}

\underline{Case 2}: $\frac 12 <p\le 1$.

If $\frac 12<p\le 1$, $s> \frac d p-d$ or $ s \in 2 \mathbb N$, then
 for any $1<p_1,p_2<\infty$ with $\frac 1p= \frac 1 {p_1} + \frac 1{p_2}$, any $s_1,s_2\ge 0$
 with $s_1+s_2=s$, 
\begin{align}
\| D^s(fg) &- \sum_{|\alpha|\le s_1} \frac 1 {\alpha !} \partial^{\alpha} f D^{s,\alpha} g
- \sum_{|\beta|\le s_2} \frac 1 {\beta!} \partial^{\beta} g D^{s,\beta} f \|_p \notag \\
& \lesssim_{s,s_1,s_2,p,p_1,p_2,d} \| D^{s_1} f \|_{p_1}
\| D^{s_2} g \|_{p_2}. \notag 
\end{align}

\end{thm}

\begin{rem}
As usual empty summation (such as $\sum_{0\le |\alpha|<0}$) is defined as zero. Let $0<s,s_1,s_2<1$ in
\eqref{a5}, then
\begin{align*}
\| D^s(fg) - f D^s g - g D^s f \|_p
\lesssim_{s,s_1,s_2,p_1,p_2,d} \| D^{s_1} f \|_{p_1} \| D^{s_2} g\|_{p_2},
\end{align*}
i.e. we recover the estimate \eqref{a4}. 
Let $s_1=0$, $s_2=s$, $0<s<1$ in \eqref{a5c}, then we get
\begin{align*}
\| D^s(fg) - g D^s f \|_p \lesssim \| f \|_p \| D^s g \|_{\operatorname{BMO}}
\lesssim \| f \|_p \| D^s g \|_{\infty}, \quad 1<p<\infty.
\end{align*}
Similarly let $s_1=s$, $0<s\le 1$ in \eqref{a5c}, then
we get
\begin{align}
\| D^s(fg) -f D^s g - g D^s f \|_p \lesssim_{s,p,d} \| D^s f \|_p \| g\|_{\bmo}.
\label{a5c_rm1}
\end{align}
Thus for $0<s\le 1$, $1<p<\infty$,
\begin{align}
\| D^s (fg) - f D^s g\|_p \lesssim_{s,p,d} \| D^s f \|_p \| g\|_{\infty}.
\label{a5c_rm2}
\end{align}
The inequality \eqref{a5c_rm2} for $0<s<1$, $1<p<\infty$ was proved in \cite{KPV93} (see also
Problem 2.7, Problem 2.8 on page 77 of  \cite{MS2}).
Let us also point it out that the estimate \eqref{a5c_rm1} suggests that, due to the presence of the
term $g D^sf$, the $L^{\infty}$ norm on the RHS of
\eqref{a5c_rm2} is sharp and cannot be replaced by the weaker BMO norm in general. See
Corollary \ref{prop922_1_cor1} for more definitive and precise statements.
\end{rem}

\begin{rem*}
If we slightly abuse our notation and denote $D^{s,\alpha}$ as a fractional
differentiation operator $\tilde D^{s-|\alpha|}$ (i.e. of order $s-|\alpha|$), then Theorem \ref{thm2} roughly says that
(suppressing constant coefficients)
\begin{align*}
D^s(fg) & \sim f D^s g + \partial f \tilde D^{s-1} g + \cdots + \partial^{[s_1]} f \tilde D^{s-[s_1]} g \notag \\
& \quad + g D^s f + \partial g \tilde D^{s-1} f + \cdots + \partial^{[s_2]} g \tilde D^{s-[s_2]} f \notag \\
& \quad + O(\| D^{s_1} f \|_{p_1} \cdot \| D^{s_2} g \|_{p_2}).
\end{align*}
In yet other words, 
neglecting error terms, the nonlocal operator $D^s$ can be effectively
regarded as a local operator obeying a generalized Leibniz rule.
\end{rem*}

Theorem \ref{thm2} actually holds for more general differential (and also pseudo-differential) operators. For
example, for $s>0$ suppose $A^s$ is a differential operator such that its symbol $\widehat{A^s}(\xi)$  is a homogeneous
function of degree $s$ and $\widehat{A^s}(\xi) \in C^{\infty}(\mathbb S^{d-1})$ (for example:
$\widehat{A^s}(\xi) = i |\xi|^{s-1} \xi_1$). Then we have the following corollary. We shall omit the proof since
it will be essentially a repetition of the proof of Theorem \ref{thm2}.

\begin{cor} \label{cor2}
 Let $1<p<\infty$ and $s>0$.
 Then for any $s_1,s_2\ge 0$ with $s_1+s_2=s$, and
 any $f, g \in \mathcal S(\mathbb R^d)$, the following hold:

\begin{enumerate}

\item If $1<p_1,p_2<\infty$ with $\frac 1p= \frac 1 {p_1} + \frac 1{p_2}$, then
\begin{align}
\| A^s(fg) &- \sum_{|\alpha|\le s_1} \frac 1 {\alpha !} \partial^{\alpha} f A^{s,\alpha} g
- \sum_{|\beta|\le s_2} \frac 1 {\beta!} \partial^{\beta} g A^{s,\beta} f \|_p \notag \\
& \lesssim_{A^s,s,s_1,s_2,p,p_1,p_2,d} \| D^{s_1} f \|_{p_1}
\| D^{s_2} g \|_{p_2}. \label{a6}
\end{align}

\item If $p_1=p$, $p_2=\infty$, then
\begin{align}
 \| A^s(fg)- & \sum_{|\alpha|<s_1}
 \frac 1 {\alpha!} \partial^{\alpha} f A^{s,\alpha} g - \sum_{|\beta| \le s_2}
 \frac 1 {\beta!} \partial^{\beta} g A^{s,\beta} f \|_p \notag \\
 & \lesssim_{A^s,s,s_1,s_2,p,d} \| D^{s_1} f \|_p \| D^{s_2} g \|_{\operatorname{BMO}}.  \label{a6a}
 \end{align}

 \item If $p_1=\infty$, $p_2=p$, then
 \begin{align}
 \| A^s(fg) &-\sum_{|\alpha|\le s_1} \frac 1 {\alpha!} \partial^{\alpha} f A^{s,\alpha} g
 - \sum_{|\beta|<s_2} \frac 1 {\beta!} \partial^{\beta} g A^{s,\beta} f \|_p \notag \\
& \lesssim_{A^s,s,s_1,s_2,p,d} \| D^{s_1} f \|_{\operatorname{BMO}} \| D^{s_2} g\|_p. \notag
\end{align}
\end{enumerate}

In the above the operator $A^{s,\alpha}$ is defined via Fourier transform as
\begin{align*}
&\widehat{A^{s,\alpha} g} (\xi) = i^{-|\alpha|} \partial_{\xi}^{\alpha} \Bigl( \widehat{ A^{s}}(\xi)
\Bigr) \hat g(\xi).
\end{align*}

\end{cor}
\begin{rem*}
A further interesting problem is to identify the explicit dependence on the operators $A^s$ in
the implied constants.  This is useful in some problems connected with a family of operators rather
than a single operator.
\end{rem*}
\begin{rem*}
One should note that the error terms on the RHS of the above inequalities involve $D^s$ rather than
$A^s$. In particular for $0<s<1$, we have the following commonly used ones:
\begin{align*}
&\| A^s (fg) - f A^s g - g A^s f \|_p 
\lesssim \| f \|_{p_1} \| D^s g \|_{p_2}, \quad 1<p_1,p_2<\infty,\; \frac 1 {p_1}+\frac 1{p_2}=\frac 1p;\\
&\| A^s (fg) - g A^s f \|_p \lesssim \| f \|_p \| D^s g \|_{\operatorname{BMO}},  \quad 1<p<\infty;\\
&\| A^s (fg) - fA^s g - g A^s f \|_p \lesssim \| f \|_{\operatorname{BMO}} \| D^s g \|_p, \quad 1<p<\infty.
\end{align*}
Also for $1\le s<2$, 
\begin{align*}
&\| A^s (fg) - f A^s g - g A^s f - \nabla g \cdot A^{s,\nabla} f \|_p 
\lesssim \| f \|_{p_1} \| D^s g\|_{p_2}, \quad 1<p_1,p_2,p<\infty,\notag \\
&\qquad\qquad \; \frac 1{p_1}
+\frac 1 {p_2} = \frac 1p;\\
&\| A^s(fg) - g A^s f - \nabla g \cdot A^{s,\nabla} f \|_p 
\lesssim \| f \|_p \| D^s g \|_{\operatorname{BMO}}, \quad 1<p<\infty; \\
&\| A^s(fg) -f A^s g - g A^s f - \nabla g \cdot A^{s,\nabla} f \|_p
\lesssim \| f \|_{\operatorname{BMO}} \| D^s g \|_p, \quad 1<p<\infty.
\end{align*}
where $\widehat{A^{s,\nabla}}(\xi) = -i \nabla_{\xi} (\widehat{A^s}(\xi) )$.
\end{rem*}

\begin{rem*}
At this point it is useful to point out the explicit connection with the classical Leibniz rule and do a sanity
check of our formulae. Let $\gamma=(\gamma_1,\cdots,\gamma_d)$ be a multi-index.
 Recall that the classic Leibniz formula for a differential operator $\partial^{\gamma}$ takes the form
 \begin{align}
 \partial^{\gamma}(fg) = \sum_{\alpha \le \gamma}
 \frac{\gamma !} {\alpha ! (\gamma-\alpha)!} \partial^{\alpha} f \partial^{\gamma-\alpha} g. \label{a7}
 \end{align}
Set $s=|\gamma|$ and denote $A^s = \partial^{\gamma}$. Then easy to check that (in our notation)
\begin{align*}
\widehat{A^{s,\alpha}} (\xi) & = i^{-|\alpha|} \partial_{\xi}^{\alpha} ( i^{|\gamma|} \xi^{\gamma} ) \notag \\
& = i^{|\gamma|-|\alpha|}
\frac{\gamma !} {(\gamma-\alpha)!} \xi^{\gamma-\alpha} = \frac{\gamma !}{(\gamma-\alpha)!}
i^{|\gamma-\alpha|} \xi^{\gamma-\alpha}.
\end{align*}
Thus $A^{s,\alpha} = \frac{\gamma !}{ (\gamma-\alpha)!} \partial^{\gamma-\alpha}$.  Clearly \eqref{a6} then takes
the form
\begin{align}
\| \partial^{\gamma}(fg) -\sum_{|\alpha|\le s_1}
&\frac{\gamma!}{\alpha! (\gamma-\alpha)!} \partial^{\alpha} f \partial^{\gamma-\alpha} g 
 -\sum_{|\beta|\le s_2} \frac{\gamma !} {\beta! (\gamma-\beta)!} \partial^{\beta} g \partial^{\gamma-\beta} f
\|_p \notag \\
& \lesssim \| D^{s_1} f \|_{p_1} \| D^{s_2} g \|_{p_2}. \label{a8}
\end{align}
Note that \eqref{a8} captures essentially the main terms in \eqref{a7}. In this sense the formula \eqref{a8}
provides a ``natural'' generalization of the classical Leibniz formula \eqref{a7} to the fractional setting.
\end{rem*}

\begin{rem*}
One need not worry about the possibility that $\partial^{\alpha} f \partial^{\gamma-\alpha} g$ may coincide
with the terms $\partial^{\beta} g \partial^{\gamma-\beta} f$. This is because due to the constraint $|\alpha|\le s_1$,
$|\beta| \le s_2$, such two terms possibly coincide only when $|\alpha|=s_1$, $|\beta|=s_2$ and $s_1$, $s_2$ are
both integers. But in this case $\partial^{\alpha} f \partial^{\gamma-\alpha} g $  can be easily bounded by
$\| D^{s_1} f\|_{p_1} \| D^{s_2} g \|_{p_2}$ and thus can be included in the error term on the RHS.
\end{rem*}

\begin{rem*}
One may wonder what is the origin for the appearance of the operators $D^{s,\alpha}$, $A^{s,\alpha}$ in the new Leibniz rule. To see this, consider the symbol $\sigma(\xi,\eta)=|\xi+\eta|^s$ corresponding to
the operator $D^s$ which on the Fourier side reads:
\begin{align*}
\sigma_D(f,g) =\frac 1 {(2\pi)^{2d}} \int \sigma(\xi, \eta) \hat f(\xi) \hat g(\eta ) e^{i x\cdot (\xi+\eta)} d
\xi d\eta.
\end{align*}
Our new fractional Leibniz rule amounts to the justification of a Taylor expansion:
\begin{align*}
\sigma(\xi, \eta) = \sum_{|\alpha|\le s_1} \frac {(\partial^{\alpha}_\xi \sigma) (0,\eta)}
{\alpha !} \xi^{\alpha} + \sum_{|\beta|\le s_2} 
\frac{ (\partial_{\eta}^{\beta} \sigma)(\xi,0) } {\beta!} \eta^{\beta} +\sigma_{\operatorname{err}} ,
\end{align*}
where (in the non-endpoint situation $1<p_1,p_2<\infty$)
\begin{align*}
\| \sigma_{\operatorname{err} } (f,g) \|_p \lesssim \| D^{s_1} f \|_{p_1} \| D^{s_2} g\|_{p_2}.
\end{align*}
One may further write 
\begin{align*}
 {(\partial^{\alpha}_\xi \sigma) (0,\eta)}
 \xi^{\alpha} = i^{-\alpha}  {(\partial^{\alpha}_\xi \sigma )(0,\eta)} 
 \cdot (i\xi)^{\alpha}
 \end{align*}
and clearly the first factor accords with the $D^{s,\alpha}$ operator mentioned earlier. In short summary our new Leibniz rule is simply  a separable Taylor expansion of the symbol 
$\sigma(\xi,\eta)$ up to $O(D^{s_1} f D^{s_2} g)$ in an appropriate sense!

\end{rem*}

In the following remarks, we discuss a few applications of the new Leibniz rule. 

\begin{rem}
Let $m\ge 1$ be an integer and recall $\partial=(\partial_1,\cdots,\partial_d)$ on $\mathbb R^d$. For any
integer $n\ge 1$ denote
\begin{align*}
\| \partial^n f \|_p = \sum_{|\alpha|=n} \| \partial^{\alpha} f \|_p.
\end{align*}
The classical Kato-Ponce inequality for the usual differential operator $\partial^m$ (WLOG one may
assume $m\ge 3$) is:
\begin{align*}
\sum_{|\gamma|=m} \| \partial^{\gamma} (fg) - f \partial^{\gamma} g \|_p \lesssim_{p,d} \| \partial^m f \|_p \| g\|_{\infty}
+ \| \partial f \|_{\infty} \| \partial^{m-1} g \|_p, \qquad 1<p<\infty.
\end{align*}
The proof of the above inequality, roughly speaking, is a two-step procedure. Step 1: Leibniz. One writes
\begin{align*}
\partial^{\gamma} (fg) - f \partial^{\gamma} g =g \partial^{\gamma} f 
+\sum_{|\alpha|=1} \binom{\gamma}{\alpha}\partial^{\alpha} f \partial^{\gamma-\alpha} g +  \sum_{2\le |\alpha|\le m-1, \alpha+\beta=\gamma} 
\binom{\gamma}{\alpha} \partial^{\alpha} f \partial^{\beta} g.
\end{align*}
Step 2: (Gagliardo-Nirenberg) Interpolation. For $2\le |\alpha|\le m-1$, by using\footnote{See Lemma
\ref{lemc2} for a general proof of the interpolation inequalities.} 
\begin{align*}
 \| \partial^{\alpha} f \|_{\frac{p(m-1)}{|\alpha|-1}} \lesssim \| \partial^m f \|_p^{\frac{|\alpha|-1}{m-1}}
 \| \partial f \|_{\infty}^{\frac{m-|\alpha|} {m-1}},\\
 \| \partial^{\beta} g \|_{\frac{p(m-1)}{|\beta|}}
 \lesssim \| \partial^{m-1} g \|_p^{\frac{|\beta|}{m-1}} \| g \|_{\infty}^{1-\frac{|\beta|} {m-1} },
 \end{align*}
we get
\begin{align*}
\|\partial^{\alpha}f  \partial^{\beta} g \|_p & \lesssim \| \partial^{\alpha} f \|_{\frac{p(m-1)}{|\alpha|-1}}
\| \partial^{\beta} g \|_{\frac{p(m-1)}{|\beta|}}
\lesssim \| \partial^m f \|_p \| g\|_{\infty}
+ \| \partial f \|_{\infty} \| \partial^{m-1} g \|_p.
\end{align*}

Thanks to the new Leibniz rule, we can effectively regard the nonlocal
operator $D^s$ as the local one and ``revive" the above classical proof of Kato-Ponce to work for the nonlocal
case.  Indeed consider the case $s>1$ and $1<p<\infty$, by using Theorem \ref{thm2} with $s_1=s$,
$s_2=0$, we get
\begin{align*}
\| D^s (fg) -\sum_{|\alpha|<s}
\frac 1 {\alpha!} \partial^{\alpha} f D^{s,\alpha} g - g D^s f\|_p
\lesssim \| g \|_{\operatorname{BMO}} \| D^s f \|_p.
\end{align*}
We then have
\begin{align*}
\| D^s (fg) -f D^s g \|_p \lesssim \| \partial f \|_{\infty} \| D^{s-1} g \|_p + \| g \|_{\infty}
\| D^s f \|_p + \sum_{2\le |\alpha|<s} \| \partial^{\alpha} f D^{s, \alpha} g \|_p.
\end{align*}
Now observe that if $1<s\le 2$ the above summation in $\alpha$ is not present. For $s>2$, by using 
\begin{align*}
\| \partial^{\alpha} f \|_{\frac{p (s-1)} {|\alpha|-1} } \lesssim \| D^s f \|_p^{\frac{|\alpha|-1}{s-1} }
\| \partial f \|_{\infty}^{\frac{ s-|\alpha|} {s-1} }; \\
\| D^{s,\alpha} g \|_{\frac{p(s-1)}{s-|\alpha|} } 
\lesssim \| D^{s-1} g \|_p^{\frac{s-|\alpha|}{s-1} } \| g \|_{\infty}^{\frac{|\alpha|}{s-1}},
\end{align*}
we get the desired inequality:
\begin{align*}
\| D^s (fg) -f D^s g\|_p \lesssim \| \partial f \|_{\infty} \| D^{s-1} g \|_p + \| g \|_{\infty}
\| D^s f \|_p.
\end{align*}
\end{rem}

\begin{rem}
In recent \cite{Yezhuan}, Ye considered the 2D  incompressible B\'{e}nard equation in which the main
unknowns are the velocity $u:\,\mathbb R^2\to \mathbb R^2$ and the temperature $\theta:\, \mathbb R^2\to \mathbb R$.
Define $\mathcal R_1=(-\Delta)^{-\frac 12} \partial_{x_1}$. Ye proved the commutator estimate
\begin{align} \label{Ye_comm}
\| [\mathcal R_1, u\cdot \nabla ] \theta \|_{L^p(\mathbb R^2)}
\lesssim_{p,p_1,p_2} \| \nabla u \|_{L^{p_1}(\mathbb R^2)} \| \theta \|_{L^{p_2}(\mathbb R^2)},
\qquad \text{if $\nabla \cdot u=0$},
\end{align}
where $1<p,p_1,p_2<\infty$ with $\frac 1p=\frac 1{p_1}+\frac 1{p_2}$. We now explain how to deduce \eqref{Ye_comm}
from Corollary \ref{cor2}. For $j=1,2$, define $A_j =\mathcal R_1 \partial_j$. Set
$s=s_1=1$, $s_2=0$ in \eqref{a6} and we get
\begin{align*}
\| A_j (u_j \theta) - u_j A_j \theta -\sum_{|\alpha|=1}
\partial_{\alpha} u_j A_j^{\alpha} \theta - \theta A_j u_j \|_p  
\lesssim_{p,p_1,p_2} \| D u_j \|_{p_1} \| \theta \|_{p_2},
\end{align*}
where
\begin{align*}
\widehat{A_j^{\alpha}}(\xi) = -i \partial_{\xi}^{\alpha} ( |\xi|^{-1} \cdot (i\xi_1)\cdot (i \xi_j)).
\end{align*}
Easy to check that $\| A_j^{\alpha} \theta \|_{p_2} \lesssim_{p_2} \| \theta \|_{p_2}$, and we get
\begin{align*}
\| \sum_{j=1}^2 ( A_j (u_j \theta) - u_j A_j \theta) \|_p
&\lesssim_{p,p_1,p_2} \|D u\|_{p_1} \|\theta \|_{p_2} \\
& \lesssim_{p,p_1,p_2} \| \nabla u \|_{p_1} \| \theta \|_{p_2}.
\end{align*}
The estimate \eqref{Ye_comm} then easily follows. Note that one actually does not need to use the
divergence-free condition $\nabla \cdot u=0$ since
\begin{align*}
\| (-\Delta)^{-\frac 12} \partial_{x_1} ( (\nabla \cdot u) \theta ) \|_p \lesssim_{p,p_1,p_2} \| \nabla u \|_{p_1}
\|\theta \|_{p_2}.
\end{align*}
\end{rem}

\begin{rem}
In recent \cite{FR_JFA}, Fefferman, McCormick, Robinson and Rodrigo (FMRR) considered a class of non-resistive MHD
equations and proved a new Kato-Ponce type inequality
\begin{align} \label{FMRR_ineq}
\| D^s ( (u\cdot \nabla)B) - (u\cdot \nabla)(D^s B)\|_{L^2(\mathbb R^d)}
\lesssim_{s,d} \| \nabla u \|_{H^s(\mathbb R^d)} \| B\|_{H^s(\mathbb R^d)},
\end{align}
where $s>d/2$, $u=(u_1,\cdots,u_d)$, $B=(B_1,\cdots, B_d)$, $\nabla u$, $B\in H^s(\mathbb R^d)$. The condition
$s>d/2$ is critical for $L^{\infty}$-embedding. For dimension $d=2$, $s=d/2=1$, 
they exhibited a pair
of divergence-free $u\in H^2(\mathbb R^2)$, $B\in H^1(\mathbb R^2)$, such that
\begin{align*}
\partial_k ( (u\cdot \nabla) B) - (u\cdot \nabla)(\partial_k B) = ( (\partial_k u) \cdot \nabla) B
\notin L^2(\mathbb R^2).
\end{align*}
In this paper, Corollary \ref{cor2} can be used to generalize the FMRR inequality
\eqref{FMRR_ineq} to all $1<p<\infty$, $s>d/p$. It is also possible to give some refined inequalities
for the borderline case $s=d/p$ and construct divergence-free counterexamples
for the nonlocal operator $D^{d/p}$ for all $1<p<\infty$. 
 See Corollary
\ref{cor3c}, Remark \ref{rev_rem5.5} and Remark \ref{rev_rem5.6} in Section 5 for more details.

\end{rem}

\begin{rem}
In \cite{CCCGW12}, Chae, Constantin, Cordoba, Gancedo and Wu considered several generalized surface
quasi-geostrophic models with singular velocities. One of the models considered therein is the following:
\begin{align*}
\begin{cases}
\partial_t \theta + v \cdot \nabla \theta =0, \quad (t,x) \in (0,\infty) \times \mathbb R^2;\\
v = D^{-1+\gamma} \nabla^{\perp} \theta=(-D^{-1+\gamma} \partial_2 \theta,
D^{-1+\gamma} \partial_1 \theta), \quad 0<\gamma<1;\\
\theta(0,x)= \theta_0(x).
\end{cases}
\end{align*}
Note that $v$ scales as $D^{\gamma} \theta$ which is quite singular for $0<\gamma<1$ and this renders
the local wellposedness a very nontrivial problem.
For $\theta_0 \in H^m(\mathbb R^2)$ with $m\ge 4$ being an integer, they proved local wellposedness
by using skew-symmetry of the operator $D^{-1+\gamma} \nabla^{\perp}$ (to rewrite the nonlinear term
in terms of a commutator) and a commutator estimate of the form (see Proposition 2.1 therein)
\begin{align*}
\sum_{j=1}^2 \| D^s \partial_j (gf ) - g D^s \partial_j f \|_{L^2(\mathbb R^2)}
& \lesssim_s \| D^s f \|_2 \| \widehat{D g} (\eta) \|_{L^1_{\eta}}
+ \| f \|_2 \| \widehat{D^{1+s} g }(\eta) \|_{L^1_{\eta}} \notag \\
& \lesssim_{s,\epsilon} \| D^s f \|_2 \| g \|_{H^{2+\epsilon} }
+ \| f \|_2 \| g \|_{H^{2+s+\epsilon}},
\end{align*}
where $s\in \mathbb R$ and $\epsilon>0$. We now show how to use our new Leibniz rule
 to obtain a more refined result, namely sharp local wellposedness in $H^s$ for any $s>2+\gamma$.  Indeed by taking $f=D^{\gamma-1}
 \nabla^{\perp} \theta$, $g=\nabla \theta$, $s_1=1$, $s_2=s-1$, $p=p_2=2$, $p_1=\infty$ in 
 Theorem \ref{thm2}, we get
 \begin{align*}
 & \| D^s (fg)
 -\sum_{|\alpha|\le 1} \frac 1 {\alpha!} \partial^{\alpha} f D^{s,\alpha} g
 -\sum_{|\beta|<s-1} \frac 1 {\beta!} \partial^{\beta} g D^{s,\beta} f \|_2 \notag \\
 & \qquad\quad \lesssim \| D f \|_{\operatorname{BMO}} \| D^{s-1} g \|_2
 \lesssim \| \theta \|_{H^s}^2.
 \end{align*}
 Now consider the contribution of each summand (in either $\alpha$ or $\beta$) separately.
 \begin{itemize}
 \item $\alpha=0$. Obviously $f D^s g = D^{\gamma-1} \nabla^{\perp} \theta
 \cdot \nabla D^s \theta$ and
 \begin{align*}
 \int_{\mathbb R^2} (D^{\gamma-1} \nabla^{\perp} \theta \cdot \nabla D^s \theta)
  D^s \theta dx =0
  \end{align*}
  by using integration by parts.
  \item $|\alpha|=1$. In this case
  \begin{align*}
  \| \partial^{\alpha} f D^{s,\alpha} g \|_2
  & = \| \partial^{\alpha} D^{\gamma-1} \nabla^{\perp} \theta \cdot D^{s,\alpha} \nabla \theta\|_2
  \notag \\
  & \lesssim \| \partial^{\alpha} D^{\gamma-1}
  \nabla^{\perp} \theta \|_{\infty}
  \cdot \| D^{s,\alpha} \nabla \theta\|_2
  \lesssim \| \theta \|_{H^s} \cdot \| \theta \|_{H^s}.
  \end{align*}
 
 \item $\beta=0$. Observe that $g D^s f = \nabla \theta \cdot D^{\gamma-1} \nabla^{\perp} D^s \theta$
 \begin{align*}
\int ( \nabla \theta \cdot D^{\gamma-1} \nabla^{\perp} D^s \theta) D^s \theta dx
= - \int D^{\gamma-1} \nabla^{\perp}
\cdot ( D^s\theta \nabla \theta )  D^s \theta dx.
\end{align*}
We then write  (this is the elegant trick used in \cite{CCCGW12})
\begin{align*}
& \int ( \nabla \theta \cdot D^{\gamma-1} \nabla^{\perp} D^s \theta) D^s \theta dx
=- \frac 12 \int D^s \theta ( D^{\gamma-1} \nabla^{\perp} 
\cdot(D^s \theta \nabla \theta) \notag \\
&\qquad \quad- \nabla \theta \cdot D^{\gamma-1} \nabla^{\perp} D^s \theta) dx.
\end{align*}
 By Corollary \ref{cor2} with $A^{\gamma}= D^{\gamma-1} \nabla^{\perp}$, $f= D^s
 \theta$, $g= \nabla \theta$, $p=p_1=2$, $p_2=\infty$, we get
 \begin{align*}
 \| D^{\gamma-1} \nabla^{\perp} 
\cdot(D^s \theta \nabla \theta) - \nabla \theta \cdot D^{\gamma-1} \nabla^{\perp} D^s \theta
\|_2 \lesssim \| D^s \theta \|_2 \| D^{\gamma} \nabla \theta \|_{\operatorname{BMO}}
\lesssim \| \theta \|_{H^s}^2.
\end{align*}

\item $1\le |\beta|\le s-2$.  If $1\le |\beta|<s-2$, then clearly by Sobolev embedding
\begin{align*}
\| \partial^{\beta} g D^{s,\beta} f \|_2 \lesssim \| \partial^{\beta} \nabla \theta \|_{\infty}
\| D^{s,\beta} D^{\gamma-1} \nabla^{\perp} \theta \|_2
\lesssim \| \theta \|_{H^s}^2.
\end{align*}
Similarly if $|\beta|=s-2$ (in this case $s$ will be an integer), 
\begin{align*}
\| \partial^{\beta} g D^{s,\beta} f \|_2
\lesssim \| \partial^{s-1} \theta \|_{\infty-} \| D^{s,\beta} D^{\gamma-1} \nabla^{\perp} \theta
\|_{2+}  \lesssim \| \theta \|_{H^s}^2.
\end{align*}

\item $s-2<|\beta|<s-1$. We have
\begin{align*}
\| \partial^{\beta} g D^{s,\beta} f \|_2 \lesssim \| \partial^{\beta} \nabla \theta \|_{(\frac 12
-\frac{s-(|\beta|+1)}2)^{-1}} \| D^{s,\beta} D^{\gamma-1} \nabla^{\perp} \theta
\|_{(\frac {s-(|\beta|+1)} 2)^{-1}} \lesssim \| \theta \|_{H^s}^2.
\end{align*}
 \end{itemize}
 
Collecting the above estimates, we get for $s>\gamma+2$,
\begin{align*}
 \frac d {dt} (\| \theta \|_{H^s}^2) \lesssim \| \theta \|_{H^s}^3
 \end{align*}
 which (together with standard mollification/regularisation arguments) 
 easily yields the desired local wellposedness in $H^s$. 
\end{rem}

In Section 5 Theorem \ref{thm3a}, we state and prove a family of refined Kato-Ponce
inequalities for the operator $D^s=(-\Delta)^{s/2}$. Those inequalities are proved with the
help of Theorem \ref{thm2}. On the other hand, for the inhomogeneous operator $J^s$, we have
the following generalised inequalities. Note that in the following inequality, some of the endpoint
cases can be further improved along similar lines as in Theorem \ref{thm3a}. For simplicity of presentation (and practical
considerations), here we only state the simplest version.

\begin{thm} \label{thm3}
Let $1<p<\infty$. Let $1<p_1,p_2,p_3,p_4\le \infty$ satisfy $\frac 1{p_1}+\frac 1{p_2} =\frac 1 {p_3}
+\frac 1 {p_4} = \frac 1 p$. Then for any $f,\, g\in \mathcal S(\mathbb R^d)$, the following hold:
\begin{itemize}
\item If $0<s\le 1$, then
\begin{align*}
\| J^s(fg) - f J^s g \|_p \lesssim_{s,p_1,p_2,p,d} \| J^{s-1} \partial f \|_{p_1} \| g\|_{p_2}.
\end{align*}

\item If $s>1$, then
\begin{align*}
\| J^s (fg) -f J^s g\|_p \lesssim_{s,p_1,p_2,p_3,p_4,p,d}
\| J^{s-1} \partial f \|_{p_1} \|g\|_{p_2} + \| \partial f \|_{p_3} \| J^{s-2} \partial g\|_{p_4}.
\end{align*}
\end{itemize}

\end{thm}

 There are also many other reformulations and generalisations of the Kato-Ponce commutator inequalities
(cf. \cite{BO_recent} and the references therein). One popular variant
 is the following fractional Leibniz rule which holds for any
$s>0$, $f$, $g\in \mathcal S(\mathbb R^d)$:
\begin{align}
 \| D^s (fg) \|_r \le C_{s,d,p_1,p_2,q_1,q_2} \cdot ( \| D^s f \|_{p_1} \| g \|_{q_1} +
 \|D^s g \|_{p_2} \| f \|_{q_2}), \label{1}
\end{align}
where $\frac 1 r = \frac 1 {p_1} + \frac 1 {q_1} = \frac 1 {p_2} + \frac 1 {q_2}$, $1<r<\infty$,
$1<p_1,p_2,q_1,q_2\le \infty$, and $C_{s,d,p_1,p_2,q_1,q_2}>0$ is a constant depending
only on $(s,d,p_1,p_2,q_1,q_2)$. One should note that the same
inequality also holds for the inhomogeneous operator $J^s$.
Recently Grafakos, Oh \cite{GO_CPDE} and Muscalu, Schlag \cite{MS2} have extended
the inequality \eqref{1} to the wider range $1/2<r<\infty$ under the assumption that $s>\max(0, \frac dr -d)$
or $s\in 2\mathbb N$. The end-point case $r=\infty$  was conjectured in
Grafakos, Maldonado and Naibo \cite{GMN} and solved in recent \cite{BL14}.

The rest of this paper is organized as follows. In Section 2 we collect some notation used in this paper and
also some preliminary lemmas. In Section 3 we prove an important paraproduct estimate and some auxiliary lemmas.
Section 4 is devoted to the proof of Theorem \ref{thm2}. In Section 5 we prove several refined inequalities
for the operator $D^s$. In Section 6 we prove refined Kato-Ponce inequalities for the operator $J^s$.
Section 7 contains several counterexamples for the operator $J^s$. Section 8 is devoted to the proof of
Theorem \ref{thm3}. Section 9 contains further divergence-free counterexamples.


\section{Notation and preliminaries}

In this section we introduce some notation and collect some preliminaries used in this paper.

We adopt the following convention for the Fourier transform pair:
\begin{align} \label{Fourier_def}
 & (\mathcal F f)(\xi)=\hat f (\xi) = \int_{\mathbb R^d} f(x) e^{-i x\cdot \xi} dx, \notag \\
 &f(x) = \frac 1 {(2\pi)^d} \int_{\mathbb R^n} \hat f(\xi) e^{i x \cdot \xi} d\xi.
\end{align}
The inverse Fourier transform is sometimes denoted as $\mathcal F^{-1}$ so that 
$$f(x)= (\mathcal F^{-1} (\hat f) )(x).$$

For any $x\in \mathbb R^d$, we denote $\langle x \rangle =(1+|x|^2)^{1/2}$. Similarly for any
$s\in \mathbb R$ we define $\langle \nabla \rangle^s  $ via its Fourier transform
$\widehat{\langle \nabla \rangle^s}(\xi) = (1+|\xi|^2)^{s/2}$. In this notation $J^s =(I-\Delta)^{s/2}
=\langle \nabla \rangle^s$.

For any real number $a\in \mathbb R$, we denote by $a+$ the quantity $a+\epsilon$ for sufficiently small
$\epsilon>0$. The numerical value of $\epsilon$ is unimportant and the needed
smallness of $\epsilon$ is usually clear from the context. 
The notation $a-$ is similarly defined. 

We denote by $\mathcal S(\mathbb R^d)$ the space of Schwartz functions and $\mathcal
S^{\prime}(\mathbb R^d)$ the space of tempered distributions. For any integer $k\ge 0$ and
open set $U \subset \mathbb R^d$, we shall denote by $C^k_{\operatorname{loc}} (U)$
the space of $k$-times continuously differentiable functions in $U$. 
  For any function $f:\; \mathbb R^d\to
\mathbb R$, we use $\|f\|_{L^p(\mathbb R^d)}$, $\|f\|_{L^p}$ or sometimes $\|f\|_p$ to denote
the  usual Lebesgue $L^p$ norm  for $0< p \le
\infty$. For a sequence of real numbers $(a_j)_{j={-\infty}}^{\infty}$, we denote
\begin{align*}
(a_j)_{l_j^p}=\|(a_j)_{j\in \mathbb Z} \|_{l^p} =\begin{cases}
(\sum_{j\in \mathbb Z} |a_j|^p )^{\frac 1p}, \qquad \text{if $0<p<\infty$}, \\
\sup_{j} |a_j|, \qquad \text{if $j=\infty$}.
\end{cases}
\end{align*}
We shall often use mixed-norm notation. For example, for a sequence of functions $f_j: \;\mathbb R^d \to \mathbb R$,
we will denote (below $0<q<\infty$)
\begin{align*}
\| (f_j)_{l_j^q} \|_{p} =\| (\sum_j |f_j(x)|^q)^{\frac 1q} \|_{L^p_x(\mathbb R^d)},
\end{align*}
with obvious modification for $q=\infty$.

For any two operators $A$, $B$, we shall denote by
\begin{align*}
[A,B] = AB- BA
\end{align*}
the usual commutator.

  For any two quantities $X$ and $Y$, we denote $X \lesssim Y$ if
$X \le C Y$ for some constant $C>0$. Similarly $X \gtrsim Y$ if $X
\ge CY$ for some $C>0$. We denote $X \sim Y$ if $X\lesssim Y$ and $Y
\lesssim X$. The dependence of the constant $C$ on
other parameters or constants are usually clear from the context and
we will often suppress  this dependence. We shall denote
$X \lesssim_{Z_1, Z_2,\cdots,Z_k} Y$
if $X \le CY$ and the constant $C$ depends on the quantities $Z_1,\cdots, Z_k$.

For any two quantities $X$ and $Y$, we shall denote $X\ll Y$ if
$X \le c Y$ for some sufficiently small constant $c$. The smallness of the constant $c$ is
usually clear from the context. The notation $X\gg Y$ is similarly defined. Note that
our use of $\ll$ and $\gg$ here is \emph{different} from the usual Vinogradov notation
in number theory or asymptotic analysis.

 We will need to use the Littlewood--Paley (LP) frequency projection
operators. To fix the notation, let $\phi_0$ be a radial function in
$C_c^\infty(\mathbb{R}^n )$ and satisfy
\begin{equation}\nonumber
0 \leq \phi_0 \leq 1,\quad \phi_0(\xi) = 1\ {\text{ for}}\ |\xi| \leq
1,\quad \phi_0(\xi) = 0\ {\text{ for}}\ |\xi| \geq 7/6.
\end{equation}
Let $\phi(\xi):= \phi_0(\xi) - \phi_0(2\xi)$ which is supported in $\frac 12 \le |\xi| \le \frac 76$.
For any $f \in \mathcal S(\mathbb R^n)$, $j \in \mathbb Z$, define
\begin{align*}
 &\widehat{P_{\le j} f} (\xi) = \phi_0(2^{-j} \xi) \hat f(\xi), \\
 &\widehat{P_j f} (\xi) = \phi(2^{-j} \xi) \hat f(\xi), \qquad \xi \in \mathbb R^n.
\end{align*}
We will denote $P_{>j} = I-P_{\le j}$ ($I$ is the identity operator).  Sometimes for simplicity of
notation (and when there is no obvious confusion) we will write $f_j = P_j f$, $f_{\le j} = P_{\le j} f$ and
$f_{a\le\cdot\le b} = \sum_{j=a}^b f_j$.
By using the support property of $\phi$, we have $P_j P_{j^{\prime}} =0$ whenever $|j-j^{\prime}|>1$.
This property will be useful in product decompositions. For example the Bony paraproduct for a pair of functions
$f,g$ take the form
\begin{align*}
f g = \sum_{i \in \mathbb Z} f_i \tilde g_i + \sum_{i \in \mathbb Z} f_i g_{\le i-2} + \sum_{i \in \mathbb Z}
g_i f_{\le i-2},
\end{align*}
where $\tilde g_i = g_{i-1} +g_i + g_{i+1}$.

The fattened operators $\tilde P_j$ are defined by
\begin{align*}
 \tilde{P}_j = \sum_{l=-{n_1}}^{n_2} P_{j+l},
\end{align*}
where $n_1\ge 0$, $n_2\ge 0$ are some finite integers whose values play no role in the argument.

Note that the Littlewood-Paley projection operators $P_j$ defined above depend on the function $\phi$.
Sometimes it is desirable to use a different function $\tilde \phi \in \mathcal S(\mathbb R^d)$. In that
case to stress the dependence on $\tilde \phi$ we shall
denote
\begin{align*}
\widehat{P^{\tilde \phi}_{j} f} (\xi) = \tilde \phi(\xi /2^j) \hat f(\xi).
\end{align*}

 We recall the Bernstein estimates/inequalities: for $1\le p\le q\le \infty$,
\begin{align*}
&\| D^s P_j f \|_p \sim 2^{js}  \| f \|_p, \qquad s \in \mathbb R; \\
& \|  P_{\le j} f \|_q  +\| P_j f \|_q \lesssim 2^{j d( \frac 1p - \frac 1 q)} \| f \|_p.
\end{align*}
In the above $D^s= (-\Delta)^{s/2}$.

 For $s\in \mathbb R$, $1\le p\le \infty$, the homogeneous Besov $\dot B^s_{p,\infty}$ (semi-)norm is given by
\begin{align*}
 \| f \|_{\dot B^s_{p,\infty} } = \sup_{j\in \mathbb Z}  \bigl( 2^{js} \| P_j f \|_{p} \bigr).
\end{align*}

For any $f \in L^1_{\operatorname{loc}}(\mathbb R^d)$, the BMO norm is given by
\begin{align*}
\| u\|_{\operatorname{BMO}} = \sup_{Q} \frac 1 {|Q|} \int_Q | u(y)-u_Q| dy,
\end{align*}
where $u_Q$ is the average of $u$ on $Q$, and the supreme is taken over all cubes $Q$ in $\mathbb R^d$.

It is well-known that the Besov $\dot B^0_{\infty,\infty}$ norm is weaker than the BMO norm. The following
proposition records this fact. We omit the (standard) proof. 
For any $\lambda>0$, $x_0\in \mathbb R^d$, denote
\begin{align*}
f^{\lambda,x_0}(y) = f(x_0 + \lambda y), \quad  y \in \mathbb R^d.
\end{align*}

\begin{prop} 
For any $f \in L^1_{\operatorname{loc}}(\mathbb R^d)$, we have
\begin{align*}
\|f \|_{\dot B^0_{\infty,\infty}}
\lesssim_d \sup_{\lambda>0, x_0\in \mathbb R^d}
\int_{\mathbb R^d} \frac{ |f^{\lambda,x_0}(y)- (f^{\lambda,x_0})_{Q_1} |}{
(1+|y|)^{d+1} } dy
\lesssim_d \| f \|_{\operatorname{BMO}},
\end{align*}
where $Q_1 = [-1,1]^d$.
\end{prop}

We recall the definition of Hardy space $\mathcal H^1 =\{ f \in L^1(\mathbb R^d): \; \mathcal R_j f \in L^1, \quad
\forall\, 1\le j\le d\}$ with the norm
\begin{align*}
\|f \|_{\mathcal H^1} = \|f\|_1 + \| \mathcal R f \|_1,
\end{align*}
where $\mathcal R f = |\nabla|^{-1} \nabla f =:(\mathcal R_1,\cdots, \mathcal R_d f) $.
Let $\varphi \in \mathcal S(\mathbb R^d)$ with $\int \varphi =1$.
Define $\varphi_t(x)= t^{-d} \varphi(x/t)$.
An equivalent norm on
$\mathcal H^1$ is given by (see \cite{FS_Hp})
\begin{align*}
\| f \|_{\mathcal H^1}=\| \sup_{t>0}|\varphi_t \star f| \,\|_1.
\end{align*}
Similarly for $0<p\le 1$
\begin{align} \label{Hp_def_e2}
\| f\|_{\mathcal H^p} = \| \sup_{t>0}|\varphi_t \star f| \|_p.
\end{align}
Alternatively one can use the following characterization
(cf. \cite{Grafakos_Modern})
\begin{align*}
\| f\|_{\mathcal H^p} \sim \| (P_j f)_{l_j^2} \|_p.
\end{align*}
We will often use this latter characterisation without explicit mentioning.

We will sometimes use  (cf. \eqref{pf_lemq1_ea3} in the proof of Lemma \ref{lemq1}) the following useful fact: if $f\in L^1_{\operatorname{loc}}(\mathbb R^d)$ 
satisfies $\|f\|_{\bmo}<\infty$,  $g \in \mathcal H^1(\mathbb R^d)$, and 
\begin{align*}
\int_{\mathbb R^d} |f(x) g(x)| dx <\infty,
\end{align*}
then
\begin{align}
| \int_{\mathbb R^d} f(x) g(x) dx| \lesssim \| f\|_{\bmo} \|g\|_{\mathcal H^1}. \label{eq_H1_BMO}
\end{align}
Of course without absolute convergence the $\mathcal H^1$-$\bmo$ pairing is defined through
a careful limiting process.

For any $f\in L^1_{\operatorname{loc}}(\mathbb R^d)$, we shall denote by $\mathcal M f:\, \mathbb R^d \to [0,\infty]$
the usual Hardy-Littlewood maximal function defined as:
\begin{align*}
(\mathcal Mf)(x) = \sup_{r>0} \frac 1 {|B_r|} \int_{B_r} |f(x-y)| dy,
\end{align*}
where $B_r=B(0,r)$ is the Euclidean open ball of radius $r$ centered at the origin.

We will often use the following Fefferman-Stein inequality without explicit mentioning.

\begin{lem}[\cite{FS}] \label{lem_FS}
 Let $f=(f_j)_{j=1}^{\infty}$ be a sequence of locally integrable functions in $\mathbb R^d$. Let $1<p<\infty$ and $1<r\le \infty$. Then
 \begin{align*}
  \| (\mathcal Mf_j )_{l_j^r} \|_p \lesssim_{r,p,d} \| (f_j)_{l_j^r} \|_p,
 \end{align*}
where $\mathcal M f$ is the usual Hardy-Littlewood maximal function.
\end{lem}

\begin{proof}
See \cite{FS}. Note that the inequality therein was stated for $1<r<\infty$. But for $r=\infty$ the 
inequality also holds trivially.
\end{proof}

\begin{lem} \label{lemp2}
Suppose $u \in \mathcal S^{\prime}(\mathbb R^d)$ with $\operatorname{supp}(\hat u) \subset \{ \xi:\; |\xi| <t \}$ for some
 $t>0$. Then for any $0< r<\infty$,
\begin{align*}
\sup_{z\in \mathbb R^d} \frac{|u(x-z)|} { (1+t|z|)^{\frac dr}}
\lesssim_{d,r} \Bigl( \mathcal M (|u|^r) (x) \Bigr)^{\frac 1r}, \qquad \forall\, x \in \mathbb R^d.
\end{align*}
In the above $\mathcal M f$ denotes the usual Hardy-Littlewood maximal function.
\end{lem}
\begin{rem}
For textbook proofs under slightly stronger conditions, see \cite{Triebel_vol1}
or \cite{Grafakos_Modern}. For example in \cite{Grafakos_Modern},  
the proof therein assumes the growth condition
\begin{align} \label{pf_lemp2_rem}
\sup_{x \in \mathbb R^d} \frac{ |u(x)|} {(1+|x|)^{\frac dr}} <\infty.
\end{align}
Here we show that this condition can be removed. The removal of such conditions is particularly
interesting for $0<r<1$. 
\end{rem}

\begin{proof}
First note that $u \in C^{\infty}(\mathbb R^d)$ since
$\hat u$ is compactly supported. By Paley-Wiener (for distributions), the function $u$ grows at most
polynomially at the spatially infinity. More precisely, there exists an
integer $N_0\ge 0$ and a constant $A_0>0$, such that 
\begin{align} \label{pf_lemp2_a0}
|u(y)| \le  A_0 |y|^{N_0}, \quad \forall\, |y| \ge 2.
\end{align}
This estimate will be used below. Note that both constants ($A_0$ and $N_0$) may depend on $u$. 

By scaling one can assume $t=1$. Also by
translation it suffices to prove the case $x=0$. Since $u=P_{<2} u$, we have
\begin{align*}
u(z) = \int_{\mathbb R^d} \psi(z-y) u(y) dy,
\end{align*}
where $\psi \in \mathcal S(\mathbb R^d)$ corresponds to $P_{<2}$. The convergence
of the integral is not an issue thanks to the estimate \eqref{pf_lemp2_a0}.

Consider first the case $r\ge 1$. Clearly
\begin{align*}
|u(z)| & \lesssim_{d,r}  \int_{|y-z|\le 1+|z|} |u(y)| dy + \sum_{i=1}^{\infty}
(2^i(1+|z|))^{-10d} \int_{|y-z| \sim 2^{i} (1+|z|)} |u(y) | dy \notag \\
& \lesssim_{d,r}  (\mathcal M(|u|^r)(0) )^{\frac 1 r} \cdot (1+|z|)^{\frac dr}.
\end{align*}
Thus this case is OK.

Next consider the case $0<r<1$. We first  assume the growth condition
\eqref{pf_lemp2_rem}
and complete the estimate. We have
\begin{align*}
|u(z)| & \lesssim_{d,r} \int |\psi(z-y)| \cdot |u(y)|^r \cdot |u(y)|^{1-r} dy \notag \\
& \lesssim_{d,r} \int |\psi(z-y)| |u(y)|^r (1+|y|)^{\frac{d(1-r)}r} dy \Bigl(\sup_{\tilde y  \in \mathbb R^d}
\frac{|u(\tilde y)|} { (1+|\tilde y|)^{\frac dr} } \Bigr)^{1-r} \notag \\
& \lesssim_{d,r} \; (1+|z|)^{\frac dr} \cdot \mathcal M (|u|^r)(0) \cdot
\Bigl(\sup_{\tilde y  \in \mathbb R^d}
\frac{|u(\tilde y)|} { (1+|\tilde y|)^{\frac dr} } \Bigr)^{1-r}.
\end{align*}
The desired inequality then follows.

Finally we show how to prove \eqref{pf_lemp2_rem} for the case $0<r<1$.  For any
$|z|=R\ge 2$, by using \eqref{pf_lemp2_a0}, it is easy to check that
\begin{align*}
|u(z)| &\lesssim_{d,r,u} \int_{|y-z|>\frac 23 R} \langle z-y\rangle^{-10d-N_0} |u(y)| dy 
+ \int_{|y-z|\le \frac 23 R} |u(y)| dy \notag \\
& \lesssim_{d,r,u} 1 + \mathcal M(|u|^r)(0) \cdot R^d \max_{ |\tilde y| \le 2R} |u(\tilde y)|^{1-r}.
\end{align*}
Define for $R\ge 2$, 
\begin{align*}
U_R= 1+ \max_{|y|\le R}  |u(y)|.
\end{align*}
Note that we may assume $\mathcal M(|u|^r)(0)<\infty$. Otherwise there is nothing to prove.
It follows that for some positive constant $B=B(u,d,r)\ge 2$,
\begin{align*}
U_R \le B \cdot R^d\cdot U_{2R}^{1-r}, \quad\forall\, R\ge 2.
\end{align*}
We inductively assume 
\begin{align*}
U_R \le A_k R^{N_k}, \quad \forall\, R\ge 2.
\end{align*}
The base case $k=0$ certainly holds in view of \eqref{pf_lemp2_a0}.
Then 
\begin{align*}
U_R \le B \cdot A_k^{1-r} 2^{N_k (1-r)} \cdot R^{d+ (1-r) N_k}.
\end{align*}
Set $N_{k+1}=d+(1-r)N_k$.  Clearly $N_k \to \frac d r$ as $k\to \infty$ and $N_k \le M$ (for some $M>0$) 
for all
$k$. Set 
\begin{align*}
A_{k+1} = B\cdot 2^{M(1-r)} A_k^{1-r}.
\end{align*}
It is easy to check that $A_k $ converges to some constant $A=B^{\frac 1r} 2^{\frac{M(1-r)}r} $ as $k\to \infty$. Thus 
\eqref{pf_lemp2_rem} is proved. 
\end{proof}

\begin{lem} \label{lemp2a}
Suppose $(f_j)_{j\in \mathbb Z}$ is a sequence of functions satisfying
\begin{align*}
 \operatorname{supp}(\widehat{f_j} ) \subset \{ \xi:\; |\xi| \le B_1 2^j \}, \qquad \forall\, j,
 \end{align*}
 where $B_1>0$ is a constant. Then for any $0<p,q<\infty$, we have
 \begin{align*}
 \| (P_j f_j)_{l_j^q} \|_p \lesssim \| (f_j)_{l_j^q} \|_p.
 \end{align*}
In the above $P_j$ can be replaced by $\tilde P_j$ or $\tilde P_{\le j}$.

On the other hand, if $1<p,q<\infty$, and $(g_j)_{j\in \mathbb Z}$ is a sequence of functions, then
\begin{align}
\| (P_j g_j)_{l_j^q} \|_p \lesssim \| (g_j)_{l_j^q} \|_p. \label{lemp2a_etmp_1}
\end{align}

\end{lem}
\begin{proof}[Proof of Lemma \ref{lemp2a}]
The inequality \eqref{lemp2a_etmp_1} follows easily from the simple fact that $|P_j g_j| \lesssim \mathcal Mg_j$.
Therefore we only need to prove the first inequality.

By Lemma \ref{lemp2},
\begin{align*}
| (P_j f_j)(x)| & \lesssim \int 2^{jd} |\psi(2^j y)| \cdot (1+2^j|y|)^{\frac dr} \cdot
\frac{|f_j(x-y)|} { (1+2^j|y|)^{\frac dr}} dy \notag \\
& \lesssim ( \mathcal M (|f_j|^r) (x) )^{\frac 1r}.
\end{align*}
Now choose $0<r<\min\{p,q\}$ and use Lemma \ref{lem_FS}. We get
\begin{align*}
\| (P_j f_j)_{l_j^q} \|_p &\lesssim \| (\mathcal M(|f_j|^r) )^{\frac 1r}_{l_j^{\frac qr}} \|_p
= \| (\mathcal M (|f_j|^r) )_{l_j^{\frac qr}} \|^{\frac 1r}_{\frac pr} \notag \\
& \lesssim \| (|f_j|^r)_{l_j^{\frac qr}} \|_{\frac pr}^{\frac 1r} = \| (f_j)_{l_j^q} \|_p.
\end{align*}

\end{proof}

The following lemma is more or less trivial. But it is certainly relieving for some
intermediate computations.

\begin{lem} \label{lemp2a_1}
Let $L\ge 2$ be an integer. Suppose $f_j\in \mathcal S(\mathbb R^d)$, $j=-L,-L+1,\cdots, L-1,L$ ,
is a (finite) sequence of functions satisfying
\begin{align*}
\operatorname{supp}(\widehat{f_j}) \subset \{ \xi:\,  |\xi| < B_1 2^j \}, \qquad \forall\, j,
\end{align*}
for some constant  $B_1>0$. Let $\psi \in C_c^{\infty}(\mathbb R^d)$ be such that
$\psi \equiv 0$ in a neighborhood of the origin. 
Then for all $1<p<\infty$,
\begin{align*}
\| \sum_{j=-L}^L P_j^{\psi} f_j \|_p \lesssim_{p,\psi, B_1,d}  \| (f_j)_{l_j^2} \|_p;
\end{align*}
and for $0<p\le 1$, 
\begin{align*}
\| \sum_{j=-L}^L P_j^{\psi} f_j \|_{\mathcal H^p} \lesssim_{p,\psi, B_1,d}  \| (f_j)_{l_j^2} \|_p.
\end{align*}

\end{lem}
\begin{proof}[Proof of Lemma \ref{lemp2a_1}]
Assume $\operatorname{supp}(\psi) \subset \{ \xi:\, 2^{-n_0} <|\xi| <2^{n_0} \}$ for some
integer $n_0>0$.  Then clearly for all $0<p<\infty$,
\begin{align*}
\| (P_k( \sum_{j=-L}^L P_j^{\psi} f_j) )_{l_k^2} \|_p
&\lesssim \sum_{|a|\le n_0+10} \| (P_k P_{k+a}^{\psi} f_{k+a} )_{l_k^2} \|_p \notag \\
&\lesssim \| (f_j)_{l_j^2} \|_p,
\end{align*}
where the last inequality follows from  Lemma \ref{lemp2a}.
\end{proof}
\begin{lem} \label{lemc1}
Let $(a_j)_{j\in \mathbb Z}$ be a sequence of real numbers. Then for any $0<\theta<1$,
$0<p\le \infty$,
$s_1\ne s_2$, $s=\theta s_1+(1-\theta)s_2$,  we have
\begin{align}
&(2^{j s} a_j)_{l_j^p}  \lesssim_{\theta,p, s_1,s_2} \Bigl| (2^{js_1}a_j)_{l_j^{\infty}} \Bigr|^{\theta}
\cdot \Bigl| (2^{js_2} a_j)_{l_j^{\infty}} \Bigr|^{1-\theta}; \notag \\
& (2^{j s} a_j)_{l_j^p} \lesssim_{\theta,p, s_1,s_2} \Bigl| (2^{js_1}a_j)_{l_j^{p}} \Bigr|^{\theta}
\cdot \Bigl| (2^{js_2} a_j)_{l_j^{\infty}} \Bigr|^{1-\theta}. \notag 
\end{align}
\end{lem}

\begin{rem} \notag
The condition $s_1\ne s_2$ is crucial. For $s=s_1=s_2$ the above inequalities are obviously false (unless $p=\infty$). 
\end{rem}

\begin{proof}[Proof of Lemma \ref{lemc1}]
The case $p=\infty$ is trivial. For $0<p<\infty$ since $(2^{j s} a_j)_{l_j^p}^p = (2^{jp s} |a_j|^p)_{l_j^1}$,
 it suffices to prove the case $p=1$, and we may assume $a_j \ge 0$. It is clear that the second inequality 
 follows from the first inequality by using the the fact that $l_j^1 \hookrightarrow l_j^{\infty}$.
Consider first the case $s_1<s_2$. Let $J_0$ be chosen later. Then
\begin{align*}
\sum_j 2^{j (\theta s_1+(1-\theta)s_2 )} a_j & = \sum_{j\le J_0} 
2^{js_1}a_j \cdot 2^{j(1-\theta)(s_2-s_1)}  + 
\sum_{j>J_0} 2^{js_2 } a_j \cdot 2^{-j(1-\theta)(s_2-s_1)}
\notag \\
& \lesssim (2^{js_1} a_j)_{l_j^{\infty}} \cdot 2^{J_0 (1-\theta)(s_2-s_1)} + (2^{js_2} a_j)_{l_j^\infty} 
\cdot 2^{-J_0(1-\theta)(s_2-s_1)}.
\end{align*}
Optimizing in $J_0$ 
 then yields the result.  
The argument for $s_1>s_2$ is similar.
\end{proof}
\begin{rem} \label{lemc2_rem-1}
Other proofs are available. For example (for the second inequality) if $s>0$, $s_1=s$, $s_2=0$, 
$\theta=\frac 12$, $p=1$, then
\begin{align*}
(2^{j \frac 12 s} a_j)_{l_j^1}^2 & \lesssim \sum_{j_1\le j_2} 2^{j_1 \frac 12 s} 2^{j_2 \frac 12 s}
|a_{j_1}| |a_{j_2}| \notag \\
& \lesssim (a_j)_{l_j^{\infty}} \sum_{j_2} 2^{j_2 s} |a_{j_2}|
\lesssim (a_j)_{l_j^{\infty}} (2^{js} a_j)_{l_j^1}.
\end{align*}
One should recognise this as the usual ``squaring trick" with low to high re-ordering. See
Remark \ref{rem2.12} below. 
\end{rem}
\begin{lem} \label{lemc2}
Let $1<r\le \infty$, $0<\theta<1$ and recall $D=(-\Delta)^{\frac 12}$.
Then for any $f \in \mathcal S(\mathbb R^d)$, we have
\begin{align*}
&\| D^{\theta s} f \|_r \lesssim \|D^s f\|_p^{\theta} \cdot \| f \|_q^{1-\theta},
\quad \text{if $s\ge 0$, $1<p,q<\infty$, and $\frac 1 r= \frac{\theta} p+ \frac{1-\theta} q$};\\
&\|D^{\theta s} f\|_r \lesssim \| D^s f\|_{\dot B^0_{\infty,\infty}}^{\theta}
\cdot \| f \|_{r(1-\theta)}^{1-\theta}, \quad \text{if $s>0$ }
\notag \\
&\qquad\qquad\qquad \text{(this corresponds to $p=\infty$,
 $q=r(1-\theta)$)};\\
& \|D^{\theta s} f \|_r \lesssim \| D^s f \|_{r\theta}^{\theta}
\cdot \| f \|_{\dot B^0_{\infty,\infty}}^{1-\theta}, \quad \text{if $s>0$ (correspondingly $p=r\theta$, $q=\infty$)}.
\end{align*}
\end{lem}
\begin{rem*}
The second and third inequalities are \emph{false} for $s=0$.
\end{rem*}

\begin{rem} \label{lemc2_rem1}
By the same proof of lemma \ref{lemc2} below,  one can show that more generally for $0<\theta<1$, $s=s_1\theta +s_2(1-\theta)$ with $s_1\ne s_2$,
$1/\theta<r\le \infty$, 
\begin{align*}
\| D^s f \|_r \lesssim \| D^{s_1} f \|_{r\theta}^{\theta} \| D^{s_2} f \|_{\dot B^0_{\infty,\infty}}^{1-\theta},
\end{align*}
provided (of course) that these quantities are well-defined (especially when dealing with
$D^{s}$ operators with $s<0$).  
\end{rem}
\begin{proof}[Proof of Lemma \ref{lemc2}]
The inequalities are trivial for $r=\infty$. Also the first inequality is trivial if $s=0$.
Thus we may assume $s>0$ and $1<r<\infty$. By Lemma \ref{lemc1}, we have
\begin{align*}
\| D^{\theta s} f \|_r & \sim \| (2^{j\theta s} f_j )_{l_j^2} \|_r \notag \\
& \lesssim \| (2^{js} f_j)_{l_j^2}^{\theta} (f_j)_{l_j^{\infty}}^{1-\theta} \|_r \notag \\
& \lesssim \| (2^{js} f_j)_{l_j^2} \|_p^{\theta} \cdot \| (f_j)_{l_j^{\infty}} \|_q^{1-\theta} \notag \\
& \lesssim \| D^s f\|_p^{\theta} \cdot \| f \|_q^{1-\theta}, \qquad \text{if $p,q<\infty$}.
\end{align*}
The argument for the second and third inequalities are similar. We omit details.
\end{proof}
\begin{rem} \label{rem2.12}
By using Remark \ref{lemc2_rem1} and taking $r=2d/(d-2)$,
$\theta=(d-2)/d$, $s=0$, $s_1=1$, $s_2=-(d-2)/2$, we obtain for $d\ge 3$
and $f \in \mathcal S(\mathbb R^d)$:
\begin{align} \label{bubble_1st_ineq}
\|f \|_{\frac {2d}{d-2}}
\lesssim \|\nabla f \|_2^{\frac{d-2}d} \| f \|^{\frac 2d}_{\dot B^{-\frac{d-2}2}_{\infty,\infty}}.
\end{align}
By the obvious embedding $\dot B^0_{\frac{2d}{d-2},\infty}
\hookrightarrow \dot B^{-\frac{d-2}2}_{\infty,\infty}$, we obtain
\begin{align*}
\| f \|_{\frac{2d}{d-2}} \lesssim \| \nabla f \|_2^{\frac{d-2} d}  
\| f \|^{\frac 2d}_{\dot B^0_{\frac{2d}{d-2},\infty}}.
\end{align*}
This refined inequality plays an important role\footnote{In fact the first inequality \eqref{bubble_1st_ineq} 
already provides a quick
route to the extraction of ``bubbles" in linear profile decomposition.}
in the theory of linear profile decomposition for nonlinear Schr\"odinger
equations (cf. Proposition 4.8 on page 36 of \cite{RV_Clay}).  Note that the proof
therein uses a squaring trick and a low to high ordering of the dyadic sum, in a way
that is similar to the idea
in Remark \ref{lemc2_rem-1}. Our treatment here seems much simpler.
\end{rem}

We shall often use the following simple lemma, sometimes without explicit mentioning.
\begin{lem} \label{lemc2_a1}
The following hold.
\begin{itemize}
\item If $s>0$, $1<p<\infty$, then
\begin{align*}
\| (2^{-js} D^s P_{\le j} f )_{l_j^2} \|_p + \| (2^{-js} D^s P_{\le j} f)_{l_j^{\infty}}\|_p \lesssim_{s,p,d}
\| f\|_p.
\end{align*}
\item If $s> 0$, $p=\infty$, then
\begin{align*}
\| (2^{-js} D^s P_{\le j} f)_{l_j^{\infty}}\|_{\infty} \lesssim \| f \|_{\dot B^0_{\infty,\infty}}.
\end{align*}
\end{itemize}
\end{lem}
\begin{proof}[Proof of Lemma \ref{lemc2_a1}]
We have
\begin{align*}
|2^{-js} (D^s P_{\le j} f)(x) | &= 2^{-js} |\sum_{k\le j} (D^s \tilde P_k P_k f)(x) | \notag \\
&\lesssim 2^{-js} \sum_{k\le j} 2^{ks} \mathcal M (P_k f)(x).
\end{align*}
Since $s>0$, we have
\begin{align*}
\bigl(2^{-js} \sum_{k\le j} 2^{ks} \mathcal M (P_k f) \bigr)_{l_j^2} \lsm ( \mathcal M (P_k f))_{l_k^2}.
\end{align*}
Thus
\begin{align*}
\| (2^{-js} D^s P_{\le j} f )_{l_j^2} \|_p & \lsm \| (\mathcal M (P_k f) )_{l_k^2} \|_p
\lsm \| f\|_p.
\end{align*}
Since the sequence $l^2$ norm controls $l^{\infty}$ norm, the inequality for $l_j^{\infty}$ follows.
Finally
\begin{align*}
\| (2^{-js} D^s P_{\le j} f )_{l_j^{\infty}} \|_{\infty}
&\lesssim ( 2^{-js} \| D^s P_{\le j} f \|_{\infty} )_{l_j^{\infty}} \notag \\
&\lesssim ( 2^{-js} \sum_{k\le j} 2^{ks} \|f\|_{\boo} )_{l_j^{\infty}} \lsm \|f\|_{\boo}.
\end{align*}

\end{proof}

The following lemma collects some useful properties of $J^s$ and $D^s$ operators. We will often
use it without explicit mentioning in later computations.

\begin{lem}[Properties of $J^s$, $D^s$ operators] \label{lem_JsDs}
Let $s>0$ and recall $D^s=(-\Delta)^{s/2}$, $J^s=(1-\Delta)^{s/2}$. Let $a$ be any
given real number. Then the following
inequalities hold for any $f \in \mathcal S(\mathbb R^d)$:
\begin{align}
&\|P_{>a} D^s f \|_p \lesssim_{a,p,d,s} \| J^{s-1} \partial f \|_p, \qquad\forall\, 1<p<\infty,\label{lem_JsDs_e1} \\
& \| P_j f \|_p \lesssim_{d,s}
\begin{cases}
2^{-j} \| J^{s-1} \partial f \|_p, \quad \text{if $j\le 0$, $1\le p\le \infty$} \\
 2^{-js} \| J^{s-1} \partial f\|_p, \quad \text{if $j>0$, $1\le p\le \infty$};
 \end{cases}
\label{lem_JsDs_e2}\\
& \| J^s f - f\|_p \lesssim_{p,d,s} \| J^{s-1} \partial f \|_p, \qquad\forall\, 1<p<\infty,
\label{lem_JsDs_e3}\\
& \| J^{-1} \partial f \|_{\bmo} \lesssim_{d} \| f\|_{\bmo}, \label{lem_JsDs_e4}\\
& \| J^s P_{>a} f  \|_{\bmo} \lesssim_{a,d,s} \| J^{s-1} \partial f  \|_{\bmo}. \label{lem_JsDs_e5}
\end{align}

\end{lem}
\begin{proof}[Proof of Lemma \ref{lem_JsDs}]
For \eqref{lem_JsDs_e1}, we write
\begin{align*}
P_{>a} D^s f  & = P_{>a} D^{s-2} (-\partial \cdot \partial)f \notag \\
& = -P_{>a} D^{s-2} J^{-(s-1)} \partial \cdot (J^{s-1} \partial f).
\end{align*}
Easy to check that $-P_{>a} D^{s-2} J^{-(s-1)} \partial $ maps $L^p$ to $L^p$ for $1<p<\infty$. Thus \eqref{lem_JsDs_e1} holds.

For \eqref{lem_JsDs_e2}, consider first $j\le 0$. Clearly
\begin{align*}
\| P_j f\|_p  & = \| P_j D^{-2} \partial J^{1-s} \cdot J^{s-1} \partial f \|_p \notag \\
& \lesssim_{d} 2^{-j} \| P_{<10}  J^{1-s} \cdot J^{s-1} \partial f \|_p \notag \\
& \lesssim_{s,d} 2^{-j} \| J^{s-1} \partial f \|_p.
\end{align*}
For $j>0$, one can easily verify that the kernel $K_j=  D^{-2} J^{1-s} \partial P_j \delta_0$ satisfies the point-wise bound
\begin{align*}
|K_j(x)| \lesssim_{d,m,s} 2^{j(-s+d)} (1+2^j |x|)^{-m}, \qquad \forall \, m \ge 1.
\end{align*}
Thus the inequality for $j>0$ also holds.

For \eqref{lem_JsDs_e3}, we first bound the low frequency piece. By using Lemma \ref{lem921_1}, we have
\begin{align*}
\| (J^s-1)P_{\le 1} f \|_p & \lesssim_{s,p,d} \sum_{j\le 1} 2^{2j} \| P_j f \|_p \notag \\
& \lesssim_{s,p,d} \sum_{j\le 1} 2^j \| J^{s-1} \partial f \|_p \lesssim \| J^{s-1} \partial f \|_p,
\end{align*}
where in the second inequality we have used \eqref{lem_JsDs_e2}.
For the high frequency piece, we first note that
\begin{align*}
\| P_{>1} f\|_p & \lesssim_{s,p,d} \sum_{j>1} 2^{-js} \| J^{s-1} \partial f \|_p \lesssim \| J^{s-1} \partial f\|_p.
\end{align*}
On the other hand, by writing $P_{>1} J^s f = -P_{>1} D^{-2}J \partial \cdot J^{s-1} \partial f$, it
is clear that
\begin{align*}
\| P_{>1} J^s f\|_p \lesssim_{s,p,d} \| J^{s-1} \partial f\|_p, \qquad \forall\, 1<p<\infty.
\end{align*}

The inequality \eqref{lem_JsDs_e4} follows easily from the fact that $J^{-1} \partial$ is a standard singular
integral operator. 

The last inequality \eqref{lem_JsDs_e5} is similarly proved by writing
\begin{align*}
J^s P_{>a} f= -P_{>a} D^{-2} J \partial \cdot J^{s-1} \partial f.
\end{align*}
\end{proof}

\section{Paraproduct estimates}

\begin{lem} \label{lemq1}
Let $\phi \in C_c^{\infty}(\mathbb R^d)$, $\psi \in C_c^{\infty} (\mathbb R^d)$, and
$\psi\equiv 0$ in a neighborhood of the origin. Define
\begin{align*}
&\widehat{P_j^{\phi} f}(\xi) = \phi(\xi/2^j) \hat f(\xi), \\
&\widehat{P_j^{\psi} f} (\xi) = \psi(\xi/2^j) \hat f(\xi), \quad \xi \in \mathbb R^d.
\end{align*}
Then for any $1<p<\infty$, $f \in L^p(\mathbb R^d)$, 
$g\in L^1_{\operatorname{loc}}(\mathbb R^d)$ with $\|g\|_{\operatorname{BMO}}<\infty$, the 
series $$\sum_{j \in \mathbb Z} P_j^{\phi} f \cdot P_j^{\psi} g$$
 converges in $L^p$, and
satisfies
\begin{align}
\| \sum_{j\in \mathbb Z} P_j^{\phi} f \cdot P_j^{\psi} g \|_p
 \lesssim_{p,d,\phi,\psi} \| f \|_p \| g\|_{\operatorname{BMO}}. \label{lemq1_e1}
 \end{align}
 In particular, for the usual Littlewood-Paley projector $ P_j$, we have
 \begin{align*}
 \| \sum_j  P_{\le j} f \cdot P_j g \|_p \lesssim_{p,d} \| f \|_p \cdot \| g \|_{\operatorname{BMO}}.
 \end{align*}
 If $\|D^{-1} f\|_p <\infty$, $\|Dg \|_{\operatorname{BMO}}<\infty$, then
 \begin{align}
 \| \sum_j P_{\le j } f \cdot P_{j} g \|_p \lesssim_{p,d} \| D^{-1} f \|_{p} \| D g \|_{\bmo}. \label{lemq1_e1_new1}
 \end{align}
 More generally if $s>0$, $f \in L^p$, $g \in L^1_{\operatorname{loc}}(\mathbb R^d)$ with
 $\|g \|_{\operatorname{BMO} } <\infty$, then
 \begin{align} \label{lemq1_e1_new2}
 \| \sum_j D^s P_{\le j } f D^{-s} P_j g \|_p \lesssim_{p,d,s} \| f \|_p \| g \|_{\operatorname{BMO}}.
 \end{align}
 \end{lem}
\begin{rem}
The estimate \eqref{lemq1_e1} actually holds under the weaker
assumption that $\psi(0)=0$. But the argument is slightly more involved.
\end{rem}

\begin{proof}[Proof of Lemma \ref{lemq1}]
For simplicity of notation we shall write $\lesssim_{p,d,\phi,\psi}$ as $\lesssim$.

\underline{Step 1}: We first show for any integer $L_1 \ge 2$, $L_2 \ge 2$,
\begin{align} \label{pf_lemq1_ea0}
\| \underbrace{ \sum_{-L_1 \le j \le L_2}
P_j^{\phi} f \cdot P_j^{\psi} g }_{=:S} \|_p \lesssim \| f\|_p \|g\|_{\operatorname{BMO}}.
\end{align}
Note that the summand is well-defined in $L^p$ since $\| P_j^{\psi} g\|_{\infty} \lesssim \|g\|_{\operatorname{BMO}}$.
We shall write $\sum_{-L_1\le j\le L_2}$ simply as $\sum_j$ and keep in mind that the summation in
$j$ is finite. By a density argument we only need to prove \eqref{pf_lemq1_ea0} for $f \in C_c^{\infty}(\mathbb R^d)$. 

It suffices to prove for any $h \in C_c^{\infty}(\mathbb R^d)$, 
\begin{align*}
|\langle S, h \rangle | \lesssim \| f\|_p \|g \|_{\operatorname{BMO}} \| h\|_{p^{\prime}},
\end{align*}
where $p^{\prime}=p/(p-1)$, and $\langle\cdot, \cdot\rangle$ denotes the usual $L^2$ pairing.

Now observe
\begin{align}
|\langle S, h \rangle|& = |\langle g, \sum_{j} P_j^{\psi} ( h P_{ j} ^{\phi} f) \rangle|  \notag \\
& \lesssim \| g \|_{\operatorname{BMO}} \| \sum_{j} P_j ^{\psi} (h P_{ j}^{\phi}  f ) \|_{\mathcal H^1},
\label{pf_lemq1_ea3}
\end{align}
where we used \eqref{eq_H1_BMO}.

Thus we only need to show
\begin{align*}
\| \sum_j P_j ^{\psi} (h P_j^{\phi}  f) \|_{\mathcal H^1} \lesssim \| h\|_{p^{\prime}} \|f\|_p.
\end{align*}

WLOG we assume $P_j^{\phi}=P_{\le j}$, $P_j^{\psi} = P_j$ the usual Littlewood-Paley projectors.
The argument can be easily modified for the general case. 

By frequency localization,
\begin{align*}
\sum_j P_j (P_{\le j} f h) &= \sum_j P_j ( P_{\le j-3} f P_{j-2<\cdot <j+2} h) + \sum_j P_j (P_{j-2\le \cdot \le j} f
P_{\le j+3} h) \notag \\
& =: \sum_j P_j (f_{\le j-3} \tilde h_j) + \sum_j P_j (\tilde f_j h_{\le j+3}).
\end{align*}

Clearly by Lemma \ref{lemp2a_1},
\begin{align*}
\| \sum_j P_j( f_{\le j-3} \tilde h_j ) \|_{\mathcal H^1}
& \lesssim \| (f_{\le j-3} \tilde h_j )_{l_j^2} \|_1
\notag \\
& \lesssim \| (f_{\le j-3} )_{l_j^{\infty}} (\tilde h_j)_{l_j^2} \|_1 \lesssim \| f\|_p \|h\|_{p^{\prime}}.
\end{align*}
The other piece is estimated similarly. Thus \eqref{pf_lemq1_ea0} holds.

\underline{Step 2}: Convergence of the series in $L^p$.  First observe that for any $M_1>L_1\ge 2$, 
\begin{align} \notag 
\| \sum_{-M_1<j<-L_1} P_j^{\phi} f  P_j^{\psi} g \|_p  
\lesssim \| \tilde P_{<-L_1+10} f \|_p \|g\|_{\operatorname{BMO}} \to 0,
\quad \text{ as $L_1\to \infty$,}
\end{align}
where $\tilde P_{<-L_1+10}$ denotes a fattened frequency projection.
We only need to show  for any $M_2>L_2 \ge 2$,  
\begin{align*}
\| \sum_{ L_2<j<M_2} P_j^{\phi} f P_j^{\psi} g \|_p \to 0, \quad \text{as $L_2 \to \infty$}.
\end{align*}
By a density argument we may assume $f \in C_c^{\infty}(\mathbb R^d)$.  Easy to check
that 
\begin{align*}
\| \phi(0) f - P_j^{\phi} f \|_p \lesssim 2^{-j } \| \nabla f \|_p.
\end{align*}
Thus
\begin{align*}
&\| \sum_{L_2<j<M_2} ( \phi(0) f - P_j^{\phi} f) P_j^{\psi} g \|_p \notag \\
\lesssim & \sum_{L_2<j<M_2} 2^{-j} \| \nabla f \|_p \|g \|_{\operatorname{BMO}} 
\to 0, \quad \text{as $L_2 \to \infty$}.
\end{align*}
It remains for us to show 
\begin{align} \label{pf_lemq1_ea5}
\| f \sum_{L_2<j<M_2} P_j^{\psi} g \|_p \to 0, \quad \text{as $L_2\to \infty$.}
\end{align}
Assume $\operatorname{supp}(f) \subset B(0,R_0)$ for some $R_0\ge 2$. Let $\chi \in C_c^{\infty}$
be such that $\chi(x)=1$ for $|x| \le 2R_0$ and $\chi(x)=0$ for $|x|>3R_0$.  Denote the
average of $g$ on $B(0,1)$ as $g_{B_1}$. 
Then it
is not difficult to check that
\begin{align*}
&\sum_{L_2< j< M_2} \| f P_j^{\psi} \bigl( (1-\chi) (g-g_{B_1}) \bigr) ) \|_p \notag \\
 \lesssim &
\sum_{L_2<j<M_2} \|f\|_p (2^j R_0)^{-1} \| (1+|y|)^{-(1+d)}|g(y)-g_{B_1}| \|_{L_y^1(\mathbb R^d)} \notag \\
 \lesssim &\sum_{L_2<j<M_2} \|f\|_p  2^{-j} \|g\|_{\operatorname{BMO} }\to 0,
\quad \text{as $L_2\to \infty$}.
\end{align*}
On the other hand, note that $\chi \cdot (g-g_{B_1}) \in L^{q}$ for any $1\le q<\infty$.
Assuming $\operatorname{supp}(\psi) \subset \{\xi:\, 2^{-n_0} <|\xi|<2^{n_0} \}$
for some integer $n_0$, then
\begin{align*}
&\| f \sum_{L_2<j<M_2} P_j^{\psi} \bigl( \chi (g-g_{B_1} ) \bigr) \|_p \notag \\
\lesssim\; & \|f\|_{2p} \| P_{>L_2-n_0-10} \bigl( \chi (g-g_{B_1} ) \bigr) \|_{2p} \to 0,
\quad \text{as $L_2\to \infty$}.
\end{align*}
Thus \eqref{pf_lemq1_ea5} is proved and the series  converges in $L^p$.

\underline{Step 3}: Proof of \eqref{lemq1_e1_new1}. In this case, 
one just observes that $f= -\nabla \cdot (-\Delta)^{-1} \nabla f$, and
\begin{align*}
 \sum_{j} P_{\le j}f P_j g  & = -
 \sum_{l=1}^d \sum_j \bigl(2^{-j} \partial_l P_{\le j} \cdot(-\Delta)^{-1}\partial_l f \bigr)
 \bigl( 2^{j} D^{-1} P_j D g\bigr).
\end{align*}
Note that for each $l$ one can write $2^{-j} \partial_l P_{\le j} = P_{j}^{\phi_l} $, and $2^{j} D^{-1} P_j = P_j^{\psi}$,
for some functions $\phi_l \in \mathcal S(\mathbb R^d)$, $\psi \in \mathcal S(\mathbb R^d)$ with
$\psi$ vanishing near the origin. The desired inequality then easily follow
from \eqref{lemq1_e1}.

\underline{Step 4}: Proof of \eqref{lemq1_e1_new2}. This is similar to the argument in Step 1. By duality and
density, it suffices to prove for any $f, h \in C_c^{\infty}(\mathbb R^d)$,
\begin{align*}
\| (D^{-s} P_j ( h D^s P_{\le j} f ) )_{l_j^2} \|_1 \lesssim \| f \|_p \| h \|_{p^{\prime}}.
\end{align*}
One can then split $h$ as $h_{<j-2}$ and $\tilde h_j =h_{[j-2,j+2]}$ and proceed to estimate
\begin{align*}
\| (D^{-s}  P_j ( h D^s P_{\le j} f ) )_{l_j^2} \|_1 
& \lesssim \|  (h_{<j-2} \tilde f_j )_{l_j^2} \|_1 + \| (\tilde h_j  2^{-js} D^s f_{\le j} ) _{l_j^2} \|_1 \notag \\
& \lesssim  \| f \|_p \| h\|_{p^{\prime}}.
\end{align*}
\end{proof}

\begin{rem}
The BMO norm on the RHS of \eqref{lemq1_e1} cannot be replaced by a weaker norm such as
$\|\cdot \|_{\dot B^0_{\infty,\infty}}$. For a counterexample one can take $p=2$, $d=1$,
$g=f$, and we shall \emph{disprove}
\begin{align*}
\| \sum_j (P_j f)^2 \|_2  \lesssim \|f\|_2 \cdot \|f\|_{\dot B^0_{\infty,\infty}}.
\end{align*}
To this end, take $\widehat{\phi_0} \in C_c^{\infty}(\mathbb R)$ such that
$0\le \widehat{\phi_0}(\xi) \le 1$ for all $\xi$, $\widehat{\phi_0}(\xi)=1$ for $|\xi|<1/10$,
and $\widehat{\phi_0}(\xi)=0$ for $|\xi|>1/5$. Then for some $\rho_0>0$, we have
\begin{align*}
|\phi_0(x)| \gtrsim 1, \qquad\text{for all $|x|<\rho_0$}.
\end{align*}

Now take
\begin{align*}
f= \sum_{j\ge 10} a_j \lambda_j^{\frac 12} \phi_0(\lambda_j x) e^{i 2^j x} =: \sum_{j\ge 10} a_j f_j,
\end{align*}
where $a_j= j^{-(\frac 12+\delta)}$, $\lambda_j = j^{\epsilon}$, with $\frac 1{10}>\epsilon \ge 4\delta>0$.
Easy to check that
\begin{align*}
& \sum_j a_j^2 <\infty \Rightarrow \; \|f\|_2<\infty,\\
& \sup_j |a_j| \lambda_j^{\frac 12} <\infty \; \Rightarrow \|f\|_{\dot B^0_{\infty,\infty}}<\infty.
\end{align*}
On the other hand, for $\frac{\rho_0}{\lambda_{l+1}} < |x|<\frac{\rho_0}{\lambda_l}$, we have
\begin{align*}
\sum_j |P_j f(x)|^2  & = \sum_j |a_j|^2 \cdot \lambda_j \cdot |\phi_0(\lambda_j x)|^2 \notag \\
& \gtrsim \sum_{j<l} |a_j|^2 \cdot \lambda_j.
\end{align*}

Thus
\begin{align*}
\int ( \sum_j |(P_j f)(x)|^2)^2 dx & \gtrsim \sum_l ( \sum_{j<l} a_j^2 \lambda_j)^2 \cdot (\frac{\rho_0}{\lambda_l}
-\frac{\rho_0}{\lambda_{l+1}} ) \notag \\
& \gtrsim \sum_{l} ( \sum_{j<l} j^{-1-2\delta} j^{\epsilon} )^2 \cdot \frac 1 {l^{1+\epsilon}} \notag \\
& \gtrsim \sum_l l^{2(\epsilon-2\delta)} \cdot \frac 1 {l^{1+\epsilon}} =\infty.
\end{align*}

Finally we should point it out that by considering real and imaginary parts, one can also make the
above counterexample  real-valued.

\end{rem}

\begin{rem}
For the periodic domain $\mathbb T=\mathbb R/2\pi \mathbb Z$, it is much easier to
construct counterexamples to the estimate 
\begin{align*}
\| \sum_j (P_j f)^2 \|_{L^2(\mathbb T)} \lesssim \| f\|_{L^2(\mathbb T)} 
\| f \|_{\boo(\mathbb T)}.
\end{align*}
Indeed if the above estimate were true, then by H\"older, one gets for any
periodic function $f$ with mean zero,
\begin{align*}
\| f \|_{L^4(\mathbb T)} \lesssim \| f \|_{\boo (\mathbb T)}.
\end{align*}
Now take $\lambda_j =4^j$, $c_j=1/\sqrt j$ and consider $f$ in the form of a 
lacunary series,
\begin{align*}
f = \sum_{j\ge 1} c_j e^{i \lambda_j x}.
\end{align*}
Clearly 
\begin{align*}
\|f \|_{\boo(\mathbb T)} \lesssim \sup_{j} |c_j|<\infty; \\
\|f\|_{L^2(\mathbb T)} \sim (\sum_{j} |c_j|^2 )^{\frac 12}=\infty
\end{align*}
which is an obvious contradiction.

\end{rem}

\begin{rem}
The first part of the statements of Lemma \ref{lemq1} can also be deduced from the following proposition due to Coifman-Meyer
(see \cite{CM78}, Chapter V. Proposition 2):
let $\sigma = \sigma(\xi,\eta) \in C^{\infty}(\mathbb R^d\times \mathbb R^d \setminus (0,0))$ satisfy
\begin{align}
&\bullet\; |\partial_{\xi}^{\alpha} \partial_{\eta}^{\beta} \sigma(\xi,\eta)|
\lesssim_{\alpha,\beta} (|\xi|+|\eta|)^{-(|\alpha|+|\beta|)}, \quad\forall\, \alpha,\beta,  \forall\, (\xi,\eta)\ne (0,0),
\label{CM_cond_e1}\\
&\bullet\; \sigma(\xi,0)=0. \label{CM_cond_e2}
\end{align}
Define
\begin{align*}
\sigma(D) (f,g)(x) = \int e^{ix \cdot(\xi+\eta)} \sigma(\xi,\eta) \hat f(\xi) \hat g (\eta) d\xi d\eta.
\end{align*}
Then for any $1<p<\infty$,
\begin{align}
\| \sigma(D)(f,g)\|_p  \lesssim_{\sigma,p,d} \| f\|_p \|g\|_{\operatorname{BMO}}. \label{lemC1_rem4_e1}
\end{align}
In our setting
\begin{align*}
\sigma(\xi,\eta) = \sum_j \phi(2^{-j} \xi) \psi(2^{-j} \eta)
\end{align*}
and it is easy to check that it satisfies the conditions \eqref{CM_cond_e1}--\eqref{CM_cond_e2}.
See also Theorem \ref{thm_cm_1} and Theorem \ref{thm_cm_1a} for more refined results.
\end{rem}

\begin{rem}
Take $f=g$ and use the Littlewood-Paley projection $P_j$ in Lemma \ref{lemq1}, and we get
\begin{align*}
\| \sum_j f_j^2 \|_p \lesssim \| f \|_p \| f\|_{\operatorname{BMO}}.
\end{align*}
Fix any $0<p_0<p$. Then
\begin{align*}
\|f\|_{2p}^2 \lesssim \| f\|_p \|f\|_{\operatorname{BMO}} \lesssim \|f \|_{p_0}^{\theta}
\|f\|_{2p}^{1-\theta} \|f\|_{\bmo},
\end{align*}
where $\theta \in (0,1)$ satisfies $\frac 1 p = \frac{\theta}{p_0} + \frac{1-\theta}{2p}$.
This in turn yields
\begin{align*}
\|f\|_{2p} \lesssim \|f\|_{p_0}^{\frac{p_0}{2p}} \|f \|_{\bmo}^{1-\frac{p_0}{2p}}.
\end{align*}
By another interpolation if necessary, one can then get for any $q>p_0$,
\begin{align*}
\|f\|_q \lesssim \| f\|_{p_0}^{\frac{p_0}q} \| f \|_{\bmo}^{1-\frac{p_0}p}
\end{align*}
which is the usual BMO refinement of H\"older interpolation inequality (albeit with no explicit
constants).

Similarly by writing
\begin{align*}
fg & = \sum_j f_{\le j+2} g_j + \sum_j f_{>j+2} g_j \notag \\
& = \sum_j f_{\le j+2} g_j + \sum_j f_j g_{<j-2}
\end{align*}
and applying Lemma \ref{lemq1}, we get for any $1<p<\infty$,
\begin{align*}
\|f g\|_p \lesssim \|f \|_p \| g\|_{\bmo} + \|g\|_p \|f\|_{\bmo}.
\end{align*}
These (and more) bilinear BMO type inequalities were derived by Kozono and Taniuchi \cite{KT00}
and have important applications in Navier-Stokes and Euler equations.

\end{rem}

The idea of duality used in  Lemma \ref{lemq1} is quite useful. For example the following well-known commutator estimate can be proved along
similar lines as in the proof of Lemma \ref{lemq1}. 

\begin{prop}
Let the dimension $d\ge 2$ and $\mathcal R_{ij} = \Delta^{-1}\partial_i \partial_j$, $1\le i,j\le d$ be the usual Riesz transform on $\mathbb R^d$.
Then for any $1<p<\infty$, $a\in L^1_{\operatorname{loc}}(\mathbb R^d)$ with 
$\|a\|_{\operatorname{BMO}}<\infty$, we have
\begin{align*}
\| [\mathcal R_{ij}, a] f \|_p \lesssim_{p,d} \| a\|_{\bmo} \|f\|_p, \qquad \forall\, f\in \mathcal S(\mathbb R^d).
\end{align*}
\end{prop}

\begin{proof}
WLOG we prove the inequality for $\mathcal R_{11} = \Delta^{-1} \partial_{11}$. By duality,
it suffices to prove for any $g\in C_c^{\infty}(\mathbb R^d)$,
\begin{align*}
\| f \mathcal R_{11} g - g \mathcal R_{11} f \|_{\mathcal H^1} \lesssim \|f\|_p \|g\|_{p^{\prime}}.
\end{align*}

Write
\begin{align*}
&f \mathcal R_{11} g = \sum_j f _{<j-2} \mathcal R_{11} g_j
+ \sum_j f_j \mathcal R_{11} g_{<j-2} + \sum_j f_j \mathcal R_{11} \tilde g_j,\\
& (\mathcal R_{11} f) g = \sum_j (\mathcal R_{11} f)_{<j-2} g_j
+\sum_j (\mathcal R_{11} f)_j g_{<j-2} + \sum_j (\mathcal R_{11} f)_j \tilde g_j.
\end{align*}

Easy to check that
\begin{align*}
& \| \sum_j f_{<j-2} \mathcal R_{11} g_j \|_{\mathcal H^1}
\lesssim \| (f_{<j-2} \mathcal R_{11} g_j )_{l_j^2} \|_1
\lesssim \| f\|_p \|g\|_{p^{\prime}}, \notag \\
& \| \sum_j( f_j \mathcal R_{11} g_{<j-2} + (\mathcal R_{11} f)_{<j-2} g_j
+ (\mathcal R_{11} f)_j g_{<j-2} ) \|_{\mathcal H^1}
\lesssim \| f\|_p \|g\|_{p^{\prime}}.
\end{align*}

For the diagonal piece, denote $A=\Delta^{-1} f_j$, $B=\Delta^{-1} \tilde g_j$, and observe
\begin{align*}
 & f_j \mathcal R_{11} \tilde g_j - (\mathcal R_{11} f_j) \tilde g_j \notag \\
 = & \; \sum_{k=1}^d \Bigl( \partial_{kk} A \partial_{11} B- \partial_{11} A \partial_{kk} B \Bigr).
 \end{align*}

 In terms of the frequency variables $(\eta,\xi-\eta)$ (i.e. $\hat A(\eta)$, $\hat B(\xi-\eta)$), note that
 \begin{align*}
  & \eta_k^2 (\xi_1-\eta_1)^2 - \eta_1^2 (\xi_k -\eta_k)^2 \notag \\
  = &\; \eta_k^2 \xi_1^2 -2 \eta_k^2 \eta_1 \xi_1 - \eta_1^2 \xi_k^2 + 2 \eta_1^2 \eta_k \xi_k.
  \end{align*}

Therefore we can write
\begin{align*}
\sum_{k=1}^d \Bigl( \partial_{kk} A \partial_{11} B
- \partial_{11} A \partial_{kk} B \Bigr) = O \Bigl( \partial^2( \partial^2 A\cdot B) ) \Bigr)
+ O \Bigl( \partial ( \partial^3 A \cdot B ) \Bigr).
\end{align*}

Clearly then
\begin{align*}
 & \| \sum_j \Bigl( f_j \mathcal R_{11} \tilde g_j -\mathcal R_{11} f_j \tilde g_j \Bigr) \|_{\mathcal H^1}
 \notag \\
 \lesssim & \sum_k   2^{2k} \sum_{j\ge k-4} \|\tilde P_j f \cdot 2^{-2j} \tilde P_j g   \|_1
 + \sum_k 2^k \sum_{j\ge k-4} 2^{-j}  \| \; \tilde P_j f \tilde P_j g  \; \|_1 \notag \\
 \lesssim & \| (\tilde P_j f \cdot \tilde P_j g)_{l_j^1} \|_1 \lesssim \| f\|_p \|g\|_{p^{\prime}}.
 \end{align*}
The desired inequality then follows.
\end{proof}

\begin{prop} \label{prop3.7a}
Denote by $H$ the  usual Hilbert transform on $\mathbb R$. Then for any $1<p<\infty$, and any integers $l$, $m\ge 0$, we have
\begin{align*}
\| \partial_x^l [H,a] \partial_x^m f \|_p \lesssim_{l,m,p} \| \partial_x^{l+m} a \|_{\operatorname{BMO}}
\| f \|_p.
\end{align*}
\end{prop}
\begin{rem}
In \cite{DMP08}, Dawson, McGahagan and Ponce proved 
\begin{align*}
\| \partial_x^l [H,a] \partial_x^m f \|_p \lesssim_{l,m,p} \| \partial_x^{l+m} a \|_{\infty}
\| f \|_p.
\end{align*}
Here Proposition \ref{prop3.7a} gives a slight improvement replacing the $L^{\infty}$-norm by BMO-norm.
\end{rem}

\begin{proof}[Proof of Proposition \ref{prop3.7a}]
Write
\begin{align*}
( \partial_x^l [H,a] \partial_x^m f)(x) = \frac 1 {(2\pi)^2}
\int \sigma(\xi,\eta) \hat a (\xi) \hat f (\eta) e^{i (\xi+\eta) \cdot x} d\xi d\eta,
\end{align*}
where
\begin{align*}
\sigma(\xi,\eta) = i^{l+m} (\xi+\eta)^l \eta^m \cdot (-i) \cdot ( \operatorname{sgn}(\xi+\eta)
-\operatorname{sgn}(\eta) ).
\end{align*}
By a slight abuse of notation, we shall denote 
\begin{align*}
\sigma(a,f)= \partial_x^l [H,a] \partial_x^m f.
\end{align*}
Now note that neglecting the measure zero sets such as $\xi+\eta=0$ or $\eta=0$, 
the factor $\operatorname{sgn}(\xi+\eta) -\operatorname{sgn}(\eta)$ in $\sigma(\xi,\eta)$ does not vanish
only when $\xi+\eta>0$, $\eta<0$ or $\xi+\eta<0$, $\eta>0$.  In either cases easy to check
that $|\eta|<|\xi|$.  Then clearly
\begin{align*}
\sigma(a,f) = \sum_j \sigma(P_j a, P_{<j-2} f) + \sum_j  \sigma(P_j a, \tilde P_j f),
\end{align*}
where $\tilde P_j= P_{[j-2,j+2]}$.
For the first piece, write
\begin{align*}
\sum_j \sigma(P_ja ,P_{<j-2} f) &= H( \sum_j \partial_x^l (P_ja \partial_x^m P_{<j-2} f) )
-\sum_j \partial_x^l ( P_j a \partial_x^m H P_{<j-2} f) \notag \\
&=H(\sum_j\partial_x^l (   (D^{-(l+m)} P_j D^{l+m} a) \partial_x^m P_{<j-2} f) )
\notag \\
&\quad -\sum_j \partial_x^l ( (D^{-(l+m)} P_j D^{l+m} a) \partial_x^m P_{<j-2} Hf) \notag \\
& = H (\sigma_1(D^{l+m} a, f) ) + \sigma_2(D^{l+m} a, Hf),
\end{align*}
where the symbols $\sigma_1$, $\sigma_2$ satisfy the conditions \eqref{CM_cond_e1}--\eqref{CM_cond_e2}
(with $\xi$ and $\eta$ swapped therein). Thus 
\begin{align*}
\| \sum_j \sigma(P_j a, P_{<j-2} f) \|_p 
\lesssim \| D^{l+m} a \|_{\operatorname{BMO} } \|f\|_p
\lesssim \| \partial_x^{l+m} a\|_{\operatorname{BMO}} \| f \|_p.
\end{align*}
For the second  piece, write
\begin{align*}
\sum_j \sigma(P_ja ,\tilde P_j f) &= H( \sum_j \partial_x^l (P_ja \partial_x^m \tilde P_j f) )
-\sum_j \partial_x^l ( P_j a \partial_x^m \tilde P_j H f) \notag \\
& = H( \sum_j \partial_x^l (D^{-(l+m)}P_j D^{l+m} a \partial_x^m \tilde P_j f) )
\notag \\
&\qquad - \sum_{\tilde l=0}^l \binom l {\tilde l}  \sum_j \partial_x^{\tilde l}  (D^{-(l+m)} P_j 
D^{l+m} a) \partial_x^{m+l-\tilde l} \tilde P_j H f).\notag 
\end{align*}
Easy to check that the associated symbols again satisfy \eqref{CM_cond_e1}--\eqref{CM_cond_e2}, and we have
\begin{align*}
\| \sum_j \sigma(P_j a, \tilde P_j f) \|_p \lesssim \| \partial_x^{l+m} a\|_{\operatorname{BMO}} \| f \|_p.
\end{align*}
\end{proof}

In \cite{DMP08}, Dawson, McGahagan and Ponce also proved the following inequality: let $0\le \alpha<1$,
$0<\beta<1$, $0<\alpha+\beta\le 1$, $1<p,q<\infty$, $\delta>1/q$, then
\begin{align*}
\| D^{\alpha}[D^{\beta}, a] D^{1-(\alpha+\beta)} f \|_{L^p(\mathbb R)}
\lesssim_{\alpha,\beta,p,q,\delta} \| J^{\delta} \partial_x a \|_{L^q(\mathbb R)}
\| f \|_{L^p(\mathbb R)},
\end{align*}
where $D= (-\partial_{xx})^{1/2}$ and $J^{\delta}=(1-\partial_{xx})^{\delta/2}$. 
Note that
\begin{align*}
D^{\alpha}[D^{\beta}, a] D^{1-(\alpha+\beta)} f 
= D^{\alpha+\beta} ( a D^{1-(\alpha+\beta)} f ) - D^{\alpha} (a D^{1-\alpha} f ).
\end{align*}
The next proposition gives a sharp version of the above estimate. Moreover it holds on any 
$\mathbb R^d$, $d\ge 1$.
\begin{prop} \label{prop3.7b}
For any $0\le \alpha<1$, $0<\beta\le 1-\alpha$, $1<p<\infty$, we have
\begin{align*}
\|D^{\alpha+\beta} ( a D^{1-(\alpha+\beta)} f ) &- D^{\alpha} (a D^{1-\alpha} f )
+ \beta \nabla a \cdot D^{-1} \nabla f  \|_{L^p(\mathbb R^d)}
\notag \\
& \lesssim_{\alpha, \beta,p,d} 
\| D a \|_{\operatorname{BMO}} \| f \|_{L^p(\mathbb R^d)}.
\end{align*}
Consequently 
\begin{align*}
\|D^{\alpha+\beta} ( a D^{1-(\alpha+\beta)} f ) - D^{\alpha} (a D^{1-\alpha} f )
  \|_{L^p(\mathbb R^d)} \lesssim_{\alpha, \beta,p,d} 
\| \nabla a \|_{L^{\infty}(\mathbb R^d)} \| f \|_{L^p(\mathbb R^d)}.
\end{align*}
\end{prop}
\begin{rem*}
Clearly for dimension $d=1$, by using Sobolev embedding $W^{\delta,q} \hookrightarrow
L^{\infty}(\mathbb R)$ for $\delta>1/q$, we recover the Dawon-McGahagan-Ponce estimate.
\end{rem*}
\begin{proof}[Proof of Proposition \ref{prop3.7b}]
First  we denote
\begin{align*}
\sigma(a,f) = D^{\alpha+\beta} ( a D^{1-(\alpha+\beta)} f ) - D^{\alpha} (a D^{1-\alpha} f )
\end{align*}
with the symbol (by a slight abuse of notation)
\begin{align*}
\sigma(\xi,\eta)= |\xi+\eta|^{\alpha+\beta} |\eta|^{1-(\alpha+\beta)}
- |\xi+\eta|^{\alpha} |\eta|^{1-\alpha}.
\end{align*}
We then have
\begin{align*}
\sigma(a,f) = \sum_j \sigma(a_j, f_{<j-2}) + \sum_j \sigma(a_{<j-2},f_j) + \sum_j \sigma(a_j, \tilde f_j).
\end{align*}
For the high-low piece, we observe 
\begin{align*}
\sum_j D^{\alpha} (a_j D^{1-\alpha} f_{<j-2} ) = 
\sum_j \tilde P_j ( \tilde P_j (Da)  2^{-j(1-\alpha)} D^{1-\alpha} f_{<j-2} ),
\end{align*}
where by an abuse of notation we use $\tilde P_j$ to denote generic smooth frequency projection operators with frequency $|\xi|\sim 2^j$. Noting that $0\le \alpha<1$, we have
\begin{align*}
\| \sum_j \tilde P_j ( \tilde P_j (Da)  2^{-j(1-\alpha)} D^{1-\alpha} f_{<j-2} ) \|_p 
\lesssim \| Da \|_{\dot B^0_{\infty,\infty}} \| f \|_p  \lesssim \| Da \|_{\operatorname{BMO}} \| f \|_p.
\end{align*}
Similar estimate holds for the piece corresponding to the operator $D^{\alpha+\beta}$ if
$\alpha+\beta<1$. For $\alpha+\beta=1$, one
can use the duality argument as in the proof of Lemma \ref{lemq1}, it is easy to check that
\begin{align*}
\| \sum_j \tilde P_j ( \tilde P_j (Da)  f_{<j-2} ) \|_p 
\lesssim \| Da \|_{\operatorname{BMO}} \| f \|_p.
\end{align*}
 Thus
\begin{align*}
\|\sum_j \sigma(a_j, f_{<j-2}) \|_p \lesssim \| Da \|_{\operatorname{BMO}} \| f \|_p.
\end{align*}
The argument for the diagonal piece is similar, and we have 
\begin{align*}
\|\sum_j \sigma(a_j, \tilde f_j) \|_p \lesssim \| Da \|_{\operatorname{BMO}} \| f \|_p.
\end{align*}
Now we focus on the low-high piece where a correction term is needed for the final estimate.  Note that on this
piece $|\xi| \ll |\eta|$, and we shall write
\begin{align*}
 &\sigma(\xi,\eta)= |\xi+\eta|^{\alpha+\beta} |\eta|^{1-(\alpha+\beta)} - |\xi+\eta|^{\alpha} |\eta|^{1-\alpha}
 \notag \\
 =& (|\xi+\eta|^{\alpha+\beta} - |\eta|^{\alpha+\beta}) |\eta|^{1-(\alpha+\beta)}
 -(|\xi+\eta|^{\alpha} - |\eta|^{\alpha}) |\eta|^{1- \alpha} \notag \\
 =: & \sigma_2(\xi,\eta) - \sigma_3(\xi,\eta).
 \end{align*}
We now consider the piece corresponding to $\sigma_3$ (the estimate for $\sigma_2$ will be similar). Observe
\begin{align*}
|\eta+\xi|^{\alpha} - |\eta|^{\alpha} 
& = \alpha \int_0^1 |\eta+\theta \xi|^{\alpha-2} (\eta+\theta \xi) d\theta \cdot \xi \notag \\
& =  \alpha |\eta|^{\alpha-2} \eta \cdot \xi + \sum_{i,j=1}^d \tilde \sigma_{ij} (\xi,\eta) \xi_i \xi_j,
\end{align*}
where 
$\tilde \sigma_{ij}(\xi,\eta)$ is of order $|\eta|^{\alpha-2}$ when $|\xi|\ll |\eta|$. It is also easy to check
that the $m^{th}$ derivatives of $\tilde \sigma_{ij}$ decays as $O(|\eta|^{\alpha-2-m})$ when
$|\xi|\ll |\eta|$.  To simplify notation we shall write $\sum_{i,j=1}^d \tilde \sigma_{ij}\xi_i \xi_j$
simply as $\tilde \sigma(\xi,\eta) \xi^2$. Now we can write
\begin{align*}
\sum_j \sigma_3(a_{<j-2},f_j)
=-\sum_j \alpha \nabla a_{<j-2} \cdot D^{-1} \nabla f_j
+ \sum_j b_j,
\end{align*}
where
\begin{align*}
b_j(x) = \frac 1 {(2\pi)^{2d}}
\int \tilde \sigma(\xi,\eta) \xi^2 |\eta|^{1-\alpha} \widehat{a_{<j-2}}(\xi)
\widehat{f_j}(\eta) e^{i(\xi+\eta)\cdot x} d\xi d\eta.
\end{align*}
It is not difficult to check that
\begin{align*}
\| \sum_j b_j \|_p & \lesssim \|  (2^{-j} \mathcal M ( \partial^2 a_{<j-2}) \mathcal M (f_j) )_{l_j^2} \|_p
\notag \\
& \lesssim \| \partial a \|_{\dot B^0_{\infty,\infty}} \| f \|_p.
\end{align*}

On the other hand,  observe
\begin{align*}
\|\sum_j  \nabla a_{\ge j-2} \cdot D^{-1} \nabla f_j \|_p \lesssim \| \nabla a \|_{\operatorname{BMO}}
\| f \|_p.
\end{align*}
The desired result then easily follows.

\end{proof}

\begin{thm} \label{thm_cm_1}
Let $n_0=2d+2$ and $\sigma=\sigma(\xi, \eta) \in C_{\operatorname{loc}}^{n_0}
(\mathbb R^d\times \mathbb R^d \setminus (0,0) )$ satisfy
\begin{itemize}
\item  $\sigma(\xi,0)=0$, for any $\xi$ ;
\item $|\partial_{\xi}^{\alpha} \partial_{\eta}^{\beta} \sigma (\xi,\eta)|
\lesssim_{\alpha,\beta,d}
(|\xi|+ |\eta| )^{- (|\alpha|+|\beta|)}$, for any $ (\xi,\eta)\ne (0,0)$,
and any $ |\alpha|+|\beta| \le n_0$. 
\end{itemize}
For $f$, $g\in \mathcal S(\mathbb R^d)$ define
\begin{align*}
\sigma(D)(f,g)(x)
= \int \sigma(\xi,\eta) \hat f(\xi) \hat g(\eta) e^{i x \cdot (\xi+\eta)} d\xi d\eta.
\end{align*}
Then for any $1<p<\infty$, we have
\begin{align*}
\| \sigma(D)(f,g) \|_p
\lesssim_{p,d} 
\| \sigma\|_{\star} \| f\|_p \| g\|_{\bmo},
\end{align*}
where
\begin{align*}
\| \sigma\|_{\star} 
= \sup_{\substack{ (\xi,\eta) \ne (0,0)\\
|\alpha|+|\beta| \le n_0}} 
| (|\xi|+|\eta|)^{|\alpha|+|\beta|}
(\partial_{\xi}^{\alpha} \partial_{\eta}^{\beta} \sigma) (\xi,\eta) |.
\end{align*}

\end{thm}

\begin{proof}[Proof of Theorem \ref{thm_cm_1}]
In this proof we shall ignore the dependence on $(p,d)$ and write
$\lesssim_{p,d}$ simply as $\lesssim$. It suffices to show for any
$h \in C_c^{\infty}(\mathbb R^d)$, 
\begin{align*}
\langle \sigma(D)(f,g), h\rangle
\lesssim \| \sigma\|_{\star} \| f\|_p \|g\|_{\bmo} \|h\|_{p^{\prime}}, 
\end{align*}
where $p^{\prime} = p/(p-1)$.

We then only need to show for any $f  \in \mathcal S(\mathbb R^d)$,
$h \in C_c^{\infty}(\mathbb R^d)$, 
\begin{align*}
\| \tilde \sigma (D) (f,h) \|_{\mathcal H^1} 
\lesssim \| \sigma \|_{\star} \| f\|_p \| h \|_{p^{\prime}},
\end{align*}
where
\begin{align*}
\tilde \sigma(D) (f,h)(x)
= \int \sigma(\xi, -\xi -\eta) \hat f(\xi) \hat h(\eta) e^{i (\xi+\eta) \cdot x}  d\xi d\eta.
\end{align*}

Now write
\begin{align*}
\tilde \sigma(D)(f,h)
=\sum_j \tilde \sigma(D) (f_{\le j-2}, h_j)
+ \sum_j \tilde \sigma(D) (f_j, h_{\le j-2})
+ \sum_j \tilde \sigma(D) (f_j, \tilde h_j),
\end{align*}
where $\tilde h_j = h_{j-1} + h_j + h_{j+1}$.  We refer to these three summands as
low-high, high-low and diagonal pieces respectively.

\texttt{Low-high piece}. Note that $|\xi| \ll |\eta|$ and $|\xi+\eta| \sim |\eta| \sim 2^j$. Thus
\begin{align*}
\| \sum_j \tilde \sigma(D) (f_{\le j-2}, h_j ) \|_{\mathcal H^1}
\lesssim \| ( \tilde \sigma(D)(f_{\le j-2}, h_j ) )_{l_j^2} \|_{L_x^1}.
\end{align*}

Now we need a simple lemma.

\begin{lem} \label{thm_cm_1_lem1}
Suppose $\chi_1$, $\chi_2 \in C_c^{\infty}(\mathbb R^d)$ such that $\chi_1$ or $\chi_2$ is
supported on an annulus. Then
\begin{align*}
 & \left| \int \sigma(\xi,-\xi-\eta)
 \chi_1 ( \frac{\xi} {2^j} ) \chi_2 (\frac{\eta}{2^j} )
 e^{i \xi \cdot (x-y) } e^{i \eta \cdot (x-z) } d\xi d\eta \right| \notag \\
 \lesssim& 
 \sup_{\substack{|\alpha|+|\beta| \le 2d+1\\
 (\xi,\eta)\ne (0,0)}} 
 | (|\xi|+|\eta|)^{|\alpha|+|\beta|}
 \partial_{\xi}^{\alpha} 
 \partial_{\eta}^{\beta} \sigma(\xi,\eta)|
 \cdot \frac {2^{2jd}} { (1+2^j|x-y| +2^j |x-z|)^{2d+1}}.
 \end{align*}
 \end{lem}
 
 \begin{proof}[Proof of Lemma \ref{thm_cm_1_lem1}]
 WLOG we can assume $|x-y| \ge |x-z|$. If $|x-y| \le 2^{-j}$, then the bound is trivial.
 Now assume $|x-y| >2^{-j}$, then just integrate by parts in  the variable $\xi$ up to $(2d+1)$-times and
 note that $|\xi|+|\xi+\eta| \sim |\xi| +|\eta|$. More precisely if one denotes $\tilde \xi = 2^{-j}
 \xi$ and $\tilde \eta = 2^{-j} \eta$, then by the support assumption on $\chi_1$ and $\chi_2$,
 one has $|\tilde \xi| + |\tilde \xi +\tilde \eta|\sim 1$.  It follows that for $|\tilde \xi|\lesssim 1$,
 $|\tilde \eta|\sim 1$ or $|\tilde \xi| \sim 1$, $|\tilde \eta|\lesssim 1$, we have
 \begin{align*}
 &\max_{|\alpha|+|\beta|=2d+1}
 |\partial_{\tilde \xi}^{\alpha}
 \partial_{\tilde \eta}^{\beta}
 ( \sigma(2^j \tilde \xi, -2^j (\tilde \xi+\tilde \eta) ) ) | \notag \\
 \lesssim\;& 2^{j(2d+1)} \max_{|\alpha|+|\beta|=2d+1} (|\tilde \xi|+|\tilde \xi+\tilde \eta|)
 | (\partial_{\xi}^{\alpha} \partial_{\eta}^{\beta} \sigma)(2^j \tilde \xi,
 -2^j(\tilde \xi +\tilde \eta)) |\lesssim \| \sigma \|_{\star}.
 \end{align*}
 
 \end{proof}

By using Lemma \ref{thm_cm_1_lem1}, it is easy to check that 
\begin{align*}
 & | \tilde \sigma(D)(f_{\le j-2}, h_j) (x) | \notag \\
 \lesssim &\;\| \sigma \|_{\star} 
 \cdot 2^{jd} \int \frac {   |f_{\le j-2}(y)| } { (1+2^j |x-y|)^{d+\frac 12} }dy 
 \cdot 2^{jd} \int \frac{  |h_j(z)| } { (1+ 2^j |x-z| )^{d+\frac 12}  } dz \notag \\
 \lesssim &\;\| \sigma \|_{\star} \cdot \mathcal M f_{\le j-2} (x) \cdot \mathcal M h_j (x).
 \end{align*}
 
 Therefore
 \begin{align*}
 \| (\tilde \sigma(D) (f_{\le j-2}, h_j ) )_{l_j^2} \|_1
 & \lesssim \| \sigma\|_{\star} \cdot \| (\mathcal M f_{\le j-2} )_{l_j^{\infty}} \|_p
 \cdot \| (\mathcal M h_j )_{l_j^2} \|_{p^{\prime}} \notag \\
 & \lesssim \| \sigma \|_{\star} \cdot \| f\|_{p} \cdot \| h\|_{p^{\prime}}.
 \end{align*}

This finishes the estimate of the low-high piece. 
The estimate for
the high-low piece is similar.

\texttt{Diagonal-piece}.  In this case we need to consider the integral
\begin{align*}
 & \tilde \sigma(D) (f_j, \tilde h_j )(x) \notag \\
 = &\int \sigma(\xi,-\xi-\eta)
 \chi_1(\frac{\xi} {2^j} ) \chi_2(\frac{\eta} {2^j} )
 \hat{f_j}(\xi) \widehat{\tilde h_j} (\eta) e^{i (\xi +\eta) \cdot x} d\xi d\eta,
 \end{align*}
 where $\chi_1$, $\chi_2$ are smooth cut-off functions with support in the annulus
 $\{z\in \mathbb R^d:\, 2^{-m_0} \le |z| \le 2^{m_0} \}$ for some integer $m_0>0$.
 
 We now consider two subcases.
 
 Subcase 1: $|\xi+\eta| \ge 2^{j-m_0-10}$. Note that in this case $|\xi+\eta| \sim 2^j$.
 By using frequency localization and Lemma \ref{thm_cm_1_lem1}, we have
 \begin{align*}
  & \| \sum_j P_{>j-m_0-10} ( \tilde \sigma (D) (f_j, \tilde h_j) ) \|_{\mathcal H^1} \notag \\
  \sim & \| \sum_j P_{j-m_0-10<\cdot <j+m_0+100}
  (\tilde \sigma(D) (f_j,\tilde h_j ) ) \|_{\mathcal H^1} \notag \\
  \lesssim & \| ( \mathcal Mf_j \cdot \mathcal M \tilde h_j )_{l_j^2} \|_1 \notag \\
  \lesssim & \| f\|_p \|h\|_{p^{\prime}}.
  \end{align*}
  
  Subcase 2: $|\xi +\eta| <2^{j-m_0-10}$. In yet other words, $|\xi+\eta|
  \ll \min \{ |\xi|, \, |\eta| \}$. Then since $\sigma(\xi,0)=0$ by assumption, 
  we just need to consider 
  \begin{align*}
  \int \chi_{|\xi+\eta| <2^{j-m_0-10}}
  ( \sigma(\xi,-\xi-\eta) -\sigma(\xi,0)) \cdot
  \chi_1(\frac{\xi}{2^j} ) \chi_2(\frac{\eta}{2^j} )
  \hat f_j(\xi) \hat{\tilde h}_j (\eta) e^{i (\xi+\eta) \cdot x} d\xi d\eta,
  \end{align*}
  where $\chi$ is a smooth cut-off function.
  
  Now we need another simple lemma.
  
  \begin{lem} \label{thm_cm_1_lem2}
  For all $0\le \theta \le 1$, we have
 \begin{align*}
  & 
  \left| \int \chi_{|\xi+\eta| <2^{j-m_0-10}}
  (\partial_{\eta} \sigma)(\xi,  -\theta (\xi+\eta) )
  \chi_1(\frac{\xi}{2^j} ) \chi_2(\frac {\eta} {2^j} )
  e^{i \xi \cdot (x-y) } e^{i \eta \cdot (x-z)} d\xi d\eta 
  \right| \notag \\
  \lesssim & \;
  \sup_{\substack{|\alpha|+|\beta| \le 2d+2 \\
  (\xi,\eta) \ne (0,0)}}
  | (|\xi|+|\eta|)^{|\alpha|+|\beta|}
  (\partial_{\xi}^{\alpha} \partial_{\eta}^{\beta} \sigma) (\xi,\eta) |
  \notag \\
  & \qquad \qquad \qquad \cdot 2^{2jd-j} \cdot (1+2^j |x-y|+2^j |x-z|)^{-(2d+1)}.
  \end{align*}
  \end{lem}

The proof of Lemma \ref{thm_cm_1_lem2} is similar to Lemma
\ref{thm_cm_1_lem1} and therefore we omit it.

By Lemma \ref{thm_cm_1_lem2}, we then have
\begin{align*}
 & \| \sum_j \tilde \sigma(D) (f_j, \tilde h_j ) \|_{\dot B^0_{1,1} } \notag \\
 \lesssim & \sum_k 2^k \sum_{j\ge k+m_0+10} 2^{-j} \| \mathcal M f_j \cdot \mathcal M \tilde h_j \|_1
 \notag \\
 \lesssim & \sum _j \| \mathcal M f_j \cdot \mathcal M \tilde h_j \|_1 \notag \\
 \lesssim & \| (\mathcal M f_j )_{l_j^2} \|_p \cdot \| (\mathcal M \tilde h_j)_{l_j^2} \|_{p^{\prime}} \notag \\
 \lesssim & \| f\|_p \| h\|_{p^{\prime}}.
 \end{align*}
 
 This concludes the proof of Theorem \ref{thm_cm_1}.
 
 \begin{rem*}
 A close inspection shows that the diagonal piece in fact belongs to $\dot B^0_{1,1}$ which embeds into
 $\mathcal H^1$. 
 \end{rem*}

\end{proof}

It is possible to refine Theorem \ref{thm_cm_1} further. The following only
requires $2d+1$ derivatives on the symbol $\sigma$. 

\begin{thm} \label{thm_cm_1a}
Let $n_0=2d+1$ and $\sigma=\sigma(\xi, \eta) \in C_{\operatorname{loc}}^{n_0}
(\mathbb R^d\times \mathbb R^d \setminus (0,0) )$ satisfy
\begin{itemize}
\item  $\sigma(\xi,0)=0$, for any $\xi$ ;
\item $|\partial_{\xi}^{\alpha} \partial_{\eta}^{\beta} \sigma (\xi,\eta)|
\lesssim_{\alpha,\beta,d}
(|\xi|+ |\eta| )^{- (|\alpha|+|\beta|)}$, for any $ (\xi,\eta)\ne (0,0)$,
and any $ |\alpha|+|\beta| \le n_0$. 
\end{itemize}
For $f$, $g\in \mathcal S(\mathbb R^d)$ define
\begin{align*}
\sigma(D)(f,g)(x)
= \int \sigma(\xi,\eta) \hat f(\xi) \hat g(\eta) e^{i x \cdot (\xi+\eta)} d\xi d\eta.
\end{align*}
Then for any $1<p<\infty$, we have
\begin{align*}
\| \sigma(D)(f,g) \|_p
\lesssim_{p,d} 
\| \sigma\|_{\star} \| f\|_p \| g\|_{\bmo},
\end{align*}
where
\begin{align*}
\| \sigma\|_{\star} 
= \sup_{\substack{ (\xi,\eta) \ne (0,0)\\
|\alpha|+|\beta| \le n_0}} 
| (|\xi|+|\eta|)^{|\alpha|+|\beta|}
(\partial_{\xi}^{\alpha} \partial_{\eta}^{\beta} \sigma) (\xi,\eta) |.
\end{align*}
\end{thm}

\begin{proof}[Proof of Theorem \ref{thm_cm_1a}]
We shall use the same notation as in the proof of Theorem \ref{thm_cm_1} and sketch the needed modifications.
It suffices to show 
\begin{align*}
\sum_k \sum_{j \ge k+m_0+10} \| P_k ( \tilde \sigma(D) (f_j, \tilde h_j) ) \|_{L_x^1}
\lesssim \| f \|_p \| h \|_{p^{\prime}}.
\end{align*}

Clearly
\begin{align}
&\| P_k ( \tilde \sigma(D) (f_j, \tilde h_j) ) \|_{L_x^1} \notag \\
\lesssim & 
\| \int \chi( \frac{\xi+\eta} {2^k} ) \chi( 2^{-j} \xi) \chi(2^{-j} \eta)
\sigma(\xi, -\xi-\eta) e^{i (\xi\cdot (x-y)+\eta\cdot (x-z)) }  \notag \\
&\qquad\qquad
d\xi d\eta f_j (y) \tilde h_j (z) dy dz \|_{L_x^1},  \label{pf_thmcm1a_e1}
\end{align}
where we have slightly abused the notation and denote all smooth cutoff functions by
the same symbol $\chi$ (whose support is on the annulus $\{ z:\; |z| \sim 1\}$). 
Now set $\tilde \xi= \xi+\eta$, $\tilde \eta = \eta$. Then
\begin{align}
 & |\int \chi(\frac{\xi+\eta} {2^k}) \chi(2^{-j} \xi) \chi(2^{-j} \eta) \sigma(\xi,-\xi-\eta)
 e^{i ( \xi \cdot (x-y) + \eta \cdot (x-z))} d\xi d\eta | \notag \\
 =& | \int \chi( 2^{-k} \tilde \xi) \chi (2^{-j}(\tilde \xi-\tilde \eta) )
 \chi(2^{-j} \tilde \eta) \sigma(\tilde \xi-\tilde \eta, -\tilde \xi) e^{i ( \tilde \xi\cdot (x-y)
 + \tilde \eta\cdot (y-z) )} d\tilde \xi d\tilde \eta|. \label{pf_thmcm1a_e3}
 \end{align}

Case 1: $|y-z|\ge \frac 1 {10} |x-y|$ or $|y-z|\ge \frac 1 {10} |x-z|$.
 In this case easy to check that $|z-x|+|y-x| \lesssim |y-z|$. Integrating by
parts in $\tilde \eta$ up to $2d+1$ times, we obtain
\begin{align*}
\eqref{pf_thmcm1a_e3} &\lesssim  (1+2^j|y-z|)^{-(2d+1)} \cdot 2^{jd} \cdot 2^{kd} \notag \\
& \lesssim  (1+2^j|x-y|)^{-(d+\frac 12)} (1+2^j|x-z|)^{-(d+\frac 12)} \cdot 2^{2jd} \cdot 2^{(k-j)d}.
\end{align*}
It follows easily that in this case 
\begin{align*}
\eqref{pf_thmcm1a_e1} \lesssim 2^{(k-j)d} \| \mathcal M f_j \mathcal M \tilde h_j \|_{L_x^1}.
\end{align*}
Summing in $k$ and $j$ then easily yields the desired inequality.

Case 2: $|y-z|<\frac 1 {10} \min\{ |x-y|, \, |x-z|\}$.  In this case $|x-y| \sim |x-z|$. 
We estimate \eqref{pf_thmcm1a_e3} in several different ways. First  integrating by parts in
$\tilde \xi$ up to $d+1$ times and then in $\tilde \eta$ up to $d$ times, we get
\begin{align*}
\eqref{pf_thmcm1a_e3} \lesssim (1+2^k|x-y|)^{-(d+1)} (1+2^j|y-z|)^{-d}
2^{jd} 2^{kd}.
\end{align*}
Then  integrating by parts in
$\tilde \xi$ up to $d$ times and then in $\tilde \eta$ up to $d+1$ times, we get
\begin{align*}
\eqref{pf_thmcm1a_e3} \lesssim (1+2^k|x-y|)^{-d} (1+2^j|y-z|)^{-(d+1)}
2^{jd} 2^{kd}.
\end{align*}
Interpolating these two estimates gives us
\begin{align} \label{pf_thmcm1a_e5}
\eqref{pf_thmcm1a_e3} \lesssim (1+2^k|x-y|)^{-(d+\frac 12)} (1+2^j|y-z|)^{-(d+\frac 12)}
2^{jd} 2^{kd}.
\end{align}
Since $\sigma(\xi,0)=0$ for any $\xi$, we have
\begin{align*}
\sigma(\tilde \xi-\tilde \eta, -\tilde \xi) = -\tilde \xi \cdot \int_0^1 (\partial_{\eta} \sigma)
(\tilde \xi-\tilde \eta,-\theta \tilde \xi) d\theta.
\end{align*}
It is then not difficult to check that
\begin{align*}
\eqref{pf_thmcm1a_e3} \lesssim  2^{k-j} 
(1+2^k|x-y|)^{-d} (1+2^j|y-z|)^{-d}
2^{jd} 2^{kd}.
\end{align*}
Interpolating this estimate with \eqref{pf_thmcm1a_e5} yields
\begin{align*}
\eqref{pf_thmcm1a_e3} \lesssim  2^{\frac 12(k-j)} 
(1+2^k|x-y|)^{-(d+\frac 14)} (1+2^j|y-z|)^{-(d+\frac 14)}
2^{jd} 2^{kd}.
\end{align*}
We then get
\begin{align*}
\eqref{pf_thmcm1a_e1} &
\lesssim 2^{kd} 2^{\frac 12(k-j)}
\int (1+2^k|x-y|)^{-(d+\frac 14)} 
|f_j(y)| (\mathcal M \tilde h_j)(y) dy dx \notag \\
& \lesssim 2^{\frac 12(k-j)}
\int |f_j(y)| (\mathcal M \tilde h_j)(y) dy.
\end{align*}
Summing in $k$ and $j$  then easily implies the desired result.
\end{proof}

The next two simple lemmas will be needed in Section 5.  It is interesting that one can obtain some
BMO (or Besov) refinements of some terms in Kato-Ponce inequalities.

\begin{lem}\label{lemc3}
Let $s>1$ be an integer and let $1<p<\infty$. Let $\mathcal R$ be the usual Riesz-type operator. Then
for any $f,g \in \mathcal S(\mathbb R^d)$, we have
\begin{align*}
\| \partial^s f \cdot \mathcal R g \|_p \lesssim_{s,p,d,r_1,r_2}
\| D^s f\|_{\bmo} \| g\|_p + \| \partial f \|_{r_1} \| D^{s-1}  g \|_{r_2},
\end{align*}
where $\frac 1p=\frac 1 {r_1}+\frac 1 {r_2}$, and $1<r_1,r_2<\infty$. The notation $\partial^s$ denotes
any differentiation operator $\partial_{x_1}^{\alpha_1}\cdots \partial_{x_d}^{\alpha_d}$ with $|\alpha|=s$.
If $r_1=\infty$, $r_2=p$, then
\begin{align*}
\| \partial^s f \cdot \mathcal R g \|_p \lesssim_{s,p,d}
\| D^s f\|_{\bmo} \| g\|_p + \| \partial f \|_{\dot B^0_{\infty,\infty}} \| D^{s-1} g\|_p.
\end{align*}
And if $r_1=p$, $r_2=\infty$, then
\begin{align*}
\| \partial^s f \cdot \mathcal R g \|_p \lesssim_{s,p,d}
\| D^s f \|_{\bmo} \| g \|_p + \| \partial f \|_p \| D^{s-1} g \|_{\dot B^0_{\infty,\infty}}.
\end{align*}

\end{lem}

\begin{rem}
The same estimates also hold for $\|\partial^s f \cdot g\|_p$.
\end{rem}
\begin{proof}[Proof of Lemma \ref{lemc3}]
Write
\begin{align*}
\partial^s f \cdot \mathcal R g
& = \sum_j \partial^s f_{<j-2} \mathcal R g_j + \sum_j \partial^s f_{\ge j-2} \mathcal R g_j \notag \\
& = \sum_j \partial^s f_{<j-2} \mathcal R g_j + \sum_j \partial^s f_{j} \mathcal R g_{\le j+2}.
\end{align*}

First note that by Lemma \ref{lemq1}, we have
\begin{align*}
\| \sum_j \partial^s f_j \cdot \mathcal R g_{\le j+2} \|_p \lesssim \| \partial^s f\|_{\bmo} \| \mathcal R g\|_p
\lesssim \| D^s f \|_{\bmo} \| g\|_p.
\end{align*}
Thus we only need to estimate the piece $\sum_j \partial^s f_{<j-2} \mathcal R g_j$. Consider
first the case $r_2<\infty$, then
\begin{align*}
\| \sum_j \partial^s f_{<j-2} \mathcal R g_j \|_p
& \lesssim \| (\partial^s f_{<j-2} \cdot \mathcal R g_j )_{l_j^2} \|_p \notag \\
& \lesssim \| (2^{-j(s-1)} \partial^s f_{<j-2} )_{l_j^{\infty}} (2^{j(s-1)} \mathcal R g_j )_{l_j^2} \|_p \notag \\
& \lesssim
\begin{cases}
 \| \partial f \|_{r_1} \| D^{s-1} g\|_{r_2}, \quad \text{if $r_1<\infty$}, \\
 \| \partial f \|_{\dot B^0_{\infty,\infty}} \| D^{s-1} g \|_p, \quad \text{if $r_1=\infty$}.
 \end{cases}
\end{align*}

If $r_2=\infty$, then $r_1=p$ and
\begin{align*}
\| \sum_j \partial^s f_{<j-2} \mathcal R g_j \|_p
& \lesssim \| (2^{-j(s-1)} \partial^s f_{<j-2} )_{l_j^2} \cdot
(2^{j(s-1)} \mathcal R g_j )_{l_j^{\infty}} \|_p \notag \\
& \lesssim \| (2^{-j(s-1)} \partial^s f_{<j-2})_{l_j^2} \|_p \cdot \| D^{s-1}  g \|_{\dot B^0_{\infty,\infty}}
\notag \\
& \lesssim \| \partial f \|_p \| D^{s-1} g \|_{\dot B^0_{\infty,\infty} }.
\end{align*}

\end{proof}

\begin{lem} \label{lemc4}
Let $s>1$ and $1<p<\infty$. Let $1<p_1,p_2\le \infty$
satisfy $\frac 1{p_1}+\frac 1{p_2}= \frac 1p$.
 Then for any $f,g \in \mathcal S(\mathbb R^d)$, we have
\begin{align}
&\| \partial f \cdot D^{s-2} \partial g \|_p \lesssim \| \partial f \|_{p_1} \| D^{s-1} g \|_{p_2},
\quad \text{if $1<p_1\le \infty$, $1<p_2<\infty$}, \label{lemc4_e1}\\
&\| \partial f \cdot D^{s-2} \partial g \|_p \lesssim
\| D^s f \|_{p_1} \| g\|_{p_2} + \| \partial f \|_p \| D^{s-1} g \|_{\bmo}, \;\, \label{lemc4_e2}\\
&\qquad\qquad\qquad
\text{if $1<p_1 <\infty$, $1<p_2<\infty$}, \notag\\
& \| \partial f \cdot D^{s-2} \partial g \|_p \lesssim
\| D^s f \|_p \| g \|_{\dot B^0_{\infty,\infty}} + \| \partial f \|_p \| D^{s-1} g\|_{\bmo}, \;\,
\label{lemc4_e3}\\
&\| \partial f \cdot D^{s-2} \partial g \|_p
\lesssim \| D^s f \|_{\dot B^0_{\infty,\infty}} \| g \|_p + \| \partial f \|_p \| D^{s-1} g \|_{\bmo}. \label{lemc4_e4}
\end{align}

\end{lem}
\begin{proof}[Proof of Lemma \ref{lemc4}]
The first inequality is trivial. For \eqref{lemc4_e2}, by using frequency localization and Lemma \ref{lemq1},
we have
\begin{align*}
\| \partial f & \cdot D^{s-2} \partial g \|_p \notag \\
& \lesssim \| \sum_j \partial f_j \cdot D^{s-2} \partial g_{<j-2}
\|_p + \| \sum_j \partial f_{\le j+2}  \cdot D^{s-2} \partial g_j \|_p \notag \\
& \lesssim \| (2^{j(s-1)} \partial f_j)_{l_j^2} (2^{-j(s-1)} D^{s-2} \partial g_{<j-2})_{l_j^{\infty}}
\|_p +\| \partial f \|_p \|D^{s-1} g \|_{\bmo} \notag \\
& \lesssim \| D^s f\|_{p_1} \| g\|_{p_2} + \| \partial f \|_p \|D^{s-1} g \|_{\bmo}.
\end{align*}
For \eqref{lemc4_e3}, just observe
\begin{align*}
\|(2^{-j(s-1)} D^{s-2} \partial g_{<j-2})_{l_j^{\infty}} \|_{\infty}
& \lesssim \| g\|_{\dot B^0_{\infty,\infty}}.
\end{align*}
For \eqref{lemc4_e4}, we have
\begin{align*}
\| \sum_j \partial f_j \cdot D^{s-2} \partial g_{<j-2}
\|_p  & \lesssim \| (2^{j(s-1)} \partial f_j)_{l_j^{\infty}} \|_{\infty} \| (2^{-j(s-1)} D^{s-2}
\partial g_{<j-2} )_{l_j^2} \|_p \notag \\
& \lesssim \| D^s f \|_{\dot B^0_{\infty,\infty}} \| g\|_p.
\end{align*}

\end{proof}

\section{Proof of Theorem \ref{thm2}}
In the first subsection, we shall give the proof of Theorem \ref{thm2} for the case $1<p<\infty$.
In the second subsection, we sketch the needed modification for the case $\frac 12 <p\le 1$.

\subsection{The case $1<p<\infty$}
To prove Theorem \ref{thm2} for the case $1<p<\infty$, we first prove the following proposition.

\begin{prop} \label{propforthm2}
Let $s>0$ and $1<p<\infty$. Let $s_1,s_2\ge 0$ and $s_1+s_2=s$.
Then for any $f$, $g \in \mathcal S(\mathbb R^d)$, we have

\begin{align}
 &\Bigl\| D^s(fg) - \sum_j \Bigl( \sum_{|\alpha|\le s_1} \frac 1 {\alpha!}
\partial^{\alpha} f_{\le j-2} D^{s,\alpha} g_j
+ \sum_{|\beta|\le s_2} \frac 1 {\beta!} \partial^{\beta} g_{\le j-2} D^{s,\beta} f_j \Bigr) \Bigr\|_p  \notag \\
\lesssim &
\begin{cases}
\| D^{s_1} f\|_{p_1} \cdot \| D^{s_2} g\|_{p_2}, \quad \text{if $1<p_1,p_2<\infty$, $\frac 1p=\frac 1 {p_1}+\frac 1{p_2}$};\\
\| D^{s_1} f\|_p \| D^{s_2} g \|_{\dot B^0_{\infty,\infty}}, \quad \text{if $p_1=p$, $p_2=\infty$};\\
\|D^{s_1} f \|_{\dot B^0_{\infty,\infty}} \cdot \|D^{s_2} g \|_p, \quad \text{if $p_1=\infty$, $p_2=p$}.
\end{cases}
\end{align}
\end{prop}

\begin{proof}[Proof of Proposition \ref{propforthm2}]
We write
\begin{align*}
f g = \sum_{j \in \mathbb Z} f_j \tilde g_j + \sum_{j\in \mathbb Z} f_{\le j-2} g_j +
\sum_{j \in \mathbb Z} g_{\le j-2} f_j,
\end{align*}
where $\tilde g_j = g_{j-1} + g_j +g_{j+1}$.

We shall analyze each term separately.

1) \texttt{The diagonal piece}. Denote $h=\sum_{j} f_j \tilde g_j$. Then
\begin{align*}
\| D^s h \|_p & \lesssim \| (2^{ks} P_k h)_{l_k^2} \|_p \notag \\
& = \| ( 2^{ks} P_k( \sum_{j\ge k-10} f_j \tilde g_j ) )_{l_k^2} \|_p \notag \\
& \lesssim \sum_{l\ge -10}  \| ( 2^{ks} f_{k+l} \tilde g _{k+l} )_{l_k^2} \|_p \notag \\
& \lesssim \sum_{l \ge -10} 2^{-ls} \| (2^{ks} f_k \tilde g_k )_{l_k^2} \|_p
\lesssim \| (2^{ks} f_k \tilde g_k )_{l_k^2} \|_p.
\end{align*}
Now discuss three cases.

If $p_1<\infty$ and $p_2<\infty$, then
\begin{align*}
\| (2^{ks} f_k \tilde g_k )_{l_k^2} \|_p & \lesssim \| (2^{ks_1} f_k)_{l_k^2} \|_{p_1}
\cdot \| (2^{k s_2} \tilde g_k )_{l_k^{\infty}} \|_{p_2} \notag \\
& \lesssim \| D^{s_1} f \|_{p_1} \cdot \| D^{s_2} g \|_{p_2}.
\end{align*}

If $p_1=p$ and $p_2=\infty$, then
\begin{align*}
\| (2^{ks} f_k \tilde g_k )_{l_k^2} \|_p & \lesssim \| D^{s_1} f \|_p \cdot \| D^{s_2} g \|_{\dot B^0_{\infty,\infty}}.
\end{align*}

Similarly if $p_1=\infty$ and $p_2=p$, then
\begin{align*}
\| (2^{ks} f_k \tilde g_k)_{l_k^2} \|_p
& \lesssim \|D^{s_1} f \|_{\dot B^0_{\infty,\infty}} \|D^{s_2} g\|_p.
\end{align*}
Thus the diagonal piece is OK.

2) \texttt{The low-high piece}. For each $j$, write
\begin{align} \label{low-high_e1}
D^s(f_{\le j-2} g_j) = [D^s, f_{\le j-2}] g_j + f_{\le j-2} D^s g_j.
\end{align}

We shall simplify further the commutator piece $[D^s, f_{\le j-2}]g_j$. To write out its explicit
form, we need to use a fattened frequency cut-off $\tilde P_j$ such that $\tilde P_j P_j =P_j$.
More precisely let $\tilde \phi \in C_c^{\infty}(\mathbb R^d)$ be such that
\begin{align*}
\tilde \phi(\xi)=1, \qquad \forall\, \xi \in \operatorname{supp}(\phi),
\end{align*}
where we recall $\widehat{P_j \delta_0}(\xi)= \phi(2^{-j} \xi)$.  Define $\widehat{\tilde P_j \delta_0}(\xi)
=\tilde \phi(2^{-j} \xi)$. 
Clearly
\begin{align*}
(D^s \tilde P_j \delta_0)(y) = 2^{j(d+s)} \psi_1 (2^j y)
\end{align*}
where $\widehat{\psi_1}(\xi)= |\xi|^s \tilde \phi(\xi)$. 
Then
\begin{align}
 & ([D^s, f_{\le j-2}] g_j )(x) \notag \\
 = & 2^{j(d+s)} \int_{\mathbb R^d}
 \psi_1(2^j y) (f_{\le j-2}(x-y) -f_{\le j-2}(x) ) g_j(x-y) dy. \label{a20}
 \end{align}

Denote
\begin{align*}
h(\theta) = f_{\le j-2} (x-\theta y).
\end{align*}

Set $m_0 = [s_1]$. We then Taylor expand $h(\theta)$ up to $m_0^{\operatorname{th}}$ term:
\begin{align} \label{e_herror}
h(1)-h(0) = h^{\prime}(0) + \cdots + \frac{h^{(m_0)}(0)} {m_0 !}
+\underbrace{ \int_0^1 \frac{(1-\theta)^{m_0}} {m_0!} h^{(m_0+1)} (\theta) d\theta }_{\operatorname{error}}.
\end{align}

\underline{Contribution of the ``error'' term}.

We first show that the contribution of the ``error term'' in \eqref{e_herror} to \eqref{a20} can be bounded
by $\|D^{s_1} f \|_{p_1} \| D^{s_2} f \|_{p_2}$ (after summation in $j$).

Easy to check that
\begin{align*}
h^{(m_0+1)} (\theta) = \sum_{|\alpha|=m_0+1} C_{\alpha} (\partial^{\alpha} f_{\le j-2}) (x-\theta y) \cdot y^{\alpha},
\end{align*}
where $C_{\alpha}$ are constant coefficients whose value do not matter in our estimates.

By Lemma \ref{lemp2}, we have
\begin{align*}
| (\partial^{m_0+1} f_{\le j-2} )(x-\theta y) | \lesssim
(1+2^j |y|)^d \mathcal M ( |\partial^{m_0+1} f_{\le j-2} |) (x).
\end{align*}

Then
\begin{align*}
 &  2^{j(d+s)} \int_{\mathbb R^d}
 |\psi_1 (2^j y) | \cdot |y|^{m_0+1} \cdot
 | \partial^{m_0+1} f_{\le j-2} (x-\theta y) | \cdot |g_j(x-y) | dy \notag \\
 \lesssim &\; 2^{-j(m_0+1-s_1)}
 \mathcal M ( |\partial^{m_0+1} f_{\le j-2} |) (x) \cdot
 (\mathcal M g_j)(x) \cdot 2^{js_2}.
 \end{align*}

 Now discuss two cases.

 Case 1: $p_1 \le \infty$, $p_2 <\infty$. Then
  \begin{align*}
   & \Bigl\| ( 2^{j(d+s)} \int \psi_1(2^j y) \cdot \int_0^1 \frac{(1-\theta)^{m_0}} {m_0!}
   h^{(m_0+1)}(\theta) d\theta \cdot g_j(x-y) dy )_{l_j^2} \Bigr\|_p \notag \\
   \lesssim & \;
   \| (2^{-j(m_0+1-s_1)} \mathcal M ( |\partial^{m_0+1} f_{\le j-2} |) )_{l_j^{\infty}}
   \cdot (\mathcal M g_j \cdot 2^{js_2} )_{l_j^2} \|_p \notag \\
   \lesssim & \| \mathcal M ( \sup_j 2^{-j(m_0+1-s_1)} |\partial^{m_0+1} f_{\le j-2} |) \|_{p_1}
   \cdot \| (\mathcal M (2^{js_2} g_j) )_{l_j^2} \|_{p_2} \notag \\
   \lesssim &\;
   \begin{cases}
   \| D^{s_1} f \|_{p_1} \cdot \| D^{s_2} g\|_{p_2}, \qquad \text{if $p_1<\infty$, $p_2<\infty$} \\
   \| D^{s_1} f \|_{\dot B^0_{\infty,\infty}} \cdot \| D^{s_2} g \|_p,
   \qquad \text{if $p_1=\infty$, $p_2<\infty$}.
   \end{cases}
   \end{align*}

 Case 2: $p_1=p$, $p_2=\infty$. Then
\begin{align*}
 & \Bigl\| ( 2^{j(d+s)} \int \psi_1(2^j y) \cdot \int_0^1 \frac{(1-\theta)^{m_0}} {m_0!}
   h^{(m_0+1)}(\theta) d\theta \cdot g_j(x-y) dy )_{l_j^2} \Bigr\|_p \notag \\
   \lesssim & \;
   \| ( 2^{-j(m_0+1-s_1)} \mathcal M ( |\partial^{m_0+1} f_{\le j-2} |) )_{l_j^2}
   \cdot (\mathcal M g_j \cdot 2^{js_2} )_{l_j^{\infty}} \|_p \notag \\
   \lesssim &\;
   \| ( \mathcal M( 2^{-j(m_0+1-s_1)} |\partial^{m_0+1} f_{\le j-2} | ) )_{l_j^2} \|_p
   \cdot \| D^{s_2} g \|_{\dot B^0_{\infty,\infty}} \notag \\
   \lesssim &\; \| D^{s_1} f\|_p \cdot \|D^{s_2} g \|_{\dot B^0_{\infty,\infty}}.
   \end{align*}
   Thus the contribution of the ``error'' term to \eqref{a20} is OK for us.

\underline{The form of the other terms in \eqref{e_herror} }.

Easy to check that
\begin{align*}
 & h^{\prime}(0) + \cdots+ \frac{h^{(m_0)}(0)} {m_0!} \notag \\
 = &\; \sum_{0<|\alpha|\le m_0} \frac 1 {\alpha!} (\partial^{\alpha} f_{\le j-2}) (x) \cdot (-1)^{|\alpha|} y^{\alpha}.
 \end{align*}

 Therefore in \eqref{a20}, we have
 \begin{align*}
  & 2^{j(d+s)} \int_{\mathbb R^d} \psi_1(2^j y) \cdot ( \sum_{l=1}^{m_0} \frac{h^{(l)}(0)} {l!} ) g_j(x-y) dy
  \notag \\
  = & \; \sum_{0<|\alpha|\le m_0}
  2^{j(d+s)} \int_{\mathbb R^d}
  \psi_1(2^j y) \cdot
  \frac 1{\alpha!} (\partial^{\alpha} f_{\le j-2} )(x)
  \cdot (-1)^{|\alpha|} \cdot y^{\alpha} \cdot g_j(x-y) dy \notag \\
  = & \; \sum_{0<|\alpha| \le m_0}
  \frac 1 {\alpha!} \cdot (-1)^{|\alpha|}
  \cdot 2^{js}
  \cdot \mathcal F^{-1} \Bigl( i^{\alpha} \partial_{\xi}^{\alpha}( \widehat{\psi_1}(\xi/2^j )) \widehat{g_j}(\xi) \Bigr)(x)
  \cdot (\partial^{\alpha} f_{\le j-2} )(x) \notag \\
  =& \; \sum_{0<|\alpha|\le m_0} \frac 1 {\alpha!} \mathcal F^{-1} \Bigl(i^{-|\alpha|} \partial_{\xi}^{\alpha}(|\xi|^s)
  \widehat{g_j}(\xi) \Bigr) (x) \cdot (\partial^{\alpha} f_{\le j-2})(x) \notag \\
  =& \; \sum_{0<|\alpha|\le m_0} \frac 1 {\alpha!} D^{s,\alpha} g_j \cdot \partial^{\alpha} f_{\le j-2}.
  \end{align*}

Note that the second term in \eqref{low-high_e1} corresponds to $\alpha=0$. Thus the low-high piece can be written as
\begin{align*}
\Bigl(\sum_j \sum_{0\le |\alpha| \le [s_1]} \frac 1{\alpha!} \partial^{\alpha} f_{\le j-2} D^{s,\alpha} g_j
\Bigr) + E_1,
\end{align*}
where
\begin{align} \label{a20_b10}
\| E_1 \|_p \lesssim
\begin{cases}
\| D^{s_1} f\|_{p_1} \cdot \| D^{s_2} g\|_{p_2}, \quad \text{if $1<p_1,p_2<\infty$};\\
\| D^{s_1} f\|_p \| D^{s_2} g \|_{\dot B^0_{\infty,\infty}}, \quad \text{if $p_1=p$, $p_2=\infty$};\\
\|D^{s_1} f \|_{\dot B^0_{\infty,\infty}} \cdot \|D^{s_2} g \|_p, \quad \text{if $p_1=\infty$, $p_2=p$}.
\end{cases}
\end{align}

3)\texttt{The high-low piece}.
This is similar to the low-high piece. One can get the results by symmetry.

Collecting all the estimates, we obtain
\begin{align}
D^s(fg) =\sum_j \Bigl( \sum_{|\alpha|\le s_1} \frac 1 {\alpha!}
\partial^{\alpha} f_{\le j-2} D^{s,\alpha} g_j
+ \sum_{|\beta|\le s_2} \frac 1 {\beta!} \partial^{\beta} g_{\le j-2} D^{s,\beta} f_j \Bigr) + \operatorname{error},
\label{a20_b12}
\end{align}
where ``error'' term satisfies the same bound as in \eqref{a20_b10}.

\end{proof}

We are now ready to complete the proof of Theorem \ref{thm2} for the case $1<p<\infty$.

\begin{proof}[Proof of Theorem \ref{thm2}, case $1<p<\infty$]

By Proposition \ref{propforthm2}, we just need to
 simplify the expression in \eqref{a20_b12}. For this we have to discuss several
cases.

Case a): $1<p_1<\infty$ and $1<p_2<\infty$.

Note that
\begin{align*}
\sum_j f_{\le j-2} D^s g_j & = f D^s g - \sum_j f_{>j-2} D^s g_j \notag \\
& = f D^s g - \sum_j f_j D^s g_{<j+2} \notag \\
& = f D^s g - \sum_j f_j D^s g_{j-2\le \cdot <j+2} - \sum_j f_j D^s g_{<j-2}.
\end{align*}

Clearly
\begin{align*}
\| \sum_j f_j D^s g_{j-2\le \cdot <j+2} \|_p
& \lesssim \| (2^{js_1} f_j)_{l_j^2} (2^{-js_1} D^s g_{j-2\le \cdot <j+2} )_{l_j^2} \|_p \notag \\
& \lesssim \| D^{s_1} f \|_{p_1} \cdot \| D^{s_2} g\|_{p_2};
\end{align*}
and by frequency localization,
\begin{align*}
\| \sum_j f_j D^s g_{<j-2} \|_p & \lesssim  \| (2^{js_1} f_j )_{l_j^2} \cdot (2^{-js_1} D^s g_{<j-2})_{l_j^{\infty}}
\|_p \notag \\
& \lesssim \| D^{s_1} f \|_{p_1} \cdot \| D^{s_2} g \|_{p_2}.
\end{align*}

Similarly for each $0<|\alpha| \le [s_1]$, $0\le |\beta| \le [s_2]$, it holds that
\begin{align*}
\| \sum_j \partial^{\alpha} f_{>j-2} D^{s,\alpha} g_j \|_p
& = \| \sum_j \partial^{\alpha} f_j D^{s,\alpha} g_{<j+2} \|_p  \lesssim \| D^{s_1} f\|_{p_1} \| D^{s_2} g\|_{p_2};\\
\| \sum_j \partial^{\beta} g_{>j-2} D^{s,\beta} f_j \|_p & = \| \sum_j \partial^{\beta} g_j D^{s,\beta}
f_{<j+2} \| \notag \lesssim \| D^{s_1} f \|_{p_1} \cdot\| D^{s_2} g \|_{p_2}.
\end{align*}

Thus for $1<p_1,p_2<\infty$, $s_1 \ge 0$, $s_2\ge 0$, $s_1+s_2=s$,
\begin{align*}
 \| D^s(fg) - \sum_{|\alpha|\le s_1} \frac 1{\alpha!} \partial^{\alpha}f
 D^{s,\alpha} g -\sum_{|\beta|\le s_2}
 \frac 1 {\beta!} \partial^{\beta} g D^{s,\beta} f
 \|_p \lesssim \| D^{s_1} f\|_{p_1} \| D^{s_2} g\|_{p_2}.
 \end{align*}

Case 2): $p_1=p$, $p_2=\infty$.
If $|\alpha|=s_1$ (in this case $s_1$ has to be an integer), then
\begin{align*}
\| \sum_j \partial^{\alpha} f_{\le j-2} D^{s,\alpha} g_j \|_p
\lesssim \|D^{s_1} f \|_p \| D^{s_2} g \|_{\operatorname{BMO}}.
\end{align*}

If $|\alpha|<s_1$, then rewrite
\begin{align*}
\sum_j \partial^{\alpha} f_{>j-2} D^{s,\alpha} g_j
& = \sum_j \partial^{\alpha} f_j D^{s,\alpha} g_{<j+2} \notag \\
& =\sum_j \partial^{\alpha} f_j D^{s,\alpha} g_{<j-2} + \sum_{j} \partial^{\alpha} f_j
D^{s,\alpha} g_{j-2\le \cdot <j+2}.
\end{align*}
Clearly then
\begin{align*}
 & \| \sum_j \partial^{\alpha} f_{>j-2} D^{s,\alpha} g_j \|_p \notag \\
 \lesssim&\; \| (\partial^{\alpha} f_j \cdot 2^{j(s_1-|\alpha|)} )_{l_j^2}
 \cdot ( 2^{-j(s_1-|\alpha|)} D^{s,\alpha} g_{<j-2} )_{l_j^{\infty} } \|_p
 + \| D^{s_1} f \|_p \cdot \| D^{s_2} g \|_{\operatorname{BMO}} \notag \\
 \lesssim&\; \|D^{s_1} f \|_p \|D^{s_2} g \|_{\dot B^0_{\infty,\infty}}
 + \|D^{s_1} f \|_p \| D^{s_2} g \|_{\operatorname{BMO}} \notag \\
 \lesssim & \| D^{s_1} f \|_p \| D^{s_2} g \|_{\operatorname{BMO}}.
 \end{align*}

 On the other hand, for each $ |\beta|=s_2$, we have
 \begin{align*}
 \| \sum_j \partial^{\beta} g_{>j-2} D^{s,\beta} f_j \|_p &= \| \sum_j \partial^{\beta} g_j
 D^{s,\beta} f_{<j+2} \|_p \notag \\
 & \lesssim \| \partial^\beta g \|_{\bmo} \| D^{s,\beta} f \|_p \notag \\
 & \lesssim \| D^{s_1} f \|_p \cdot \| D^{s_2} g \|_{\operatorname{BMO}}.
 \end{align*}
Similarly for each $0\le |\beta|<s_2$, we have
\begin{align*}
 \| \sum_j &\partial^{\beta} g_{>j-2} D^{s,\beta} f_j \|_p  \notag \\
 &= \| \sum_j \partial^{\beta} g_j
 D^{s,\beta} f_{<j+2} \|_p \notag \\
 & \lesssim \| \sum_j \partial^{\beta} g_j D^{s,\beta} f_{<j-2} \|_p
 + \| \sum_j \partial^{\beta} g_j D^{s,\beta} f_{j-2\le \cdot <j+2} \|_p \\
 & \lesssim  \| (\partial^{\beta} g_j D^{s,\beta} f_{<j-2} )_{l_j^2} \|_p+
 \| D^{s_1} f \|_p \cdot \| D^{s_2} g \|_{\operatorname{BMO}} \notag \\
 & \lesssim \| (2^{(s_2-|\beta|)j} \partial^{\beta} g_j)_{l_j^{\infty}} (2^{-(s_2-|\beta|)j}
 D^{s,\beta} f_{<j-2} )_{l_j^2} \|_p +
 \| D^{s_1} f \|_p \cdot \| D^{s_2} g \|_{\operatorname{BMO}} \notag \\
 & \lesssim \| D^{s_2} g\|_{\dot B^0_{\infty,\infty}} \| D^{s_1} f \|_p +
 \| D^{s_1} f \|_p \cdot \| D^{s_2} g \|_{\operatorname{BMO}} \notag \\
 & \lesssim
 \| D^{s_1} f \|_p \cdot \| D^{s_2} g \|_{\operatorname{BMO}}.
 \end{align*}

 Collecting the estimates, we have the following:

 If $s_1$ is not an integer, then
 \begin{align*}
 \| D^s(fg) -\sum_{|\alpha|\le [s_1]} \frac 1 {\alpha !}
 \partial^{\alpha} f D^{s,\alpha} g
 -\sum_{|\beta|\le [s_2]} \frac 1 {\beta!}
 \partial^{\beta} g D^{s,\beta} f \|_p \lesssim \| D^{s_1} f\|_p \| D^{s_2} g\|_{\operatorname{BMO}}.
 \end{align*}
 If $s_1\ge 0$ is an integer, then
 \begin{align*}
 \| D^s (fg) -\sum_{|\alpha|<s_1} \frac 1 {\alpha !} \partial^{\alpha} f D^{s,\alpha} g
 -\sum_{|\beta| \le [s_2]} \frac 1 {\beta!} \partial^{\beta} g D^{s,\beta} f \|_p
 \lesssim \| D^{s_1} f \|_p \| D^{s_2} g\|_{\operatorname{BMO}}.
 \end{align*}
Then clearly for all $s_1\ge 0$, we have
\begin{align*}
 \| D^s (fg) -\sum_{|\alpha|<s_1} \frac 1 {\alpha !} \partial^{\alpha} f D^{s,\alpha} g
 -\sum_{|\beta| \le s_2} \frac 1 {\beta!} \partial^{\beta} g D^{s,\beta} f \|_p
 \lesssim \| D^{s_1} f \|_p \| D^{s_2} g\|_{\operatorname{BMO}}.
 \end{align*}

Case 3): $p_1=\infty$, $p_2=p$.
This is similar to Case 2 (with $f$ and $g$ swapped).  Clearly
\begin{align*}
 \| D^s (fg) -\sum_{|\alpha|\le s_1} \frac 1 {\alpha !} \partial^{\alpha} f D^{s,\alpha} g
 -\sum_{|\beta| <s_2} \frac 1 {\beta!} \partial^{\beta} g D^{s,\beta} f \|_p
 \lesssim \| D^{s_1} f \|_{\operatorname{BMO}} \| D^{s_2} g\|_{p}.
 \end{align*}

\end{proof}

\subsection{Proof of Theorem \ref{thm2}, case $\frac 12 <p\le 1$}
Here we shall use the condition $s>\frac dp -d $ or $s \in 2 \mathbb N$. A close inspection of
the proof for the case $1<p<\infty$ shows that we only need to modify the estimate for the
diagonal piece $\| \sum_j D^s (f_j \tilde g_j) \|_p$. For the low-high and high-low pieces, the
estimate works for the whole range $\frac 12<p<\infty$, $s>0$ (note  that when
$\frac 12<p\le 1$ we require $1<p_1,p_2<\infty$). The constraint $s>\frac dp-d$
or $s\in 2\mathbb N$ for $\frac 12 <p\le 1$ is only needed for the diagonal piece. To deal with this
situation, we shall use the approach in Grafakos-Oh \cite{GO_CPDE} and write
\begin{align*}
&\Bigl( D^s(f_j \tilde g_j ) \Bigr) (x)   \notag \\
=&
\frac 1 {   (2\pi)^{2d} } 
\int 2^{js} \cdot  |2^{-j}(\xi+\eta)|^s  \chi(2^{-j} \xi) \chi(2^{-j} \eta) 
\widehat{f_j}(\xi) \widehat{\tilde g_j} (\eta) e^{i(\xi+\eta) \cdot x} d\xi d\eta,
\end{align*}
where $\chi$ is a smooth cut-off function with support in $\{ \xi: 2^{-m_0} <|\xi| <2^{m_0} \}$ for
some integer $m_0\ge 1$.  Let $ \chi_1 \in C_c^{\infty}( \mathbb R^d ) $ be such that $\chi_1(z)=1$
for $|z| \le 2^{m_0+1} $ and $\chi_1(z)=0$ for $|z|\ge 2^{m_0+1.5}$. By using Fourier series, easy to check that for $|z| \le 2^{m_0+2} $,
\begin{align*}
|z|^s \chi_1 (z) = \sum_{m \in \mathbb Z^d} C_m^s e^{  \frac {2\pi i z\cdot m} L },
\end{align*}
where $L=2^{m_0+2}$ and $C_m^s= O((1+|m|)^{-d-s})$. If $s\in 2\mathbb N$, then
$C_m^s = O((1+|m|)^{-C})$ for any $C>0$. We then have
\begin{align*}
\sum_j D^s (f_j \tilde g_j) = \sum_{m\in \mathbb Z^d} C_m^s \sum_j 2^{js} P_j^m f_j P_j^m \tilde g_j,
\end{align*}
where $\widehat{P_j^m f}(\xi) =\phi_1(2^{-j} \xi) \hat f(\xi)$, and $\phi_1(\xi)= \chi(\xi) e^{2\pi i m\cdot \xi/L}$.
It follows that
\begin{align*}
\| \sum_j D^s (f_j \tilde g_j) \|_p^p
& \lesssim  \sum_{m \in \mathbb Z^d} |C_m^s|^p\; \| \sum_j 2^{js} P_j^m f_j P_j^m \tilde g_j \|_p^p 
\notag \\
& \lesssim \sum_{m\in \mathbb Z^d} |C_m^s|^p  \; (\| (2^{js_1} P_j^m f_j )_{l_j^2} \|_{p_1}
\| (2^{js_2} P_j^m \tilde g_j)_{l_j^2} \|_{p_2} )^p \notag \\
& \lesssim (\| D^{s_1} f \|_{p_1} \| D^{s_2} g\|_{p_2} )^p,
\end{align*}
where in the last inequality above, we have used $p(d+s)>d$ (for $s\in 2\mathbb N$ we just use
the fast decay of $C_m^s$) and the fact that the operator norm of $P_j^m$ on 
$L^r(\mathbb R^d, l^2)$ ($1<r<\infty$)
is bounded by $C_{r,d}\cdot \log (10+|m|)$ ($C_{r,d}$ depends only on $r$ and $d$). 

\section{Refined Kato-Ponce inequalities}

\begin{thm} \label{thm3a}
Let $s>0$, $1<p<\infty$, $1<p_1,p_2,p_3,p_4\le \infty$ and
$\frac 1{p_1}+\frac 1{p_2}=\frac 1{p_3}+\frac 1{p_4}=\frac 1p$.
Then the following hold for any $f, g \in \mathcal S(\mathbb R^d)$:

\begin{itemize}
\item If $0<s<1$, then
\begin{align*}
&\| D^s(fg) -f D^s g - g D^s f \|_p \lesssim
\begin{cases}
\| D^s f \|_{p_1} \| g\|_{p_2}, \quad \text{if $1<p_1,p_2<\infty$}; \\
\| D^s f \|_p \| g \|_{\bmo}, \quad \text{if $p_1=p, p_2=\infty$}.
\end{cases}\\
&\| D^s(fg) -f D^s g \|_p \lesssim \| D^s f \|_{\bmo} \| g\|_p, \quad \text{if $p_1=\infty$, $p_2=p$}.
\end{align*}

\item If $s=1$, then
\begin{align*}
&\| D(fg)-f Dg- gDf + \partial f \cdot D^{-1} \partial g \|_p
\lesssim \| D f \|_{p_1} \| g \|_{p_2}, \quad \text{if $1<p_1,p_2<\infty$}; \\
& \| D(fg) - f Dg - g Df \|_p \lesssim \| Df \|_p \| g \|_{\bmo}, \quad \text{if $p_1=p,\,p_2=\infty$};\\
& \| D(fg)- f Dg + \partial f \cdot D^{-1} \partial g \|_p \lesssim
\| Df \|_{\bmo} \| g\|_p, \quad \text{if $p_1=\infty$, $p_2=p$}.
\end{align*}

\item If $1<s<2$, then
\begin{align*}
& \| D^s(fg) -f D^s g -gD^s f + s \partial f \cdot D^{s-2} \partial g \|_p
\lesssim \| D^s f \|_{p_1} \| g\|_{p_2}, \quad \notag \\
&\qquad\qquad\qquad \text{if $1<p_1,p_2<\infty$}; \\
& \| D^s(fg) -f D^s g - g D^s f +s \partial f \cdot D^{s-2} \partial g \|_p
\lesssim \| D^s f \|_p \|g\|_{\bmo}, \quad \notag \\
&\qquad\qquad\qquad\text{if $p_1=p$, $p_2=\infty$}; \\
& \| D^s(fg) -f D^s g  + s \partial f \cdot D^{s-2} \partial g \|_p \lesssim
\| D^s f \|_{\bmo} \|g\|_p, \quad \text{if $p_1=\infty$, $p_2=p$}.
\end{align*}

\item If $s\ge 2$ and $1<p_1<\infty$, then
\begin{align*}
\| D^s(fg) - fD^s g -g D^s f +s \partial f \cdot D^{s-2} \partial g \|_p
\lesssim A+B,
\end{align*}
where
\begin{align*}
&A= \begin{cases}
\| D^s f\|_{p_1} \|g\|_{p_2}, \quad \text{if $1<p_1,p_2<\infty$}; \\
\|D^s f \|_p \| g\|_{\bmo}, \quad \text{if $p_1=p$, $p_2=\infty$};
\end{cases}\\
&B=\begin{cases}
\| \partial f \|_{p_3} \cdot \|D^{s-1} g \|_{p_4}, \quad \text{if $1<p_3,p_4<\infty$}; \\
\| \partial f \|_{\dot B^0_{\infty,\infty}} \| D^{s-1} g \|_p, \quad \text{if $p_3=\infty$, $p_4=p$}; \\
\| \partial f \|_p \| D^{s-1} g \|_{\dot B^0_{\infty,\infty}}, \quad \text{if $p_3=p$, $p_4=\infty$}.
\end{cases}
\end{align*}

\item If $s\ge 2$ and $p_1=\infty$, $p_2=p$, then
\begin{align*}
\| D^s(fg) - fD^s g  +s \partial f \cdot D^{s-2} \partial g \|_p
\lesssim \| D^s f \|_{\bmo} \|g\|_p
+B,
\end{align*}
where $B$ is defined the same as above.

\end{itemize}

\end{thm}

\begin{proof}[Proof of Theorem \ref{thm3a}]
We have to discuss several cases. The following discussions are a little bit tedious and duly long.

$\bullet$ The case $0<s<1$.

Taking $s_1=s$, $s_2=0$ in \eqref{a5}, we get
\begin{align*}
\| D^s(fg) - f D^s g - g D^s f \|_p \lesssim \| D^s f \|_{p_1} \| g \|_{p_2}, \qquad \text{if $1<p_1,p_2<\infty$}.
\end{align*}
Similarly from \eqref{a5c}--\eqref{a5e}, we get
\begin{align*}
&\| D^s (fg) -f D^s g - g D^s f \|_p \lesssim \| D^s f \|_p \| g \|_{\bmo}, \qquad \text{if $p_1=p$, $p_2=\infty$},\\
&\| D^s (fg) - f D^s g \|_p \lesssim \| D^s f \|_{\bmo} \| g \|_p, \qquad \text{if $p_1=\infty$, $p_2=p$}.
\end{align*}

\medskip

$\bullet$ The case $1\le s<2$.

\underline{Subcase $1<p_1,p_2<\infty$}:

Again taking $s_1=s$, $s_2=0$ in the inequalities \eqref{a5}, we get
\begin{align*}
&\| D^s (fg) - f D^s g - \sum_{|\alpha|=1} \partial^{\alpha} f D^{s,\alpha} g - g D^s f \|_p
\lesssim \| D^s f \|_{p_1} \| g \|_{p_2}, \quad \notag \\
&\qquad\qquad\qquad\qquad \text{if $1<p_1,p_2<\infty$}.
\end{align*}

Note that (recall $\partial=(\partial_1,\cdots, \partial_d)$)
\begin{align*}
\sum_{|\alpha|=1} \partial^{\alpha} f D^{s,\alpha} g = -s \partial f \cdot D^{s-2} \partial g.
\end{align*}

Thus for $1\le s<2$ and $1<p_1,p_2<\infty$, we get
\begin{align*}
\| D^s(fg) -f D^s g - gD^s f + s \partial f \cdot D^{s-2} \partial g \|_p \lesssim \| D^s f \|_{p_1}
\| g \|_{p_2}.
\end{align*}

\underline{Subcase $p_1=p$, $p_2=\infty$}:

Consider first $s=1$.
By using \eqref{a5c}, we have
\begin{align*}
& \| D^s (fg) - f D^s g - g D^s f \|_p \lesssim \| D^s f \|_p  \| g\|_{\bmo}.
\end{align*}
Next if $1<s<2$, then by \eqref{a5c} (note that the terms $|\alpha|=1$ are now included),
\begin{align*}
\| D^s (fg) -f D^s g -g D^s f + s \partial f \cdot D^{s-2} \partial g \|_p
\lesssim \| D^s f \|_p \|g\|_{\bmo}.
\end{align*}

\underline{Subcase $p_1=\infty$, $p_2=p$}:

By using \eqref{a5e}, obviously we have
\begin{align*}
& \| D^s (fg)- f D^s g + s \partial f \cdot D^{s-2} \partial g \|_p \lesssim
\| D^s f \|_{\bmo} \| g \|_p.
\end{align*}

$\bullet$ The case $s=2$.

In this case since $D^2=-\Delta$, we can directly use the formula
\begin{align*}
\Delta (fg) - f \Delta g -g \Delta f = 2 \partial f \cdot \partial g.
\end{align*}
Thus
\begin{align*}
D^s(fg) - fD^s g - gD^s f + s \partial f \cdot D^{s-2} \partial g =0
\end{align*}
and no estimate is needed.

$\bullet$ The case $s>2$.

\underline{Subcase 1}: $1<p_1,p_2<\infty$.

By \eqref{a5}, we have
\begin{align*}
 & \| D^s(fg) -f D^s g - g D^s f + s\partial f \cdot D^{s-2} \partial g \|_p \notag \\
 \lesssim & \| D^s f \|_{p_1} \| g\|_{p_2} +
 \sum_{2\le |\alpha| \le s} \| \partial^{\alpha} f \cdot D^{s,\alpha} g \|_p.
 \end{align*}

 Then for each $2\le |\alpha|<s$, by Lemma \ref{lemc2},
 \begin{align*}
  \| D^{|\alpha|} f
 \|_{(\frac 1 {p_1} \cdot \frac{|\alpha|-1} {s-1}
 + \frac 1 {p_3} \cdot \frac{s-|\alpha|}{s-1} )^{-1} }
 \lesssim
 \begin{cases}
 \| D^s f \|_{p_1}^{\frac{|\alpha|-1}{s-1}} \| D f \|_{p_3}^{\frac{s-|\alpha|}{s-1}},
 \quad \text{if $1<p_3<\infty$}; \\
 \| D^s f\|_{p_1}^{\frac{|\alpha|-1} {s-1}} \| D f\|_{\dot B^0_{\infty,\infty}}^{\frac{s-|\alpha|}{s-1}},
 \quad \text{if $p_3=\infty$};
 \end{cases}
 \\
 \| D^{s-|\alpha|} g \|_{( \frac 1 {p_2} \cdot \frac{|\alpha|-1}{s-1} + \frac 1 {p_4}
\cdot \frac {s-|\alpha|}{s-1} )^{-1} }
\lesssim
\begin{cases}
\| g \|_{p_2}^{\frac{|\alpha|-1}{s-1}} \| D^{s-1} g \|_{p_4}^{\frac{s-|\alpha|}{s-1}},
\qquad \text{if $p_4<\infty$}; \\
\| g \|_{p_2}^{\frac{|\alpha|-1}{s-1}}  \| D^{s-1} g \|_{\dot B^0_{\infty,\infty}}^{\frac{s-|\alpha|}{s-1}},
\quad \text{if $p_4=\infty$}.
\end{cases}
\end{align*}

Note that $p\le (\frac 1{p_1} \cdot \frac{|\alpha|-1}{s-1} + \frac 1 {p_3} \cdot \frac{s-|\alpha|}{s-1} )^{-1}<\infty$
and $p\le (\frac 1 {p_2} \cdot \frac{|\alpha|-1}{s-1} +\frac 1 {p_4} \cdot \frac{s-|\alpha|}{s-1} )^{-1}<\infty$. Clearly
\begin{align*}
&\| \partial^{\alpha} f \|_{(\frac 1{p_1} \cdot \frac{|\alpha|-1}{s-1}
+ \frac 1 {p_3} \cdot \frac{s-|\alpha|}{s-1} )^{-1}}
\lesssim
\| D^{|\alpha|} f \|_{(\frac 1{p_1} \cdot \frac{|\alpha|-1}{s-1} + \frac 1 {p_3} \cdot \frac{s-|\alpha|}{s-1} )^{-1}};\\
& \| D^{s,\alpha} g \|_{(\frac 1 {p_2} \cdot \frac{|\alpha|-1}{s-1} +\frac 1 {p_4} \cdot \frac{s-|\alpha|}{s-1} )^{-1}}
\lesssim \| D^{s-|\alpha|} g \|_{
(\frac 1 {p_2} \cdot \frac{|\alpha|-1}{s-1} +\frac 1 {p_4} \cdot \frac{s-|\alpha|}{s-1} )^{-1}}.
\end{align*}

If $|\alpha|=s$, then obviously
\begin{align*}
\sum_{|\alpha|=s} \| \partial^{\alpha} f \cdot D^{s,\alpha} g \|_p
\lesssim \| D^s f \|_{p_1} \| g \|_{p_2}.
\end{align*}

Thus
\begin{align*}
&\sum_{2\le |\alpha|\le s}
\| \partial^{\alpha} f \cdot D^{s,\alpha} g \|_p \notag \\
\lesssim &
\| D^s f \|_{p_1} \| g\|_{p_2}
+
\begin{cases}
\| \partial f \|_{p_3} \| D^{s-1} g \|_{p_4}, \quad
\text{if $1<p_3,p_4<\infty$}; \\
\| \partial f \|_{\dot B^0_{\infty,\infty}}
\| D^{s-1} g \|_p, \quad \text{if $p_3=\infty$, $p_4=p$}; \\
\| \partial f \|_p \| D^{s-1} g \|_{\dot B^0_{\infty,\infty}},
\quad \text{if $p_3=p$, $p_4=\infty$}.
\end{cases}
\end{align*}

\underline{Subcase 2}: $p_1=\infty$ and $p_2=p$. By \eqref{a5e},
\begin{align*}
&\| D^s (fg) - f D^s g +s \partial f \cdot D^{s-2} \partial g \|_p \notag \\
\lesssim & \| D^s f \|_{\bmo} \| g \|_p
+ \sum_{2\le |\alpha|\le s} \| \partial^{\alpha} f \cdot D^{s,\alpha} g \|_p.
\end{align*}

For each $2\le |\alpha|<s$, by Lemma \ref{lemc2},
\begin{align*}
& \| D^{|\alpha|} f \|_{(\frac 1 {p_3} \cdot \frac{s-|\alpha|}{s-1} )^{-1} }
\lesssim \| D^s f \|_{\dot B^0_{\infty,\infty}}^{\frac{|\alpha|-1}{s-1}}
\| D f \|_{p_3}^{\frac{s-|\alpha|}{s-1}}, \qquad \text{if $1<p_3<\infty$}; \\
& \| D^{|\alpha|} f \|_{\dot B^0_{\infty,1}}
\lesssim \| D^s f \|_{\dot B^0_{\infty,\infty}}^{\frac{|\alpha|-1}{s-1}}
\| Df \|_{\dot B^0_{\infty,\infty}}^{\frac{s-|\alpha|}{s-1}}, \qquad \text{if $p_3=\infty$}; \\
& \| D^{s-|\alpha|} g \|_{(\frac 1 p \cdot \frac{|\alpha|-1}{s-1} + \frac 1 {p_4}
\cdot \frac{s-|\alpha|}{s-1} )^{-1}}
\lesssim \| g\|_p^{\frac{|\alpha|-1}{s-1}} \| D^{s-1} g \|_{p_4}^{\frac{s-|\alpha|}{s-1}},
\qquad \text{if $1<p_4<\infty$}; \\
& \| D^{s-|\alpha|} g \|_{(\frac 1p \cdot \frac{|\alpha|-1}{s-1})^{-1}}
\lesssim \| g \|_p^{\frac{|\alpha|-1}{s-1}} \| D^{s-1} g \|_{\dot B^0_{\infty,\infty}}^{\frac{s-|\alpha|}{s-1}},
\qquad \text{if $p_4=\infty$}.
\end{align*}
The second inequality above can be easily proved by splitting into low and high frequencies.

For $|\alpha|=s$ (in this case $s$ is an integer), by Lemma \ref{lemc3}, we have
\begin{align*}
&\sum_{|\alpha|=s}
\| \partial^{\alpha} f \cdot D^{s,\alpha} g \|_p \notag \\
\lesssim & \| D^s f \|_{\bmo}
\| g \|_p
+ \begin{cases}
\| \partial f \|_{p_3} \| D^{s-1} g \|_{p_4}, \quad \text{if $1<p_3,p_4<\infty$};\\
\| \partial f \|_{\dot B^0_{\infty,\infty}} \| D^{s-1} g \|_p,
\quad \text{if $p_3=\infty$, $p_4=p$}; \\
\| \partial f \|_p \| D^{s-1} g \|_{\dot B^0_{\infty,\infty}},
\quad \text{if $p_3=p$, $p_4=\infty$}.
\end{cases}
\end{align*}

Thus
\begin{align*}
&\sum_{2\le |\alpha|\le s} \| \partial^{\alpha} f \cdot D^{s,\alpha} g \|_p \notag \\
\lesssim &\| D^s f\|_{\bmo} \| g \|_p +
 \begin{cases}
\| \partial f \|_{p_3} \| D^{s-1} g \|_{p_4}, \quad \text{if $1<p_3,p_4<\infty$};\\
\| \partial f \|_{\dot B^0_{\infty,\infty}} \| D^{s-1} g \|_p,
\quad \text{if $p_3=\infty$, $p_4=p$}; \\
\| \partial f \|_p \| D^{s-1} g \|_{\dot B^0_{\infty,\infty}},
\quad \text{if $p_3=p$, $p_4=\infty$}.
\end{cases}
\end{align*}

\underline{Subcase 3}: $p_1=p$, $p_2=\infty$. By \eqref{a5c}, we have
\begin{align*}
 & \| D^s (fg) -f D^s g - g D^s f + s\partial f \cdot D^{s-2} \partial g \|_p \notag \\
 \lesssim & \| D^s f \|_p \| g \|_{\bmo}
 + \sum_{2\le |\alpha|<s} \| \partial^{\alpha} f \cdot D^{s,\alpha} g \|_p.
 \end{align*}

 For each $2\le |\alpha<s$, by Lemma \ref{lemc2},
 \begin{align*}
& \| D^{|\alpha|} f \|_{(\frac 1 p \cdot \frac{|\alpha|-1}{s-1}
 + \frac 1 {p_3} \cdot \frac{s-|\alpha|}{s-1} )^{-1}}
 \lesssim
 \begin{cases}
 \| D^s f \|_p^{\frac{|\alpha|-1}{s-1}} \| D f \|_{p_3}^{\frac{s-|\alpha|}{s-1}},
 \quad \text{if $1<p_3<\infty$};\\
 \| D^s f \|_p^{\frac{|\alpha|-1}{s-1}}
 \| D f\|_{\dot B^0_{\infty,\infty}}^{\frac{s-|\alpha|}{s-1}},
 \quad \text{if $p_3=\infty$};
 \end{cases}\\
 &\| D^{s-|\alpha|} g \|_{(\frac 1 {p_4} \cdot \frac{s-|\alpha|} {s-1} )^{-1}}
 \lesssim \| g \|_{\dot B^0_{\infty,\infty}}^{\frac{|\alpha|-1}{s-1}}
 \| D^{s-1} g \|_{p_4}^{\frac{s-|\alpha|}{s-1}}, \quad \text{if $1<p_4<\infty$};\\
 &\| D^{s-|\alpha|} g \|_{\dot B^0_{\infty,1}}
 \lesssim \| g \|_{\dot B^0_{\infty,\infty}}^{\frac{|\alpha|-1}{s-1}}
 \| D^{s-1} g \|_{\dot B^0_{\infty,\infty}}^{\frac{s-|\alpha|}{s-1}},
 \quad \text{if $p_4=\infty$}.
 \end{align*}
Thus
\begin{align*}
&\sum_{2\le |\alpha|<s} \| \partial^{\alpha} f \cdot D^{s,\alpha} g \|_p \notag \\
\lesssim&
\| D^s f \|_p \| g \|_{\bmo}
+\begin{cases}
\| \partial f \|_{p_3} \| D^{s-1} g \|_{p_4}, \quad \text{if $1<p_3,p_4<\infty$};\\
\| \partial f \|_{\dot B^0_{\infty,\infty}} \| D^{s-1} g \|_p,
\quad \text{if $p_3=\infty$, $p_4=p$}; \\
\| \partial f \|_p \| D^{s-1} g \|_{\dot B^0_{\infty,\infty}},
\quad \text{if $p_3=p$, $p_4=\infty$}.
\end{cases}
\end{align*}

\end{proof}

\begin{cor} \label{cor3a}
Let $s>0$ and $1<p<\infty$. Then for any $f,g \in \mathcal S(\mathbb R^d)$,
\begin{align} \label{cm101}
\| D^s (fg) - f D^s g \|_p \lesssim \| D^s f \|_{p_1} \| g \|_{p_2}
+ \| \partial f \|_{p_3} \| D^{s-1} g \|_{p_4},
\end{align}
where $1<p_1,p_2,p_3,p_4\le \infty$, and $\frac 1{p_1}+\frac 1{p_2} =\frac 1 {p_3}
+ \frac 1 {p_4} = \frac 1p$.

For $0<s< 1$ and $1<p<\infty$, the following inequality hold:
\begin{align}
\| D^s(fg) -f D^s g \|_p \lesssim \| D^s f \|_{p_1} \| g\|_{p_2}, \label{cm100}
\end{align}
where $\frac 1 {p_1}+\frac 1 {p_2} =\frac 1p$, and $1<p_1,p_2\le \infty$.

For $s=1$ and $1<p<\infty$,
\begin{align}
\| D(fg) - f Dg \|_p \lesssim \| \partial f \|_{p_1} \| g\|_{p_2}, \label{cm99}
\end{align}
where $\frac 1 {p_1}+\frac 1 {p_2} =\frac 1 p$, and $1<p_1,p_2\le \infty$.

For $p_1=\infty$, $p_2=p$ and $1<p<\infty$, we have the following BMO-refinements:
\begin{itemize}
\item If $0<s<1$, then
\begin{align}
\| D^s(fg) -f D^s g \|_p \lesssim \| D^s f \|_{\bmo} \| g \|_p. \label{cm98}
\end{align}
\item If $s\ge 1$, then
\begin{align}
\| D^s(fg) -f D^s g \|_p \lesssim \| \partial f \|_{p_3} \| D^{s-1} g \|_{p_4} +
\| D^s f\|_{\bmo} \|g\|_p, \label{cm97}
\end{align}
where $\frac 1{p_3}+\frac 1{p_4} = \frac 1p$, $1< p_3\le \infty$, $1<p_4<\infty$.
\item Also for $s\ge 1$,
\begin{align}
\| D^s (fg) -f D^s g \|_p \lesssim \| \partial f \|_p \| D^{s-1} g \|_{\bmo}
+ \| D^s f \|_{\bmo} \| g \|_p. \label{cm96}
\end{align}
\end{itemize}

\end{cor}

\begin{proof}[Proof of Corollary \ref{cor3a}]
We shall follow the same order as in the statement of Theorem \ref{thm3a} and discuss the regimes
$0<s<1$, $s=1$ and $s>1$ respectively.

$\bullet$ The case $0<s<1$.

Clearly by Theorem \ref{thm3a},
\begin{align*}
\| D^s(fg) - f D^s g \|_p \lesssim
\begin{cases}
\| D^s f\|_{p_1} \| g\|_{p_2}, \quad \text{if $1<p_1< \infty$, $1<p_2\le \infty$}, \\
\| D^s f\|_{\bmo} \| g\|_p, \quad \text{if $p_1=\infty$, $p_2=p$}.
\end{cases}
\end{align*}
Thus \eqref{cm100} and \eqref{cm98} hold.

$\bullet$ The case $s=1$.

If $1<p_1<\infty$, $1<p_2< \infty$, then
\begin{align*}
\| D(fg) -f Dg \|_p & \lesssim \| Df \|_{p_1} \|g\|_{p_2} + \| \partial f \cdot D^{-1} \partial g \|_p \notag \\
& \lesssim \| \partial f \|_{p_1} \| g\|_{p_2}.
\end{align*}
If $p_1=p$, $p_2=\infty$, then recall
\begin{align*}
\| D(fg) -f Dg -g Df \|_p \lesssim \| Df \|_p \| g\|_{\operatorname{BMO}}.
\end{align*}
Thus
\begin{align*}
\| D (fg) - f Dg \|_p \lesssim \| \partial f \|_p \| g \|_{\infty}.
\end{align*}

So \eqref{cm99} holds except the case $p_1=\infty$.

If $p_1=\infty$, $p_2=p$, then
\begin{align*}
\| D(fg) - f Dg \|_p & \lesssim \| Df \|_{\bmo} \| g \|_p + \| \partial f \cdot D^{-1} \partial g \|_p \notag
\end{align*}
Clearly if $1<p_3\le \infty$, $1<p_4<\infty$ with $\frac 1{p_3} + \frac 1 {p_4}=\frac 1p$, then
\begin{align*}
\| \partial f \cdot D^{-1} \partial g \|_p \lesssim \| \partial f \|_{p_3} \| g\|_{p_4}.
\end{align*}
Thus \eqref{cm97} holds for $s=1$, and also \eqref{cm99} holds.

On the other hand if $p_3=p$, $p_4=\infty$, then
\begin{align*}
\| \partial f \cdot D^{-1} \partial g \|_p & \lesssim \| \sum_j \partial f_{\le j+2} \cdot D^{-1} \partial g_j \|_p
+ \|\sum_j \partial f_j \cdot D^{-1} \partial g_{<j-2} \|_p \notag \\
& \lesssim \| \partial f \|_p \| g \|_{\bmo} + \| Df \|_{\bmo} \| g\|_p.
\end{align*}
Clearly \eqref{cm96} holds for $s=1$.

$\bullet$ The case $s>1$.

Consider first $1<p_1,p_2<\infty$. By Theorem \ref{thm3a} and Lemma \ref{lemc4}, we have
\begin{align*}
\| D^s(fg) -f D^s g \|_p & \lesssim \| D^s f\|_{p_1} \|g\|_{p_2}
+ \| \partial f \cdot D^{s-2} \partial g \|_p+\|\partial f \|_{p_3} \| D^{s-1} g\|_{p_4} \notag \\
& \lesssim \| D^s f \|_{p_1} \| g \|_{p_2}
+ \| \partial f \|_{p_3} \| D^{s-1} g\|_{p_4},
\end{align*}
for any $1<p_3,p_4\le \infty$. Thus \eqref{cm101} holds. Similarly one can easily check that
\eqref{cm97} and \eqref{cm96} holds.

\end{proof}

\begin{cor} \label{cor3b}
Let $1<p<\infty$. Let $1<p_1,p_2,p_3,p_4\le \infty$ satisfy $\frac 1 {p_1}+\frac 1{p_2}=
\frac 1 {p_3}+\frac 1 {p_4}=\frac 1p$. Then for any $f, g \in \mathcal S(\mathbb R^d)$,
the following hold:

\begin{itemize}
\item If $0<s\le 1$, then
\begin{align*}
\| D^s (fg) - f D^s g \|_p \lesssim \| D^{s-1} \partial f \|_{p_1} \| g\|_{p_2}.
\end{align*}

\item If $s>1$, then
\begin{align*}
\| D^s (fg) - f D^s g \|_p \lesssim \| D^{s-1} \partial f \|_{p_1} \| g\|_{p_2}
+ \| \partial f \|_{p_3} \| D^{s-1} g \|_{p_4}.
\end{align*}
\end{itemize}
\end{cor}

\begin{proof}
This directly follows from Corollary \ref{cor3a} and the identity $D^s f= -D^{-1} \partial \cdot D^{s-1} \partial f$. The only difference is  for the
case $p_1=\infty$, $p_2=p$. In that case since we have BMO-refinements the inequality
is obvious.
\end{proof}

For $s>0$, let $A^s$ be a differential operator such that its symbol $\widehat{A^s}(\xi)$ is a homogeneous
function of degree $s$ and $\widehat{A^s}(\xi) \in C^{\infty}(\mathbb S^{d-1})$. Then the following
corollary can be proved in much the same way as  for the operator $D^s$.

\begin{cor} \label{cor3c}
Let $s>0$ and $1<p<\infty$. Then the following hold for any $f$, $g\in \mathcal S(\mathbb R^d)$:
\begin{itemize}
\item If $0<s\le 1$, then
\begin{align*}
\| A^s(fg) -f A^s g - g A^s f\|_p \lesssim_{A^s,s,p,d} \|D^s f \|_p \| g\|_{\bmo}.
\end{align*}
\item If $s>1$, then
\begin{align}
 & \| A^s(fg) -f A^s g - g A^s f - \partial f \cdot A^{s,\partial} g \|_p
 \notag \\
 \lesssim_{A^s,s,p,d}& \; \| D^s f\|_p \|g\|_{\bmo} +
 \| \partial f \|_{\dot B^0_{\infty,\infty}} \| D^{s-1} g\|_p, \label{cor3c_e1}
 \end{align}
 where
 \begin{align*}
 \widehat{A^{s,\partial}}(\xi) = \frac 1 {i} \partial_{\xi} ( \widehat{A^s}(\xi)).
 \end{align*}
In fact a stronger inequality holds for $s>1$,
\begin{align}
 & \| A^s(fg) -f A^s g - g A^s f - \partial f \cdot A^{s,\partial} g \|_p
 \notag \\
 \lesssim_{A^s,s,p,d}& \; \| D^s f\|_p \|g\|_{\boo} +
 \| \partial f \|_{\dot B^0_{\infty,\infty}} \| D^{s-1} g\|_p, \label{cor3c_e1_e0}
 \end{align}

 \item For all $s>0$,
 \begin{align}
 \| A^s(fg) -f A^s g \|_p
 \lesssim_{A^s,s,p,d} \| D^s f \|_p \| g\|_{\infty}
 + \| \partial f \|_{\infty} \| D^{s-1} g\|_p. \label{cor3c_e1_e1}
 \end{align}
\end{itemize}
 \end{cor}

\begin{proof}[Proof of Corollary \ref{cor3c}]
All the statements except \eqref{cor3c_e1_e0} follow by mimicking the proof for the operator $D^s$.
The argument proceeds by using Corollary \ref{cor2} together with further interpolation inequalities.
We omit the details.

On the other hand, it is possible to give a more ``direct" proof of the above inequalities.
We illustrate this by sketching a proof
for \eqref{cor3c_e1_e1} which is most useful in practice. To this end, we first decompose
\begin{align*}
fg = \sum_{j} f_{<j-2} g_j + \sum_j f_j g_{<j-2} + \sum_j f_j \tilde g_j.
\end{align*}
Then
\begin{align*}
\| [A^s, f_{<j-2} ]g_j \|_p & \lesssim \| ( 2^{js} \mathcal M(\partial f_{<j-2} ) 2^{-j} \mathcal M g_j )_{l_j^2} \|_p
\notag \\
& \lesssim \| D^{s-1} g\|_p \| \partial f \|_{\infty};\\
\|[A^s, g_{<j-2}] f_j \|_p & \lesssim \| (2^{js} \mathcal M (\partial g_{<j-2}) 2^{-j} \mathcal M f_j)_{l_j^2}\|_p
\notag \\
& \lsm \| D^s f\|_p \| g\|_{\boo}, \\
\| A^s (f_j \tilde g_j )\|_p & \lesssim \| (2^{ks} P_k( \sum_{j>k-10} f_j \tilde g_j)_{l_k^2} \|_p \notag \\
& \lesssim \| (2^{js} f_j \tilde g_j )_{l_j^2} \|_p \lsm \| D^s f\|_p \|g\|_{\boo}.
\end{align*}
On the other hand,
\begin{align*}
\sum_j f_{<j-2} A^s g_j = f A^s g - \sum_j f_j A^s g_{\le j+2}, \\
\sum_j g_{<j-2} A^s f_j = g A^s f - \sum_j g_j A^s f_{\le j+2}.
\end{align*}
By Lemma \ref{lemq1}, we have
\begin{align*}
\| \sum_j f_j A^s g_{\le j+2} \|_p \lesssim \| Df \|_{\bmo} \| A^s D^{-1} g\|_p \lesssim \| \partial f\|_{\infty}
\| D^{s-1} g\|_p;\\
\| \sum_j g_j A^s f_{\le j+2} \|_p \lesssim \| g\|_{\bmo} \| A^s f\|_p \lesssim \| g \|_{\infty} \| D^s f\|_p.
\end{align*}
Thus the inequality \eqref{cor3c_e1_e1} follows. A close inspection of the above shows that we actually proved
\begin{align*}
\| A^s (fg) -f A^s g-g A^s f\|_p \lesssim_{s,p,d}\| D^s f\|_p \| g\|_{\bmo} + \| \partial f\|_{\infty}
\| D^{s-1} g\|_p.
\end{align*}

We now show how to prove \eqref{cor3c_e1_e0}.
For this we just need to modify the estimate of
the pieces $\sum_j f_j A^s g_{\le j+2}$, $\sum_j g_j A^s f_{\le j+2}$ and $\sum_j [A^s, f_{<j-2}] g_j$
in the preceding argument.
We first explain how to estimate the first two pieces. In the following computation
we shall avoid
using Lemma \ref{lemq1} since this is where $\bmo$-norm comes in.
Now split
\begin{align*}
\sum_j f_j A^s g_{\le j+2} = \sum_j f_j A^s g_{<j-2} + \sum_j f_j A^s g_{j-2\le \cdot j+2}.
\end{align*}
By frequency localization,
\begin{align*}
\| \sum_j f_j A^s g_{<j-2} \|_p & \lesssim \| (f_j A^s g_{<j-2} )_{l_j^2} \|_p \notag \\
& \lesssim \| (2^{js} f_j)_{l_j^2} \|_p \| (2^{-js} A^s g_{<j-2} )_{l_j^{\infty}} \|_{\infty} \notag \\
& \lesssim \| D^s f\|_p \| g\|_{\boo}.
\end{align*}
Denote $\tilde g_j = g_{j-2\le \cdot \le j+2}$. Then
\begin{align}
\| \sum_j f_j A^s \tilde g_j \|_p  &\lesssim \| (2^{j\frac{s+1}2} f_j)_{l_j^2}
( 2^{-j\frac{s+1}2} A^s \tilde g_j)_{l_j^2} \|_p \notag \\
& \lesssim \| D^{\frac {s+1}2} f\|_{2p} \| D^{\frac{s-1}2} g \|_{2p}. \label{fAsgj_e1}
\end{align}
For $s>0$ and $s\ne 1$, we have\footnote{These interpolation inequalities can be proved in the same way as in
Lemma \ref{lemc2}.}
\begin{align*}
&\| D^{\frac{s+1}2} f \|_{2p} \lsm \| D^s f \|_p^{\frac 12} \| \partial f \|_{\boo}^{\frac 12}, \\
&\| D^{\frac{s-1} 2} g \|_{2p} \lsm \| D^{s-1} g \|_p^{\frac 12} \| g \|_{\boo}^{\frac 12}.
\end{align*}
Thus for $s>0$ and $s\ne 1$,
\begin{align*}
\| \sum_j f_j A^s g_{\le j+2} \|_p \lesssim \| D^s f \|_p \|g\|_{\boo} + \| \partial f \|_{\boo}
\| D^{s-1} g\|_p.
\end{align*}
On the other hand for the piece $\sum_j g_j A^s f_{\le j+2}$ we only need to handle the part
$\sum_j g_j A^s f_{<j-2}$ since the other part is treated the same way as in \eqref{fAsgj_e1}. Now for
$s>1$, by using Lemma \ref{lemc2_a1},
\begin{align*}
\| \sum_j g_j A^s f_{<j-2} \|_p & \lsm \| (g_j A^s f_{<j-2})_{l_j^2} \|_p \notag \\
& \lsm \| (2^{j(s-1)}g_j)_{l_j^2} \|_p \| (2^{-j(s-1)} A^s f_{<j-2})_{l_j^{\infty}} \|_{\infty} \\
& \lsm \| D^{s-1} g\|_p \| \partial f \|_{\boo}.
\end{align*}

Next for the commutator piece $[A^s, f_{<j-2}]g_j$, we denote 
$$K_j = \sum_{a=-10}^{10} A^s P_{j+a} \delta_0,$$
and write
\begin{align}
&([A^s, f_{<j-2}]g_j)(x) \notag \\
=&\; \int K_j(y) \bigl(f_{<j-2}(x-y) -f_{<j-2}(x) -(\partial f_{<j-2})(x) \cdot (-y) \bigr)
g_j(x-y)
dy \label{pf_cor3c_e1_e1_1a} \\
& \quad + \partial f_{<j-2}(x) \int K_j(y) (-y) g_j(x-y) dy. \label{pf_cor3c_e1_e1_1b}
\end{align}
For \eqref{pf_cor3c_e1_e1_1a}, easy to check that
\begin{align*}
\| \sum_j \eqref{pf_cor3c_e1_e1_1a}\|_p & \lesssim \| ( \mathcal M (\partial^2 f_{<j-2}) \cdot 2^{j(s-2)}
\mathcal Mg_j )_{l_j^2} \|_p \notag \\
& \lesssim \| D^{s-1} g\|_p \| \partial f \|_{\boo}.
\end{align*}
On the other hand,
\begin{align*}
\sum_j \eqref{pf_cor3c_e1_e1_1b}
 = \partial f \cdot A^{s,\partial} g -\sum_j \partial f_j \cdot A^{s,\partial}g_{\le j+2}.
\end{align*}
Clearly
\begin{align*}
\| \sum_j \partial f_j \cdot A^{s,\partial} g_{<j-2} \|_p
&\lesssim \| ( 2^{j(s-1)}\partial f_j \cdot 2^{-j(s-1)} A^{s,\partial} g_{<j-2})_{l_j^2} \|_p \notag \\
&\lsm \| D^{s} f\|_p \| g\|_{\boo}.
\end{align*}
The piece $\sum_j \partial f_j \cdot A^{s,\partial} g_{j-2\le \cdot \le j+2}$ can be estimated
in the same way as in \eqref{fAsgj_e1}:
\begin{align*}
\| \sum_j \partial f_j \cdot A^{s,\partial} g_{j-2 \le \cdot \le j+2}\|_p
\lsm \| D^s f\|_p \| g\|_{\boo} + \| \partial f\|_{\boo} \| D^{s-1} g\|_p.
\end{align*}
Thus \eqref{cor3c_e1_e0} holds.

\end{proof}

\begin{rem} \label{rev_rem5.5}
As was already mentioned in the introduction, Corollary \ref{cor3c} can
be used to prove \eqref{FMRR_ineq} and its $L^p$-versions. Indeed set
$A_j^{s+1} = D^s \partial_j$, $f=u_j$, $g=B$, then
\begin{align*}
\| D^s \partial_j (u_j B) - u_j D^s \partial_j B \|_p
\lesssim_{s,p,d} \| D^{s+1} u\|_p \| B\|_{\infty}
+ \| \partial u \|_{\infty} \| D^s B\|_p,
\end{align*}
or upon summing in $j$ (and using $\nabla \cdot u=0$)
\begin{align*}
\| D^s ( (u\cdot \nabla) B) - (u\cdot \nabla) (D^s B) \|_p
\lesssim_{s,p,d}
\| D^{s+1} u\|_p \| B\|_{\infty}
+ \| \partial u \|_{\infty} \| D^s B\|_p.
\end{align*}
Now if $s>d/p$, we can use Sobolev embedding to get
\begin{align}
\| D^s( (u\cdot \nabla) B) - (u\cdot \nabla) D^s B \|_p
\lesssim_{s,p,d}
\|J^s \nabla u\|_p \| J^s B\|_p.
\label{rem_cor3c_e2}
\end{align}
\end{rem}

\begin{rem} \label{rev_rem5.6}
One can construct divergence-free counterexamples to the
estimate \eqref{rem_cor3c_e2} for the borderline case
$s=d/p$. The key is to use the estimate \eqref{cor3c_e1}
(for the operator $A_j^{s+1} = D^s \partial_j$).
Easy to check that (by using $\nabla \cdot u=0$)
\begin{align*}
  & \| D^s \partial_j (u_j B) - u_j D^s \partial_j B
  + \partial u_j \cdot s D^{s-2} \partial  \partial_j B \|_p \notag \\
  \lesssim_{s,p,d}& \; \| D^{s+1} u\|_p \| B\|_{\bmo}
  + \| \partial u\|_{\dot B^0_{\infty,\infty}} \| D^s B\|_p.
  \end{align*}
  Upon summing in $j$ and using $s=d/p$, we get
  \begin{align*}
& \| D^s ((u\cdot \nabla )B) - (u\cdot D^s \nabla)B
  + s \sum_{j=1}^d ( \partial u_j \cdot D^{s-2} \partial)\partial_j B \|_p \notag \\
\lesssim_{s,p,d}&\; \| D^{s+1} u \|_p \| D^s B\|_p.
\end{align*}
Now to finish the construction it suffices for us to exhibit a pair of divergence-free $u$, $B$
with the property
\begin{align*}
\| J^{1+\frac dp} u \|_p + \| J^{\frac dp} B\|_p \le 1,
\end{align*}
such that
\begin{align*}
\| \sum_{j=1}^d (\partial u_j \cdot D^{s-2} \partial) \partial_j B \|_p \gg 1.
\end{align*}
For a construction of such examples, see Remark 1.17 in \cite{BL14}. Alternatively
one can use the idea in Section 9 of this paper.
\end{rem}
\section{Refined Kato-Ponce inequalities for the operator $J^s$}

\begin{lem} \label{lem921_1}
Let $s>0$ and $\tjs=J^s-I$ (i.e. $\widehat{\tjs}(\xi)= (1+|\xi|^2)^{s/2} -1$). For any
$\phi \in C_c^{\infty}(\mathbb R^d)$ with $\operatorname{supp}(\phi) \subset\{ \xi:\; 0<c_1<|\xi|<c_2<\infty\}$,
define $P_j^{\phi}$ as
\begin{align*}
\widehat{P_j^{\phi} f}(\xi) = \phi(\frac{\xi}{2^j}) \hat f(\xi), \quad j \in \mathbb Z.
\end{align*}
Define $K_j = \tjs P_j^{\phi} \delta_0$ where $\delta_0$ is the usual Dirac delta function. Then
for any integer $m\ge 1$, and any $x\in \mathbb R^d$,
\begin{align*}
&|K_j(x)| \lesssim_{m,\phi,d,s} 2^{j(2+d)} (1+2^j |x|)^{-m}, \quad \text{if $j\le 0$}, \\
&|K_j(x)| \lesssim_{m,\phi,d,s} 2^{j(s+d)} (1+2^j |x|)^{-m}, \quad \text{if $j>0$}.
\end{align*}
\end{lem}

\begin{proof}
Consider first $j\le 0$. Clearly
\begin{align*}
K_j(x) &= \frac 1 {(2\pi)^d} \int_{\mathbb R^d} ( (1+|\xi|^2)^{\frac s2} -1 ) \phi(\frac{\xi}{2^j} )
e^{i \xi \cdot x} d\xi \notag \\
& = \frac 1 {(2\pi)^d}\cdot 2^{jd} \int_{\mathbb R^d}
( (1+|2^j \tilde \xi|^2)^{\frac s2} -1) \phi(\tilde \xi) e^{i \tilde \xi \cdot 2^j x} d\tilde \xi.
\end{align*}
If $|2^jx|\le 1$, then since $|(1+|2^j \tilde \xi|^2)^{\frac s2} -1| \lesssim 2^{2j}$ for $|\tilde \xi| \sim 1$, $j\le 0$,
we get
\begin{align*}
|K_j(x)| \lesssim 2^{jd} \cdot 2^{2j} = 2^{j(d+2)}.
\end{align*}
If $|2^j x|>1$, then we can integrate by parts $m$-times and get
\begin{align*}
|K_j(x) | \lesssim 2^{jd} (2^j |x|)^{-m} \cdot 2^{2j}.
\end{align*}
The case for $j>0$ is similar. We omit details.
\end{proof}

\begin{thm} \label{thm921_1}
Let $s>0$ and $1<p<\infty$. Then the following hold for any $f,g\in \mathcal S(\mathbb R^d)$:
\begin{itemize}
\item If $0<s\le 1$, then
\begin{align*}
\| J^s (fg) -f J^s g - g(J^s f -f) \|_p \lesssim_{s,d,p} \| J^{s-1} \partial f \|_p \| g \|_{\bmo}.
\end{align*}

\item If $s>1$, then
\begin{align*}
 \| J^s (fg) &-f J^s g - g(J^s f -f) + s \partial f \cdot J^{s-2} \partial g \|_p \notag \\
\lesssim_{s,p,d} & \; \| J^{s-1} \partial f \|_p \| g\|_{\boo} + \| \partial f \|_{\boo}
\| J^{s-2} \partial g \|_p.
\end{align*}
\end{itemize}
\end{thm}

\begin{proof}
Define $\tjs=J^s-I$. On the Fourier side, we have $\widehat{\tjs}(\xi)= (1+|\xi|^2)^{s/2}-1$. Note that
for $|\xi|\ll 1$, $\widehat{\tjs}(\xi) \sim |\xi|^2$. This property will be useful for controlling
low frequencies in later computations.

Now
\begin{align}
 J^s(fg)- f J^s g & = \tjs(fg) - f\tjs g \notag \\
 & = \sum_j ( \tjs(f_{<j-2} g_j) - f_{<j-2} \tjs g_j ) \label{921_e1} \\
 & \quad + \sum_j ( \tjs (f_j g_{<j-2}) - f_j \tjs g_{<j-2} ) \label{921_e2} \\
 & \quad + \sum_j ( \tjs (f_j \tilde g_j) - f_j \tjs \tilde g_j), \label{921_e3}
 \end{align}
 where $\tilde g_j =\sum_{a=-2}^2 g_{j+a}$.

 \texttt{Estimate of \eqref{921_e1}}.

 Let $\tilde P_j =\sum_{l=-10}^{10} P_{j+l}$ and $K_j = \tjs \tilde P_j \delta_0$. Then
 \begin{align}
 &\tjs(f_{<j-2}g_j) -f_{<j-2} \tjs g_j \notag \\
 =& \; \int_{\mathbb R^d} K_j(y) ( f_{<j-2}(x-y) -f_{<j-2}(x) ) g_j(x-y) dy \notag \\
 =& \;\int_{\mathbb R^d} K_j(y) ( f_{<j-2}(x-y) -f_{<j-2}(x)
 + \partial f_{<j-2}(x) \cdot y) g_j(x-y) dy \label{921_e1a}\\
 & \;\; - \int_{\mathbb R^d} K_j(y) \partial f_{<j-2}(x) \cdot y g_j(x-y)dy. \label{921_e1b}
 \end{align}

 \underline{Estimate of \eqref{921_e1a}}.

 By Fundamental Theorem of calculus and Lemma \ref{lemp2}, we have
 \begin{align*}
  & |f_{<j-2}(x-y) -f_{<j-2}(x) - \partial f_{<j-2}(x) \cdot (-y) | \notag \\
  \lesssim &\; \int_0^1 | (\partial^2 f_{<j-2})(x-\theta y) | d\theta \cdot |y|^2 \notag \\
  \lesssim & \; \mathcal M( \partial^2 f_{<j-2})(x) \cdot (1+2^j |y|)^d \cdot |y|^2.
  \end{align*}

  Therefore by Lemma \ref{lem921_1},
  \begin{align*}
  &\| \sum_{j\le 0} \eqref{921_e1a} \|_p
  \lesssim \sum_{j\le 0} \| \partial^2 f_{<j-2}\|_p \cdot \| g_j \|_{\infty}
  \lesssim \| J^{s-1} \partial f \|_p \| g\|_{\dot B^0_{\infty,\infty}}.\\
  &\| \sum_{j>0} \eqref{921_e1a} \|_p
  \lesssim \| (2^{js} \mathcal M(\partial^2 f_{<j-2}) \cdot 2^{-2j} \cdot \mathcal Mg_j)_{l_j^2(j>0)} \|_p.
  \end{align*}

  If $0<s\le 1$, then
  \begin{align*}
  \| (2^{-2j} 2^{js} \mathcal M(\partial^2 f_{<j-2}) )_{l_j^2(j>0)} \|_p
  & \lesssim \| (2^{j(s-2)} | \partial^2 f_{<j-2} |)_{l_j^2(j>0)} \|_p \notag \\
  & \lesssim \| \partial^2 f_{<-4} \|_p + \| (2^{j(s-2)} |\partial^2 f_{-4\le \cdot <j-2} |)_{l_j^2(j>0)} \|_p
  \notag \\
  & \lesssim \| J^{s-1} \partial f\|_p + \| D^s f_{\ge -4} \|_p \notag \\
  & \lesssim \| J^{s-1} \partial f \|_p.
  \end{align*}
If $s>1$, then
\begin{align*}
 & \| (2^{-2j} \mathcal M (\partial^2 f_{<j-2}) \cdot 2^{js} \cdot \mathcal Mg_j)_{l_j^2(j>0)} \|_p \notag \\
 \lesssim & \; \| ( \|\partial f \|_{\dot B^0_{\infty,\infty}} \cdot 2^{j(s-1)} \mathcal M g_j)_{l_j^2(j>0)}
 \|_p \notag \\
 \lesssim & \; \| \partial f \|_{\dot B^0_{\infty,\infty}} \| J^{s-2} \partial g \|_p.
 \end{align*}
 Thus
 \begin{align*}
 \| \sum_j \eqref{921_e1a} \|_p
 \lesssim
 \begin{cases}
 \| J^{s-1} \partial f \|_p \| g\|_{\dot B^0_{\infty,\infty}},
 \quad \text{if $0<s\le 1$}, \\
 \| J^{s-1} \partial f \|_p \| g \|_{\dot B^0_{\infty,\infty}}
 + \| \partial f \|_{\dot B^0_{\infty,\infty}} \cdot \| J^{s-2} \partial g \|_p,
 \quad \text{if $s>1$}.
 \end{cases}
 \end{align*}

\underline{Estimate of \eqref{921_e1b}}.

Note that
\begin{align*}
 & \int_{\mathbb R^d} K_j(y) y g_j(x-y) dy \notag \\
 = &\; \mathcal F^{-1} ( \frac 1 {-i} \partial_{\xi} ( \widehat{K_j}(\xi)) \widehat{g_j}(\xi) ) \notag \\
 = & \; \mathcal F^{-1} ( \frac 1 {-i} \partial_{\xi} ( 
 ((1+|\xi|^2)^{\frac s2}-1) \tilde \psi(\frac{\xi}{2^j} ) )
 \psi(\frac{\xi}{2^j} ) \hat g(\xi) ) \notag \\
 =& \; \mathcal F^{-1} ( \frac 1 {-i} \partial_{\xi} ( (1+|\xi|^2)^{\frac s2}) \psi(\frac{\xi}{2^j}) \hat g(\xi) )
 \notag \\
 =&\; s J^{s-2} \partial g_j.
 \end{align*}

 So
 \begin{align*}
 \sum_j \eqref{921_e1b} = -s \sum_{j} \partial f_{<j-2} \cdot J^{s-2} \partial g_j.
 \end{align*}

If $0<s<1$, then
\begin{align*}
 & \| \sum_j \partial f_{<j-2} \cdot J^{s-2} \partial g_j \|_p \notag \\
 \lesssim & \; \sum_{j\le 0} \| \partial f_{<0}  \|_p \| g_j\|_{\infty} \cdot 2^j
 + \| ( |\partial f_{<j-2}| \cdot |J^{s-2} \partial g_j |)_{l_j^2(j>0)} \|_p \notag \\
 \lesssim & \; \| J^{s-1} \partial f \|_p \| g \|_{\dot B^0_{\infty,\infty}}
 + \| ( |\partial  f_{<j-2}| \cdot 2^{j(s-1)} )_{l_j^2(j>0)} \|_p \cdot \| g\|_{\dot B^0_{\infty,\infty}} \notag \\
 \lesssim & \; \| J^{s-1} \partial f \|_p \cdot \| g \|_{\dot B^0_{\infty,\infty}}.
 \end{align*}

 If $s=1$, then by Lemma \ref{lemq1},
 \begin{align*}
 \| \sum_j \partial f_{<j-2} \cdot J^{-1} \partial g_j \|_p \lesssim \| \partial f \|_p \| g\|_{\bmo
}.
 \end{align*}

 If $s>1$, then we write
 \begin{align*}
  \sum_j \partial f_{<j-2} \cdot J^{s-2} \partial g_j
  & = \partial f \cdot J^{s-2} \partial g - \sum_j \partial f_{\ge j-2} \cdot J^{s-2} \partial g_j \notag \\
  & = \partial f \cdot J^{s-2} \partial g - \sum_j \partial f_j \cdot J^{s-2} \partial g_{\le j+2}.
  \end{align*}

  Note that by using Lemma \ref{lemq1}, we have
  \begin{align*}
  \| \sum_j \partial f_j \cdot J^{s-2} \partial g_{\le j+2} \|_p
  \lesssim \| \partial f \|_{\operatorname{BMO}} \cdot \| J^{s-2} \partial g \|_p.
  \end{align*}
  This bound can be improved as we show below.
  
First we deal with the low frequency piece:
  \begin{align*}
  \| \sum_{j\le 10} \partial f_j \cdot J^{s-2} \partial P_{\le j+2}  g
  \|_p & \lesssim 
  \sum_{j \le 10} \| \partial f \|_p 
  \| g\|_{\boo} \cdot 2^j
  \lesssim \| J^{s-1} \partial f \|_p \| g\|_{\boo}.
  \end{align*}
  
  For non-low frequencies, we first decompose $g=g_{\le 1} + g_{>1}$ and bound 
  the piece containing $g_{\le 1}$ as:
  \begin{align*}
  \| \sum_{j>10} \partial f_j 
  \cdot J^{s-2} \partial P_{\le j+2} (g_{\le 1} ) \|_p
 \lesssim \sum_{j>10} \| \partial f_j \|_p \| g\|_{\boo}
 \lesssim \| J^{s-1} \partial f \|_p  \|g\|_{\boo}.
 \end{align*}
  
  For the piece containing $g_{>1}$, we further write
  \begin{align*}
&  \sum_{j>10} 
  \partial f_j \cdot J^{s-2} \partial P_{\le j+2} (g_{>1}) \notag \\
   = & \sum_{j>10} \partial f_j \cdot J^{s-2} \partial P_{\le j-2} (g_{>1})
  + \sum_{j>10} \partial f_j \cdot J^{s-2} \partial P_{j-2<\cdot \le j+2} (g_{>1}) \notag \\
   = & \sum_{j>10} \partial f_j \cdot J^{s-2} \partial P_{>1} P_{\le j-2} g
  + \sum_{j>10} \partial f_j \cdot J^{s-2} \partial \tilde g_j.
  \end{align*}
  Firstly thanks to frequency localisation,
  \begin{align*}
&  \| \sum_{j>10} \partial f_j \cdot J^{s-2} \partial P_{>1} P_{\le j-2} g
  \|_p  \notag \\
   \lesssim & \| ( \partial f_j \cdot J^{s-2} \partial P_{>1} P_{\le j-2} g )_{l_j^2(j>10)} \|_p
  \notag \\
   \lesssim & \| (2^{j(s-1)} \partial f_j )_{l_j^2(j>10)} \|_p \| 
  (2^{-j(s-1)} J^{s-2} \partial P_{>1} P_{\le j-2} g )_{l_j^{\infty} (j>10)} \|_{\infty} \notag \\
  \lesssim & \| J^{s-1} \partial f \|_p \cdot \| g\|_{\boo}.
  \end{align*}
  Then recalling $s>1$,
  \begin{align*}
   & \| \sum_{j>10} \partial f_j \cdot J^{s-2} \partial P_{j-2<\cdot \le j+2} (g_{>1}) \|_p 
   \notag \\
   \lesssim &\; 
   \| (2^{j\frac{s-1}2} \partial f_j )_{l_j^2(j>10)} \|_{2p} 
   \| (2^{-j\frac{s-1}2} J^{s-2} \partial  P_{j-2<\cdot \le j+2} 
   (g_{>1}) )_{l_j^2(j>10)} \|_{2p} \notag \\
   \lesssim & \| J^{s-1} \partial f \|_p^{\frac 12}
   \| \partial f \|_{\boo}^{\frac 12}
   \cdot \|g\|_{\boo}^{\frac 12} \cdot \|J^{s-2} \partial g \|_{p}^{\frac 12} \notag \\
   \lesssim & \| J^{s-1} \partial f \|_p \| g\|_{\boo}
   + \| \partial f \|_{\boo} \| J^{s-2} \partial g \|_p,
   \end{align*}
  where in the second inequality above we have used (a version of) Lemma \ref{lemc2}.

  Thus for $0<s<1$,
  \begin{align*}
  \| \sum_j \eqref{921_e1b} \|_p \lesssim \| J^{s-1} \partial f \|_p \| g \|_{\dot B^0_{\infty,\infty}};
  \end{align*}
  for $s=1$,
  \begin{align*}
  \| \sum_j \eqref{921_e1b} \|_p \lesssim \| J^{s-1} \partial f \|_p \| g \|_{\bmo};
  \end{align*}
  for $s>1$,
  \begin{align*}
  \| (\sum_j \eqref{921_e1b} ) - s \partial f \cdot J^{s-2} \partial g \|_p
  \lesssim  \| J^{s-1} \partial f \|_p \| g\|_{\boo}
   + \| \partial f \|_{\boo} \| J^{s-2} \partial g \|_p.
  \end{align*}

  \texttt{Estimate of \eqref{921_e2}}.

  We first estimate the piece $\sum_j f_j \tjs g_{<j-2}$.

  Clearly by frequency localization,
  \begin{align*}
  \| \sum_j f_j \tjs g_{<j-2} \|_p \lesssim \| (f_j \tjs g_{<j-2})_{l_j^2} \|_p.
  \end{align*}

  For $j\le 0$, by Lemma \ref{lem921_1},
  \begin{align*}
  \| \tjs g_{<j-2} \|_{\infty}  \lesssim \sum_{k<j} 2^{2k} \| g_k \|_{\infty} 
   \lesssim 2^{2j} \| g\|_{\dot B^0_{\infty,\infty}}.
  \end{align*}
  For $j>0$,
  \begin{align*}
  \| \tjs g_{<j-2} \|_{\infty} & \lesssim \| \tjs g_{<0} \|_{\infty} + \sum_{0\le k<j-2}
  2^{ks} \| g_k\|_{\infty} \notag \\
  & \lesssim \| g\|_{\dot B^0_{\infty,\infty}} \cdot 2^{js}.
  \end{align*}

Thus
\begin{align*}
\| (f_j \tjs g_{<j-2} )_{l_j^2} \|_p
& \lesssim ( \| (f_j \cdot 2^{2j} )_{l_j^2(j\le 0)} \|_p +
\| (f_j \cdot 2^{js} )_{l_j^2(j>0)} \|_p ) \cdot \| g \|_{\boo} \notag \\
& \lesssim \| J^{s-1} \partial f \|_p \cdot \| g \|_{\boo}.
\end{align*}

Next we estimate the piece $\sum_j \tjs (f_j g_{<j-2})$. Write
\begin{align*}
  & \sum_{j} \tjs (f_j g_{<j-2}) \notag \\
  = & \; \sum_{j} [\tjs, g_{<j-2}] f_j + \sum_j g_{<j-2} \tjs f_j \notag \\
  = & \; \sum_j [\tjs, g_{<j-2}] f_j + g \tjs f - \sum_j g_{\ge j-2} \tjs f_j \notag \\
  = & \; \sum_j [\tjs,g_{<j-2}] f_j - \sum_j g_j \tjs f_{\le j+2} + g \tjs f.
  \end{align*}

By Lemma \ref{lem921_1},
\begin{align*}
| ([\tjs, g_{<j-2}]f_j)(x)|
& \lesssim \int_{\mathbb R^d} |K_j(y)|
\cdot | g_{<j-2}(x-y)-g_{<j-2}(x)| \cdot |f_{j}(x-y)| dy \notag \\
& \lesssim \| \partial g_{<j-2} \|_{\infty} \cdot \int_{\mathbb R^d}
|K_j(y)| \cdot |y| \cdot |f_j(x-y)| dy \notag \\
& \lesssim \| \partial g_{<j-2} \|_{\infty} \cdot
\begin{cases}
2^j \mathcal M f_j(x), \quad \text{if $j\le 0$},\\
2^{j(s-1)} \mathcal M f_j(x), \quad \text{if $j>0$}.
\end{cases}
\end{align*}

Thus by frequency localization,
\begin{align*}
\| \sum_j [\tjs, g_{<j-2}] f_j \|_p
& \lesssim \| ([\tjs, g_{<j-2}] f_j)_{l_j^2} \|_p \notag \\
& \lesssim \; \|g\|_{\boo} \cdot ( \| (2^{2j} \mathcal M f_j)_{l_j^2(j\le 0)} \|_p
+ \| (2^{js} \mathcal Mf_j)_{l_j^2(j>0)} \|_p ) \notag \\
& \lesssim \; \| g\|_{\boo} \cdot \| J^{s-1} \partial f\|_p.
\end{align*}

By Lemma \ref{lemq1}, it is easy to see that for $s>0$:
\begin{align*}
\| \sum_j g_j \tjs f_{\le j+2} \|_p &\lesssim \| g\|_{\bmo} \| \tjs f\|_p 
 \lesssim \| g\|_{\bmo} \| J^{s-1} \partial f\|_p.
\end{align*}
However we shall improve this bound below in the case $s>1$.  Write $f_{\le j+2} = f_{<j-2} +\tilde f_j$ with
$\tilde f_j := f_{j-2\le \cdot \le j+2}$. Then
\begin{align*}
\| \sum_j &g_j \tjs f_{<j-2} \|_p \notag \\
& \lesssim \| (g_j \tjs f_{<j-2} )_{l_j^2} \|_p \notag \\
& \lesssim \| (g_j 2^j)_{l_j^2(j \le 0)} \|_p \cdot \| \partial f \|_{\boo}
+ \|  (g_j \cdot 2^{j(s-1)} )_{l_j^2(j>0)} \|_p \cdot \| \partial f\|_{\boo}
\notag \\
& \lesssim \| J^{s-2} \partial g\|_p \|\partial f \|_{\boo}.
\end{align*}
For $\tilde f_j$, we bound the low frequency piece as
\begin{align*}
\| \sum_{j\le 10} g_j \tjs \tilde f_j \|_p
& \lesssim \sum_{j\le 10}  \| g_j\|_{\infty} \cdot \| \tilde  f_j \|_p \cdot 2^{2j} 
\notag \\
& \lesssim \| g\|_{\boo} \cdot \| J^{s-1} \partial f \|_p.
\end{align*}
For the non-low frequency piece, we have
\begin{align}
\| \sum_{j>10} g_j \tjs \tilde f_j 
\|p & \lesssim \| (2^{-j(\frac{s-1}2)} \tjs \tilde f_j)_{l_j^2(j>10)} \|_{2p}
\| (2^{j(\frac{s-1}2)} g_j)_{l_j^2(j>10)} \|_{2p} \notag \\
& \lesssim \| J^{s-1} \partial f\|_p^{\frac 12} 
\| \partial f\|_{\boo}^{\frac 12} \|g\|_{\boo}^{\frac 12}
\| J^{s-2} \partial g \|_p^{\frac 12} \notag \\
& \lesssim \| J^{s-1} \partial f \|_p \| g\|_{\boo}
+ \| \partial f \|_{\boo} \| J^{s-2} \partial g\|_p. \label{nonlow_similar_e1135_00}
\end{align}

\texttt{Estimate of \eqref{921_e3}}.

Clearly by Lemma \ref{lem921_1},
\begin{align*}
\| \sum_{j\le 0} \tjs (f_j \tilde g_j) \|_p
& \lesssim \sum_{j\le 0} 2^{2j} \| f_j \|_p \| \tilde g_j \|_{\infty} \notag \\
& \lesssim \| J^{s-1} \partial f \|_p \| g\|_{\boo}.
\end{align*}

Also
\begin{align*}
\| \sum_{j>0} \tjs P_{\le 0} (f_j \tilde g_j) \|_p
& \lesssim \sum_{j>0} \| f_j\|_p \| g_j\|_{\infty} \notag \\
& \lesssim \| J^{s-1} \partial f \|_p \cdot \| g\|_{\boo}.
\end{align*}

On the other hand,
\begin{align*}
\| \sum_{j>0}\tjs P_{>0} (f_j \tilde g_j) \|_p
& \lesssim \| (2^{ks}P_k( \sum_{j>k-4} (f_{>-10})_j \tilde g_j ) )_{l_k^2} \|_p \notag \\
& \lesssim \| (2^{ks} \sum_{j>k-4}|(f_{>-10})_j \tilde g_j |)_{l_k^2} \|_p \notag \\
& \lesssim \| (2^{js} (f_{>-10})_j  \tilde g_j )_{l_j^2} \|_p \notag \\
& \lesssim \| (2^{js} (f_{>-10})_j)_{l_j^2} \|_p \cdot \| g\|_{\boo} \notag \\
& \lesssim \| J^{s-1} \partial f \|_p \cdot \| g\|_{\boo}.
\end{align*}

Similarly
\begin{align*}
\| \sum_{j\le 0} f_j \tjs \tilde g_j \|_p &\lesssim \sum_{j\le 0} 2^{2j} \|f_j\|_p \| \tilde g_j\|_{\infty} \notag \\
& \lesssim \| J^{s-1} \partial f \|_p \cdot \| g\|_{\boo}.
\end{align*}
Also
\begin{align*}
\| \sum_{j>0} f_j \tjs \tilde g_j \|_p & \lesssim
\| D^s f_{>-4} \|_p \| g\|_{\bmo} \notag \\
& \lesssim \| J^{s-1} \partial f \|_p \| g\|_{\bmo}.
\end{align*}
For the case $s>1$, by an estimate similar to \eqref{nonlow_similar_e1135_00},
we have
\begin{align*}
\| \sum_{j>0} f_j \tjs \tilde g_j \|_p 
& \lesssim \| J^{s-1} \partial f \|_p \| g\|_{\boo}
+ \| \partial f \|_{\boo} \| J^{s-2} \partial g\|_p.
\end{align*}
This ends the estimate of the diagonal piece.
\end{proof}

\section{Counterexamples}
In the previous section, we have proved several refined Kato-Ponce type inequalities for the operator $J^s$. 
In particular, we recall that for $1<p<\infty$, $0<s\le 1$:
\begin{align*}
\| J^s(fg) - fJ^s g\|_p \lesssim \| J^{s-1} \partial f \|_p \| g\|_{\infty} \lesssim \| J^s f \|_p \| g\|_{\infty};
\end{align*}
and for $s>1$,
\begin{align*}
\| J^s(fg) - f J^s g \|_p
\lesssim \| J^{s-1} \partial f \|_p \| g \|_{\infty} + \| \partial f \|_{\infty} \| J^{s-2} \partial g \|_p.
\end{align*}
In this section we collect several counterexamples which amounts to showing that in the above inequalities,
the $L^{\infty}$-norms on the RHS cannot be replaced by the weaker BMO norm, not even 
partially. Roughly speaking, Proposition \ref{prop922_1} shows that for $0<s\le 1$, one cannot
hope any quantitative bound of the form
\begin{align*}
\| J^s(fg)- f J^s g \|_p \le F( \| J^s f\|_p, \| g \|_{\operatorname{BMO} } ),
\end{align*}
where $F$ is some function; Similarly  for $1<s\le 1+\frac dp$, one \textbf{cannot} have (see Proposition
\ref{prop922_2})
\begin{align*}
\| J^s (fg) - f J^s g \|_p \le F(\|J^s f \|_p, \| g \|_{\operatorname{BMO}}, \| \partial f \|_{\infty},
\| J^{s-1} g \|_p),
\end{align*}
or (Proposition \ref{prop922_3})
\begin{align*}
\| J^s (fg) - f J^s g \|_p \le F(\|J^s f \|_p, \| g \|_{\operatorname{\infty}}, \| \partial f \|_{\operatorname{BMO}},
\| J^{s-1} g \|_p).
\end{align*}
For $s>1+\frac dp$, Proposition \ref{prop924_1} and Proposition \ref{prop924_2} show that (note
below ``$\lnsim$"!)
\begin{align*}
& \| J^s(fg) -f J^s g \|_p \displaystyle \lnsim
 \| J^s f \|_p \| g \|_{\infty}+ \| J^{s-1} g \|_p \| \partial f \|_{\operatorname{BMO}}, \\
 & \| J^s(fg) - f J^s g \|_p \lnsim
 \| J^s f \|_p \| g \|_{\operatorname{BMO}}
 + \| J^{s-1} g \|_p \| \partial f \|_{\infty}.
\end{align*}
In yet other works $L^{\infty}$ norms cannot be replaced by BMO norms even partially.

\begin{lem} \label{lem922_1}
Let $s>0$ and $1<p<\infty$. Then for any $\phi \in \mathcal S(\mathbb R^d)$ and $k\ge 1$, we have
\begin{align*}
& \| J^s( \phi(x) e^{ikx_1}) - \langle k \rangle^s \phi(x) e^{ikx_1} \|_p \lesssim_{\phi,d,s} k^{s-1}, \\
& \| J^{s-1} \partial_1 ( \phi(x) e^{ikx_1} ) - \langle k \rangle^{s-1} (ik) \phi(x) e^{ik x_1} \|_p
\lesssim_{\phi,d,s} k^{s-1}, \\
& \| D^s( \phi(x) e^{ikx_1} ) - k^s \phi(x) e^{ikx_1} \|_p \lesssim_{\phi,d,s} k^{s-1}.
\end{align*}
Also
\begin{align*}
& \| J^s( \phi(x) e^{ikx_1}) - k^{s} \phi(x) e^{ikx_1} \|_p \lesssim_{\phi,d,s} k^{s-1}, \\
& \| J^{s-1} \partial_1 ( \phi(x) e^{ikx_1} ) - i k^s \phi(x) e^{ik_1} \|_p
\lesssim_{\phi,d,s} k^{s-1}.
\end{align*}

Moreover, if $\phi$ is not identically zero, then there is a constant $C_1=C_1(\phi,p,d)>0$,
$k_0=k_0(\phi,p,d)>0$, such that if $k>k_0$, then
\begin{align}
\| \phi(x) \cos (kx_1) \|_p \ge C_1. \label{lem922_1_e5}
\end{align}
\end{lem}
\begin{proof}[Proof of Lemma \ref{lem922_1}]
Denote $e_1=(1,0,\cdots,0)$. By definition, we have
\begin{align*}
   J^s  (\phi(x) &e^{ikx_1}) \notag \\
   &= \frac 1 {(2\pi)^d} \int_{\mathbb R^d} \langle \xi \rangle^s \hat{\phi} (\xi-k e_1)
  e^{i \xi \cdot x} d\xi \notag \\
  & = \frac 1 {(2\pi)^d} e^{ikx_1} \int_{\mathbb R^d} \langle \xi + k e_1 \rangle^s \hat \phi(\xi) e^{i \xi \cdot x} d\xi \notag \\
  & = \frac 1 {(2\pi)^d} e^{ikx_1} \langle k \rangle^s \int_{\mathbb R^d} \hat \phi (\xi) e^{i \xi \cdot x} d\xi
  \notag \\
  &\quad+\frac 1 {(2\pi)^d} e^{i kx_1} \langle k \rangle^s \int_{\mathbb R^d} \hat \phi(\xi) ( \frac{\langle \xi + ke_1 \rangle^s}
   {\langle k \rangle^s} -1) e^{ i \xi \cdot x} d\xi \notag \\
   & = \langle k \rangle^s \phi(x) e^{i kx_1}
   + \frac 1 {(2\pi)^d}e^{i kx_1} \langle k \rangle^s \int_{\mathbb R^d}
   \hat \phi(\xi) ( \frac{\langle \xi +ke_1 \rangle^s}{\langle k \rangle^s} -1 )e^{i \xi \cdot x} d\xi.
   \end{align*}
Rewrite
\begin{align*}
 & \int_{\mathbb R^d} \hat \phi(\xi) ( \frac{\langle \xi +ke_1 \rangle^s} { \langle k \rangle^s}
 -1) e^{i\xi \cdot x} d\xi \notag \\
 = &\; \int_{\mathbb R^d} \hat \phi(\xi) \chi_{|\xi|\ll k} ( \frac{\langle \xi+ ke_1\rangle^s} {\langle k\rangle^s}-1)
 e^{i\xi \cdot x } d\xi \notag \\
 & \quad + \int_{\mathbb R^d} \hat \phi(\xi) \chi_{|\xi| \gtrsim k} (
 \frac{\langle \xi+ ke_1\rangle^s}{\langle k\rangle^s} -1) e^{i \xi \cdot x} d\xi,
 \end{align*}
 where $\chi_{|\xi|\ll k}$ is a smooth cut-off function localized to the regime $|\xi|\ll k$, and $\chi_{|\xi|\gtrsim k}
 =1-\chi_{|\xi|\ll k}$. Consider first $|x|\lesssim 1$.
 In the regime $|\xi|\ll k$, one can use the factor
 $(\langle \xi +ke_1 \rangle^s \langle k \rangle^{-s} -1)$ to get $1/k$ decay.  In the regime $|\xi|\gtrsim k$,
 one can use the decay of $\hat \phi(\xi)$ to get $1/k$ decay. For $|x| \gtrsim 1$, one can just do repeated integration
 by parts.  It is then easy to check that
 \begin{align*}
 \left| \int_{\mathbb R^d} \hat \phi(\xi)
  ( \frac{\langle \xi +ke_1 \rangle^s} { \langle k \rangle^s}
 -1) e^{i\xi \cdot x} d\xi \right| \lesssim \langle x \rangle^{-10d} \langle k\rangle^{-1}.
 \end{align*}
 From this the desired result follows.

Now for the estimates of the operator $D^s$, we denote by $P_{|\xi| \sim k}$ the frequency
projection to the block $\{ \xi: \, ck < |\xi|< \frac 1 c  k\}$ where $c>0$ is a small constant.
It is easy to check that
\begin{align*}
&\| D^s( ( P_{|\xi| \sim k}  \phi) \cdot e^{ikx_1} ) \|_p
\lesssim k^s \| P_{|\xi| \sim k} \phi \|_p \lesssim 1; \\
& \| k^s (P_{|\xi| \sim k} \phi) e^{ ikx_1} \|_p \lesssim 1.
\end{align*}
Thus one only needs to consider the regime $|\xi| \ll k $ and $|\xi|\gg k$ for which the
factor $|\xi +k e_1|^s$ will cause no trouble when integrating by parts in $\xi$. This part of
the argument is then similar to the $J^s$ case.
The estimates for $J^{s-1} \partial_1$ is also similar. Thus we omit details.

We now prove \eqref{lem922_1_e5}. Note that if $p=2$, then
\begin{align*}
\| \phi(x) \cos k x_1 \|_2^2 = \int_{\mathbb R^d}
|\phi(x)|^2 \frac{1+\cos(2k x_1)} 2 dx.
\end{align*}
Easy to check that for any integer $m\ge 1$,
\begin{align*}
\left| \int_{\mathbb R^d} |\phi(x)|^2 \cos (2k x_1) dx \right| \lesssim_m (k^2+1)^{-m}.
\end{align*}
Thus if $k$ is sufficiently large, we have
\begin{align*}
\| \phi(x) \cos( kx_1) \|_2^2 \gtrsim 1.
\end{align*}

Next if $1<p<2$, then
\begin{align*}
\| \phi(x) \cos kx_1 \|_2 & \lesssim \| \phi (x) \cos kx_1 \|_p^{\frac p2} \| \phi(x) \cos kx_1\|_{\infty}^{1-\frac p2}
\notag \\
& \lesssim \| \phi\|_{\infty}^{1-\frac p2} \| \phi(x) \cos kx_1\|_p^{\frac p2}.
\end{align*}
Thus
\begin{align*}
\| \phi(x) \cos k x_1\|_p \gtrsim 1.
\end{align*}
Similarly the inequality also holds for $2<p<\infty$.

\end{proof}

\begin{lem} \label{lem922_2}
Assume $1<p<\infty$. For any $M>0$, there exists $g \in \mathcal S(\mathbb R^d)$, such that
\begin{align*}
\| \langle \nabla \rangle^{\frac dp} g \|_p \le 1, \quad \| g\|_{\bmo} \le 1,
\end{align*}
but
\begin{align*}
\| g \|_{\infty} >M.
\end{align*}
\end{lem}

\begin{proof}[Proof of Lemma \ref{lem922_2}]
Since $\|g\|_{\bmo} \lesssim_{p,d} \| |\nabla| ^{\frac dp} g \|_p$, we only need to show the existence
of $g \in \mathcal S(\mathbb R^d)$, such that $\| \langle \nabla \rangle^{\frac dp} g \|_p \ll 1$,
and $\|g\|_{\infty} \gg 1$.

To this end, let $\phi \in \mathcal S(\mathbb R^d)$ be such that $\operatorname{supp}(\hat \phi)
\subset\{ \xi:\, \frac 12 <|\xi|<2 \}$, and $\phi(0)= \frac 1 {(2\pi)^d}
\int_{\mathbb R^d} \hat \phi(\xi) d\xi \ne 0$.  Define
\begin{align*}
g(x) =\sum_{j=1}^N \frac 1 j \phi(2^j x).
\end{align*}
Clearly
\begin{align*}
\| g \|_{\infty} \ge |g(0)| > O(\log N) |\phi(0)| >M,
\end{align*}
if $N$ is taken sufficiently large.

On the other hand, if $1<p\le 2$, then
\begin{align*}
  \| \langle \nabla \rangle^{\frac dp} g \|_p
 \lesssim & \| ( \frac 1j \cdot 2^{\frac {jd}p} \phi(2^j x) )_{l_j^2(j: \, 1\le j\le N)} \|_p \notag \\
 \lesssim & \| ( \frac 1 j 2^{\frac {jd} p} \phi(2^j x) )_{l_j^p(j:\, 1\le j\le N)} \|_p \notag \\
 \lesssim & (\frac 1j)_{l_j^p(j:\, 1\le j\le N)} \cdot \|\phi\|_p \notag \\
 \lesssim & \|\phi\|_p \lesssim 1.
 \end{align*}

 If $2<p<\infty$, then
 \begin{align*}
 &\| ( \frac 1 j \cdot 2^{\frac {jd} p} \phi(2^j x) )_{l_j^2(j:\, 1\le j\le N)} \|_p
 \notag \\
 \lesssim & \;  \Bigl(  \frac 1j \| 2^{\frac {jd}p} \phi(2^j x) \|_p \Bigr)_{l_j^2(j:\, 1\le j\le N)} \notag \\
\lesssim & \; (\frac 1j)_{l_j^2(j:\, 1\le j\le N)} \| \phi\|_p \lesssim \|\phi\|_p \lesssim 1.
\end{align*}

Thus multiplying $g$ by a small constant if necessary, we can easily achieve $\|\langle \nabla \rangle^{\frac dp} g
\|_p \le 1$ with $\| g\|_{\infty}>M$.

\end{proof}

\begin{prop} \label{prop922_1}
Assume $0<s\le 1$ and $1<p<\infty$. For any $M>0$, there exist $f$, $g\in \mathcal S(\mathbb R^d)$ such
that
\begin{align*}
\| J^s f\|_p + \| g\|_{\bmo} \le 1,
\end{align*}
but
\begin{align*}
\| J^s(fg) -f J^s g\|_p >M.
\end{align*}
\end{prop}

\begin{proof}[Proof of Proposition \ref{prop922_1}]
By Theorem \ref{thm921_1}, and noting $\|J^{s-1}\partial f \|_p \lesssim \| J^s f\|_p$, we only
need to choose $f$, $g \in \mathcal S(\mathbb R^d)$ such that
\begin{align*}
\| J^s f \|_p + \| g\|_{\bmo} \le 1,
\end{align*}
but
\begin{align*}
\| g( J^s f-f) \|_p >M.
\end{align*}

By Lemma \ref{lem922_2}, we can choose $g\in \mathcal S(\mathbb R^d)$, such that
\begin{align*}
\| g\|_{\bmo} \le \frac 12,
\end{align*}
and for some $x_0 \in \mathbb R^d$, $\delta_0>0$,
\begin{align*}
|g(x)|>N\gg 1, \quad \forall\, |x-x_0|\le \delta_0.
\end{align*}

Then we choose $\phi \in C_c^{\infty}(B(x_0,\frac{\delta_0}2))$, such that $\|\phi\|_p =1$.
Define
\begin{align*}
f(x) = \frac 1 {k^s} \phi(x) \cos (kx_1).
\end{align*}

By Lemma \ref{lem922_1}, it is easy to check that
\begin{align*}
\| J^s f - \phi(x) \cos kx_1\|_p \lesssim_{\phi,d,s} k^{-1}.
\end{align*}

On the other hand, by Lemma \ref{lem922_1},
\begin{align*}
\| g(x) \phi(x) \cos k x_1\|_p \gtrsim N \| \phi(x) \cos k x_1\|_p \gtrsim N.
\end{align*}

Clearly we then have
\begin{align*}
\| g (J^s f-f)\|_p \gtrsim N- O(\frac 1k) -O(\frac 1 {k^s})>M,
\end{align*}
if $N$ and $k$ are sufficiently large.
\end{proof}

The same construction used in Proposition \ref{prop922_1} can be used to obtain the following
corollary. In particular, it shows that the estimate
\begin{align*}
\| D^s(fg) - f D^s g \|_p \lesssim_{s,p,d} \| D^s f \|_p \|g\|_{\infty},
\end{align*}
for $0<s\le 1$, $1<p<\infty$ is sharp.

\begin{cor} \label{prop922_1_cor1}
Assume $1<p<\infty$ and $0<s\le 1$. Then for any $M>0$, there exist $f$, $g\in \mathcal S(\mathbb R^d)$,
such that
\begin{align*}
\| J^s f\|_p + \| g\|_{\bmo} \le 1,
\end{align*}
but
\begin{align*}
\| D^s (fg) - f D^s g\|_p >M.
\end{align*}
\end{cor}

\begin{prop} \label{prop922_2}
Assume $1<p<\infty$ and $1<s\le 1+\frac d p$. Then for any $M>0$, there exist
$f$, $g\in \mathcal S(\mathbb R^d)$ such that
\begin{align*}
\| J^s f \|_p + \| g\|_{\bmo} + \| \partial f \|_{\infty} + \|J^{s-1} g \|_p \le 1,
\end{align*}
but
\begin{align*}
\| J^s(fg) - f J^s g \|_p >M.
\end{align*}
\end{prop}

\begin{proof}[Proof of Proposition \ref{prop922_2}]
By Theorem \ref{thm921_1}, we only need to choose $f$, $g\in \mathcal S(\mathbb R^d)$ such that
\begin{align*}
\| J^s f \|_p + \| g\|_{\bmo} + \| \partial f \|_{\infty} + \|J^{s-1} g\|_p \le 1,
\end{align*}
but
\begin{align*}
\| g \cdot (J^s f -f ) \|_p >M.
\end{align*}

By Lemma \ref{lem922_2}, we can find $g \in \mathcal S(\mathbb R^d)$ such that
\begin{align*}
\| J^{\frac dp} g \|_p \le 1, \quad \text{but $\|g\|_{\infty}>N\gg 1$}.
\end{align*}

Note that $\|g\|_{\bmo} \lesssim \| J^{\frac dp} g\|_p \lesssim 1$, and for $1<s\le 1+\frac dp$,
\begin{align*}
\| J^{s-1} g\|_p \lesssim \| J^{\frac dp} g \|_p \lesssim 1.
\end{align*}

Since $\|g\|_{\infty}>N\gg 1$, we may assume for some $B(x_0, \delta_0)$,
\begin{align*}
|g(x)|>N, \quad \text{for all $x\in B(x_0,\delta_0)$}.
\end{align*}

We then choose $\phi \in C_c^{\infty}(B(x_0,\frac{\delta_0}2))$ with $\|\phi\|_p=1$, and define
\begin{align*}
f(x) = \frac 1 {k^s} \phi(x) \cos kx_1.
\end{align*}

Since $s>1$, it is easy to check that $\| \partial f \|_{\infty} \lesssim k^{-(s-1)} \ll 1$ if
$k$ is large. Also by Lemma \ref{lem922_1},
\begin{align*}
&\| J^s f - \phi(x) \cos kx_1 \|_p+\|f\|_p \lesssim_{\phi, d,s,p} k^{-1}, \\
& \| g(x) \phi(x) \cos k x_1 \|_p \ge N \| \phi(x) \cos kx_1\|_p \gtrsim N.
\end{align*}

Clearly we get
\begin{align*}
\| g (J^s f -f )\|_p >N- O(\frac 1 k) \gg M,
\end{align*}
if $N$ and $k$ are taken sufficiently large.

\end{proof}

\begin{lem} \label{lem922_3}
Assume $1<p<\infty$. For any $M>0$, there exists $f\in \mathcal S(\mathbb R^d)$ such that
\begin{align*}
\| \langle \nabla \rangle^{1+\frac dp} f \|_p \le 1,
\end{align*}
but
\begin{align*}
\| \partial f \|_{\infty}>M.
\end{align*}
\end{lem}
\begin{proof}[Proof of Lemma \ref{lem922_3}]
This is similar to Lemma \ref{lem922_2}. Let $\phi \in \mathcal S(\mathbb R^d)$ be such that
$\operatorname{supp}(\hat \phi) \subset \{ \xi:\; \frac 12 <|\xi|<2\}$ and
\begin{align*}
\int_{\mathbb R^d} \hat \phi(\xi) \frac{\xi_1^2} {|\xi|^2} d\xi >0.
\end{align*}

Define
\begin{align*}
f(x) = \sum_{j=1}^N \frac 1 j (\Delta^{-1} \partial_1 \phi)(2^j x) \cdot 2^{-j}.
\end{align*}

Then
\begin{align*}
|(\partial_1 f)(0)| & \ge ( \sum_{j=1}^N \frac 1j) \cdot | (\Delta^{-1} \partial_1 \partial_1 \phi)(0)| \notag \\
& \ge O(\log N) \gg M,
\end{align*}
if $N$ is taken sufficiently large. The rest of the argument now is similar to that in Lemma \ref{lem922_2}. We omit
details.

\end{proof}

\begin{prop} \label{prop922_3}
Assume $1<p<\infty$ and $1<s\le 1+\frac dp$. Then for any $M>0$, there exist $f$, $g \in \mathcal S(\mathbb R^d)$,
such that
\begin{align*}
\| J^s f\|_p + \| g\|_{\infty} + \| J^{s-1} g\|_p + \| \partial f \|_{\bmo} \le 1,
\end{align*}
but
\begin{align*}
\| J^s (fg) - f J^s g\|_p >M.
\end{align*}
\end{prop}

\begin{proof}[Proof of Proposition \ref{prop922_3}]
Thanks to Theorem \ref{thm921_1}, we only need to choose $f$, $g\in \mathcal S(\mathbb R^d)$ such that
\begin{align*}
\| J^s f \|_p + \| g\|_{\infty} + \| J^{s-1} g \|_p + \| \partial f \|_{\bmo} \le 1,
\end{align*}
but
\begin{align*}
\| \partial f \cdot J^{s-2} \partial g \|_p >M.
\end{align*}

Now by Lemma \ref{lem922_3}, we can choose $f \in \mathcal S(\mathbb R^d)$, such that
\begin{align*}
\| J^s f \|_p \le 1, \quad \text{but $\| \partial_1 f\|_{\infty} >N \gg 1$}.
\end{align*}

Thus for some $B(x_0,\delta_0)$,
\begin{align*}
|(\partial_1 f)(x) |>N\gg 1, \qquad \text{for all $x\in B(x_0,\delta_0)$}.
\end{align*}

Now choose $\phi \in C_c^{\infty}(B(x_0,\frac 12 \delta_0))$ such that $\| \phi\|_p =1$.
Define
\begin{align*}
g(x) =\frac 1 {k^{s-1}} \phi(x) \sin (kx_1).
\end{align*}

By Lemma \ref{lem922_1}, we have (note $s-1>0$ so the hypothesis of Lemma \ref{lem922_1} is
valid for $\tilde s=s-1$),
\begin{align*}
&\| J^{s-2} \partial_1 g - \phi(x) \cos kx_1\|_p \lesssim k^{-1}, \\
& \sum_{j=2}^d \| J^{s-2} \partial_j g \|_p \lesssim k^{-1}.
\end{align*}

From these we get
\begin{align*}
\| \partial f \cdot J^{s-2} \partial g \|_p & \gtrsim \| (\phi(x) \cos k x_1) \partial_1 f\|_p - O(k^{-1}) \notag\\
&\gtrsim \; N- O(k^{-1}) >M,
\end{align*}
if $N$ and $k$ are sufficiently large.

\end{proof}

\begin{prop} \label{prop924_1}
Assume $1<p<\infty$ and $s>1+\frac dp$. Then for any $M>0$, there exist $f$, $g\in \mathcal S(\mathbb R^d)$,
such that
\begin{align*}
\| J^s(fg) - f J^s g \|_p
> M \Bigl( \| J^s f\|_p \| g\|_{\infty} +
\| J^{s-1} g\|_p \| \partial f \|_{\bmo} \Bigr).
\end{align*}
\end{prop}
\begin{proof}[Proof of Proposition \ref{prop924_1}]
By Theorem \ref{thm921_1}, we only need to find $f$, $g$ such that
\begin{align} \label{prop924_1_e1}
\| \partial f \cdot J^{s-2} \partial g \|_p
>M \Bigl( \| J^s f \|_p \| g\|_{\infty} + \| J^{s-1} g\|_p \| \partial f \|_{\bmo} \Bigr).
\end{align}

To this end, we first choose $f\in \mathcal S(\mathbb R^d)$ such that
\begin{align*}
\| \partial f \|_{\bmo} \le 1,
\end{align*}
but
\begin{align*}
\| \partial f \|_{\infty} >M^2 \gg 1.
\end{align*}

Without loss of generality we may assume that for some $B(x_0,\delta_0)$,
\begin{align*}
|(\partial_1 f)(x)|>M^2, \quad \forall\,x \in B(x_0,\delta_0).
\end{align*}

Let $\phi \in C_c^{\infty}(B(x_0,\frac 12 \delta_0))$ be such that $\|\phi\|_p=1$ and define
\begin{align*}
g(x) =\frac 1 {k^{s-1}} \phi(x) \sin (kx_1).
\end{align*}

Then by Lemma \ref{lem922_1}, we have
\begin{align*}
& \| J^{s-2} \partial_1 g - \phi(x) \cos kx_1 \|_p \lesssim_{\phi,d,s,p} k^{-1}, \\
& \sum_{j=2}^d \| J^{s-2} \partial_j g\|_p \lesssim_{\phi,d,s,p} k^{-1}.
\end{align*}

Thus
\begin{align*}
\| \partial f \cdot J^{s-2} \partial g \|_p
&\gtrsim \| \phi(x) \cos k x_1 \cdot \partial_1 f(x)\|_p - O(k^{-1}) \notag \\
& \gtrsim M^2 -O(k^{-1}).
\end{align*}

On the other hand
\begin{align*}
&\|J^s f\|_p \| g\|_{\infty} \lesssim \frac 1 {k^{s-1}} \|\phi\|_{\infty} \| J^s f\|_p, \\
& \| J^{s-1} g \|_p \| \partial f \|_{\bmo} \lesssim 1 + O(\frac 1 k).
\end{align*}

Clearly if $k$ is sufficiently large, then \eqref{prop924_1_e1} follows.
\end{proof}

\begin{prop} \label{prop924_2}
Assume $1<p<\infty$ and $s>1+\frac dp$. Then for any $M>0$, there exist $f$, $g \in \mathcal S(\mathbb R^d)$,
such that
\begin{align*}
\| J^s(fg) -f J^s g\|_p >M ( \| J^s f\|_p \| g\|_{\bmo}
+\| J^{s-1} g\|_p \| \partial f \|_{\infty} ).
\end{align*}
\end{prop}
\begin{proof}[Proof of Proposition \ref{prop924_2}]
Again by Theorem \ref{thm921_1}, we only need to find $f$ and $g \in \mathcal S(\mathbb R^d)$, such
that
\begin{align} \label{prop924_2_e1}
\| g( J^s f-f) \|_p > M ( \| J^s f\|_p \| g\|_{\bmo}
+ \| J^{s-1} g\|_p \| \partial f \|_{\infty}).
\end{align}

Choose $g\in \mathcal S(\mathbb R^d)$ such that
\begin{align*}
\|g\|_{\bmo} \le 1,
\end{align*}
and for some $x_0 \in \mathbb R^d$, $\delta_0>0$,
\begin{align*}
 |g(x)|>M^2, \qquad \forall\, |x-x_0|\le \delta_0.
 \end{align*}

 Let $\phi \in C_c^{\infty}(B(x_0,\frac{\delta_0}2))$ be such that $\|\phi\|_p=1$. Define
 \begin{align*}
 f(x) = \frac 1 {k^s} \phi(x) \cos (kx_1).
 \end{align*}

 Then by Lemma \ref{lem922_1},
 \begin{align*}
 \| J^s f -f - \phi(x) \cos kx_1 \|_p \lesssim_{\phi,d,s,p} k^{-1}.
 \end{align*}

 Then
 \begin{align*}
 \| g( J^s f-f) \|_p & \gtrsim \| g(x) \phi(x) \cos kx_1\|_p - O(k^{-1}) \notag \\
 & \gtrsim M^2 -O(k^{-1}).
 \end{align*}

 On the other hand,
 \begin{align*}
 & \| J^s f\|_p \| g\|_{\bmo} \lesssim 1 + O(k^{-1}), \\
 & \| J^{s-1} g \|_p \| \partial f \|_{\infty} \lesssim \| J^{s-1} g \|_p
 \cdot \frac 1 {k^{s-1} } ( \|\phi\|_{\infty} + \| \partial \phi\|_{\infty}).
 \end{align*}

 Thus \eqref{prop924_2_e1} follows easily.

\end{proof}

\section{Proof of Theorem \ref{thm3}}
Denote $\tjs =J^s-I$ and write
\begin{align*}
\tjs (fg) = \sum_j \tjs (f_{<j-2} g_j) + \sum_{j} \tjs (g_{<j-2} f_j) + \sum_j \tjs ( f_j \tilde g_j),
\end{align*}
where $\tilde g_j = \sum_{a=-2}^2 g_{j+a}$.

\underline{The diagonal piece}.   By Lemma \ref{lem921_1}, we have
\begin{align*}
\| \sum_{j\le 0} \tjs  (f_j \tilde g_j) \|_p & \lesssim \sum_{j\le 0} 2^{2j} \| f_j \tilde g_j \|_p \notag \\
& \lesssim \| J^{s-1} \partial f \|_{\dot B^0_{p_1,\infty} } \|g\|_{p_2}.
\end{align*}
Next
\begin{align*}
 \| \sum_{j>0} \tjs  P_{\le 0} (f_j \tilde g_j) \|_p  \lesssim \sum_{j>0} \|f_j\|_{p_1} \| \tilde g_j\|_{p_2}
\lesssim \| J^{s-1} \partial f \|_{\dot B^0_{p_1,\infty}} \|g\|_{p_2}.
\end{align*}
Then
\begin{align*}
\| \sum_{j>0} \tjs P_{>0} (f_j \tilde g_j) \|_p
& \lesssim \| (P_k \tjs (\sum_{j\ge k-4} f_j \tilde g_j )_{l_k^2(k>0)} \|_p \notag \\
& \lesssim \| (2^{ks}( \sum_{j\ge k-4} f_j \tilde g_j ) )_{l_k^2(k>0)} \|_p \notag \\
& \lesssim \| (2^{js} f_j \tilde g_j)_{l_j^2(j>-10)} \|_p.
\end{align*}
Now consider two cases.

If $p_1<\infty$ and $p_2\le \infty$, then
\begin{align*}
\| (2^{js} f_j \tilde g_j)_{l_j^2(j>-10)} \|_p
& \lesssim \| (2^{js} f_j)_{l_j^2(j>-10)} \|_{p_1} \| g\|_{p_2} \notag \\
& \lesssim \| J^{s-1} \partial f \|_{p_1} \|g\|_{p_2}.
\end{align*}

If $p_1=\infty$, $p_2=p$, then
\begin{align*}
\| (2^{js} f_j \tilde g_j )_{l_j^2(j>-10)} \|_p
& \lesssim \| (2^{js} f_j)_{l_j^{\infty}(j>-10)} \|_{\infty} \| (\tilde g_j)_{l_j^2} \|_{p} \notag \\
& \lesssim \| J^{s-1} \partial f \|_{\dot B^0_{\infty,\infty} } \| g\|_p.
\end{align*}

Collecting the estimates, we get
\begin{align*}
\| \sum_j \tjs (f_j \tilde g_j ) \|_p \lesssim
\begin{cases}
\| J^{s-1} \partial f \|_{p_1} \|g\|_{p_2}, \quad \text{if $p_1<\infty$}; \\
\| J^{s-1} \partial f \|_{\dot B^0_{\infty,\infty}} \| g\|_p, \quad \text{if $p_1=\infty$}.
\end{cases}
\end{align*}

\underline{The low-high piece}.
We first write
\begin{align*}
\tjs (f_{<j-2} g_j) = [\tjs, f_{<j-2}] g_j + f_{<j-2} \tjs g_j.
\end{align*}

Clearly
\begin{align*}
\| \sum_{j\le 0} [\tjs, f_{<j-2}] g_j \|_p & \lesssim \sum_{j\le 0} 2^{j}
\| \partial f_{<0} \|_{p_1} \|g\|_{p_2}
\lesssim \| J^{s-1} \partial f \|_{p_1} \| g\|_{p_2}.
\end{align*}

Let $\tilde P_j = \sum_{l=-10}^{10} P_{j+l}$ and denote $K_j = \tjs \tilde P_j \delta_0$.
Then in the same way as in \eqref{921_e1a}--\eqref{921_e1b}, we have
\begin{align}
  & \tjs (f_{<j-2} g_j) - f_{<j-2} \tjs g_j \notag\\
  = & \int_{\mathbb R^d} K_j(y) ( f_{<j-2}(x-y) -f_{<j-2}(x) + \partial f_{<j-2}(x) \cdot y) g_j(x-y) dy
  \label{930_e1a} \\
  & \quad - \int_{\mathbb R^d} K_j(y) \partial f_{<j-2}(x) \cdot y g_j (x-y) dy. \label{930_e1b}
  \end{align}

\texttt{Estimate of \eqref{930_e1a}}.
First
\begin{align*}
 & | f_{<j-2}(x-y) -f_{<j-2}(x) + \partial f_{<j-2}(x) \cdot y | \notag \\
 \lesssim & \; \mathcal M ( \partial^2 f_{<j-2} )(x) \cdot (1+2^j |y|)^d \cdot |y|^2.
 \end{align*}

 So
 \begin{align*}
 \| \sum_{j>0} \eqref{930_e1a} \|_p
 \lesssim \| (2^{js}\cdot 2^{-2j} \cdot \mathcal M(\partial^2 f_{<j-2}) \cdot \mathcal Mg_j)_{l_j^2(j>0)} \|_p.
 \end{align*}

 If $0<s\le 1$, $p_1<\infty$, and $p_2\le \infty$, then
 \begin{align*}
 &\| ( 2^{j(s-2)} \cdot \mathcal M(\partial^2 f_{<j-2}) \cdot \mathcal M g_j )_{l_j^2(j>0)} \|_p \notag \\
\lesssim &\; \| (2^{j(s-2)} \cdot \mathcal M(\partial^2 f_{<j-2}) )_{l_j^2(j>0)} \|_{p_1}
\cdot \| (\mathcal M g_j )_{l_j^{\infty}} \|_{p_2} \notag \\
\lesssim & \; \| J^{s-1} \partial f \|_{p_1} \| g\|_{p_2}.
\end{align*}

If $0<s\le 1$, $p_1=\infty$, $p_2=p$, then
\begin{align*}
 &\| ( 2^{j(s-2)} \cdot \mathcal M(\partial^2 f_{<j-2}) \cdot \mathcal M g_j )_{l_j^2(j>0)} \|_p \notag \\
\lesssim & \; \| (2^{j(s-2)} \cdot \mathcal M (\partial^2 f_{<j-2}) )_{l_j^{\infty} (j>0)} \|_{\infty}
\cdot \| (\mathcal M g_j)_{l_j^2} \|_p \notag \\
\lesssim & \; \| J^{s-1} \partial f\|_{\infty} \|g\|_p.
\end{align*}

If $s>1$, $p_3\le \infty$ and $p_4<\infty$, we first note
\begin{align*}
| \partial^2 f_{<j-2} |\lesssim \sum_{k<j-2} 2^k |\tilde P_k (\partial f) | \lesssim 2^j \mathcal M(\partial f).
\end{align*}
Then
\begin{align*}
\| (2^{j(s-2)} \cdot \mathcal M(\partial^2 f_{<j-2}) \cdot \mathcal Mg_j)_{l_j^2(j>0)}\|_p
& \lesssim \| \mathcal M (\partial f) \|_{p_3} \cdot \| (2^{j(s-1)} g_j)_{l_j^2(j>0)} \|_{p_4} \notag \\
& \lesssim \| \partial f \|_{p_3} \| J^{s-2} \partial g \|_{p_4}.
\end{align*}

If $s>1$, $p_3=p$ and $p_4=\infty$, then
\begin{align*}
&\| (2^{j(s-2)} \cdot \mathcal M(\partial^2 f_{<j-2}) \cdot \mathcal Mg_j)_{l_j^2(j>0)}\|_p \notag\\
 \lesssim &\; \| (2^{-j} \mathcal M (\partial^2 f_{<j-2}))_{l_j^2(j>0)} \|_{p}
\cdot \| (2^{j(s-1)} \mathcal M g_j)_{l_j^{\infty} (j>0)} \|_{\infty} \notag \\
 \lesssim &\; \| \partial f \|_p \| J^{s-2} \partial g\|_{\infty}.
\end{align*}

Thus \eqref{930_e1a} is OK for us.

\texttt{Estimate of \eqref{930_e1b}}.

In the same way as in \eqref{921_e1b}, we have
\begin{align*}
\sum_{j>0} \eqref{930_e1b} = -s \sum_{j>0} \partial f_{<j-2} \cdot J^{s-2} \partial g_j.
\end{align*}

If $0<s<1$, then
\begin{align*}
  & \| (\partial f_{<j-2} \cdot J^{s-2} \partial g_j)_{l_j^2(j>0)} \|_p \notag \\
  \lesssim & \| \partial f_{\le 0} \|_{p_1} \cdot \| (J^{s-2} \partial g_j)_{l_j^2(j>0)} \|_{p_2}
  + \| ( \partial f_{0<\cdot <j-2} \cdot J^{s-2} \partial g_j )_{l_j^2(j>0)} \|_p.
  \end{align*}
  If $p_2<\infty$, then
  \begin{align*}
  \| \partial f_{\le 0} \|_{p_1} \| (J^{s-2} \partial g_j)_{l_j^2(j>0)} \|_{p_2}
  \lesssim \| J^{s-1} \partial f \|_{p_1} \|g\|_{p_2};
  \end{align*}
and also
  \begin{align*}
   & \| (\partial f_{0<\cdot <j-2} \cdot J^{s-2} \partial g_j)_{l_j^2(j>0)} \|_p \notag \\
   \lesssim & \; \| ( 2^{j(s-1)} \partial f_{0<\cdot <j-2} )_{l_j^{\infty}} \|_{p_1}
   \| (2^{j(1-s)} J^{s-2} \partial g_j)_{l_j^2(j>0)} \|_{p_2} \notag \\
   \lesssim & \; \| J^{s-1} \partial f\|_{p_1} \| g\|_{p_2}.
   \end{align*}

  If $p_2=\infty$, then $p_1=p$, and
  \begin{align*}
  \| \partial f_{\le 0} \|_p \| (J^{s-2} \partial g_j)_{l_j^2(j>0)} \|_{\infty}
  \lesssim \| J^{s-1} \partial f \|_p \| g\|_{\infty};
  \end{align*}
  and also
  \begin{align*}
   & \| (\partial f_{0<\cdot <j-2} \cdot J^{s-2} \partial g_j)_{l_j^2(j>0)} \|_p \notag \\
   \lesssim & \; \| (2^{j(s-1)} \partial f_{0<\cdot <j-2} )_{l_j^2} \|_p
   \| (2^{j(1-s)} J^{s-2} \partial g_j)_{l_j^{\infty}(j>0)} \|_{\infty} \notag \\
   \lesssim & \| J^{s-1} \partial f \|_p \| g\|_{\infty}.
   \end{align*}

Next consider $s=1$.

If $p_2=\infty$, then by Lemma \ref{lemq1},
\begin{align*}
\| \sum_{j>0} \partial f_{<j-2} \cdot J^{-1} \partial g_j \|_p & \lesssim
\| \partial f \|_{p} \| J^{-1} \partial g \|_{\bmo} \notag \\
& \lesssim \| \partial f\|_p \| g\|_{\bmo}.
\end{align*}

If $p_2<\infty$, then
\begin{align*}
\| \sum_{j>0} \partial f_{<j-2} \cdot J^{-1} \partial g_j \|_p
& \lesssim \| ( \partial f_{<j-2} \cdot J^{-1} \partial g_j)_{l_j^2} \|_p \notag \\
& \lesssim \| \partial f \|_{p_1} \| J^{-1} \partial g\|_{p_2}
\lesssim \| \partial f \|_{p_1} \|g\|_{p_2}.
\end{align*}

Next consider $s>1$.

If $p_4=\infty$, then again by Lemma \ref{lemq1},
\begin{align*}
\| \sum_{j>0} \partial f_{<j-2} \cdot J^{s-2} \partial g_j \|_p
& \lesssim \| \partial f\|_p \cdot \| J^{s-2} \partial g \|_{\bmo}.
\end{align*}

If $p_4<\infty$, then clearly
\begin{align*}
\| \sum_{j>0} \partial f_{<j-2} J^{s-2} \partial g_j \|_p
\lesssim \| \partial f\|_{p_3} \| J^{s-2} \partial g \|_{p_4}.
\end{align*}

This ends the estimate of \eqref{930_e1b}. We have finished the estimate of the commutator
piece $[\tjs, f_{<j-2}]g_j$.

To finish the estimate of the low-high piece, we still need to estimate the contribution
of the piece $\sum_{j} f_{<j-2} \tjs g_j$. Write
\begin{align*}
\sum_j f_{<j-2} \tjs g_j & = f \tjs g - \sum_j f_{\ge j-2} \tjs g_j \notag \\
& = f \tjs g - \sum_j f_j \tjs g_{\le j+2}.
\end{align*}

Clearly
\begin{align*}
\|\sum_{j\le 10} f_j \tjs g_{\le j+2} \|_p
\lesssim \sum_{j\le 10} \|f_j\|_{p_1} \cdot 2^{2j} \|g\|_{p_2}
\lesssim \| J^{s-1} \partial f\|_{p_1} \| g\|_{p_2}.
\end{align*}

On the other hand,
\begin{align*}
\sum_{j>10} f_j \tjs g_{\le j+2} = \sum_{j>10} f_j \tjs g_{<j-2} + \sum_{j>10} f_j \tjs \tilde g_j,
\end{align*}
where $\tilde g_j = g_{j-2\le \cdot \le j+2}$.

By frequency localization, for $p_1<\infty$, $p_2\le \infty$, we have
\begin{align*}
 &\| \sum_{j>10} f_j \tjs g_{<j-2} \|_p \notag \\
\lesssim&\; \| (2^{js} f_j \cdot 2^{-js} \tjs g_{<j-2} )_{l_j^2(j>10)} \|_p \notag \\
\lesssim & \; \| (2^{js} f_j)_{l_j^2(j>10)} \|_{p_1} \| (2^{-js} \tjs g_{<j-2} )_{l_j^{\infty}(j>10)} \|_{p_2} \notag \\
\lesssim & \; \| J^{s-1} \partial f\|_{p_1} \| g\|_{p_2}.
\end{align*}
If $p_1=\infty$, $p_2=p$, then
\begin{align*}
 &\| \sum_{j>10} f_j \tjs g_{<j-2} \|_p \notag \\
 \lesssim & \; \| (2^{js} f_j )_{l_j^{\infty}(j>10)} \|_{\infty} \| (2^{-js} \tjs g_{<j-2} )_{l_j^2(j>10)} \|_p
 \notag \\
 \lesssim & \; \| J^{s-1} \partial f \|_{\dot B^0_{\infty,\infty} } \|g\|_p.
 \end{align*}

For the piece $\sum_{j>10} f_j \tjs \tilde g_j$, if $p_1<\infty $ and $p_2<\infty$, then
\begin{align*}
\| \sum_{j>10} f_j \tjs \tilde g_j\|_p
& \lesssim \| (2^{js} f_j)_{l_j^2(j>10)} \|_{p_1} \| (2^{-js} \tjs \tilde g_j )_{l_j^2(j>10)} \|_{p_2} \notag \\
& \lesssim \| J^{s-1} \partial f \|_{p_1} \| g\|_{p_2}.
\end{align*}

If $p_1=p$, $p_2=\infty$, then by Lemma \ref{lemq1},
\begin{align*}
\| \sum_{j>10} f_j \tjs \tilde g_j \|_p \lesssim \| D^s f_{>0} \|_{p} \| D^{-s} \tjs g_{>0} \|_{\bmo}
\lesssim \| J^{s-1} \partial f \|_p \| g\|_{\bmo}.
\end{align*}

Similarly if $p_1=\infty$, $p_2=p$, then
\begin{align*}
\| \sum_{j>10} f_j \tjs  \tilde g_j \|_p \lesssim \| J^{s-1} \partial f \|_{\bmo} \| g\|_p.
\end{align*}

\underline{The high-low piece}.

First
\begin{align*}
\| \sum_{j\le 0} \tjs (g_{<j-2} f_j) \|_p \lesssim \sum_{j\le 0} 2^{2j}
\| g\|_{p_2} \| f_j\|_{p_1}
\lesssim \| J^{s-1} \partial f \|_{p_1} \|g\|_{p_2}.
\end{align*}

If $p_1<\infty$, then
\begin{align*}
\| \sum_{j>0} \tjs (g_{<j-2} f_j) \|_p
\lesssim \| (2^{js} f_j g_{<j-2} )_{l_j^2(j>0)} \|_p
\lesssim \| J^{s-1} \partial f \|_{p_1} \| g\|_{p_2}.
\end{align*}

If $p_1=\infty$, then $p_2=p$. In this case we write
\begin{align*}
\tjs (g_{<j-2} f_j) = [\tjs,g_{<j-2}] f_j + g_{<j-2} \tjs f_j.
\end{align*}

We first estimate the commutator as
\begin{align*}
&\| \sum_{j>0} [\tjs, g_{<j-2}] f_j \|_p \notag \\
\lesssim & \; \| (2^{js} \mathcal M (\partial g_{<j-2}) \cdot 2^{-j} \cdot \mathcal M f_j)_{l_j^2(j>0)} \|_p
\notag \\
\lesssim & \; \| (2^{js} \mathcal M f_j )_{l_j^{\infty}(j>0)} \|_{\infty}
\| (2^{-j} \mathcal M (\partial g_{<j-2}))_{l_j^2(j>0)} \|_p \notag \\
\lesssim & \; \| J^{s-1} \partial f\|_{\dot B^0_{\infty,\infty} } \|g\|_p.
\end{align*}
Finally by Lemma \ref{lemq1},
\begin{align*}
\| \sum_{j>0} g_{<j-2} \tjs f_j \|_p & \lesssim \| \tjs f_{>0} \|_{\bmo} \| g\|_p \notag \\
& \lesssim \| J^{s-1} \partial f \|_{\bmo } \| g\|_p.
\end{align*}
This concludes the proof of Theorem \ref{thm3}.

\section{Further counterexamples}
In this section we collect further counterexamples for the operator $J^s$ on divergence-free velocity
fields which are deeply connected with the investigation of norm estimates of incompressible Euler equations in Sobolev spaces  (see \cite{BL1,BL2, LS16} and the references therein). In typical energy-type estimates for Euler equations, we have
for $s>0$, $1<p<\infty$, and divergence free $u$:
\begin{align*}
\| J^s(  (u\cdot \nabla) u) - (u\cdot \nabla) J^s u\|_p \lesssim_{s,p,d} \| \partial u \|_{\infty}
\| J^s u\|_p.
\end{align*}
A natural question\footnote{We thank Zhuan Ye for raising this question.}
 is whether $\|\partial u \|_{\infty}$ can be replaced by $\| \partial u \|_{\operatorname{BMO}}$.
If this is the case it would yield single exponential in time growth of  Sobolev norms of two-dimensional Euler
thanks to conservation of vorticity.  Proposition \ref{prop_1007_1} and Proposition \ref{prop_1007_2} disprove
any such possibility. In particular, we show that for $1<p<\infty$, $s>1+\frac dp$ and any Schwartz $u$ with 
$\nabla \cdot u=0$, one cannot
hope any quantitative bound of the form
\begin{align*}
\| J^s ( (u\cdot \nabla) u) - (u\cdot \nabla)(J^s u) \|_p \le F( \| \partial u \|_{\operatorname{BMO}}) \cdot \|J^s u\|_p,
\end{align*}
and also for $0<s\le 1+\frac dp$, one cannot have
\begin{align*}
\| J^s ( (u\cdot \nabla) u) - (u\cdot \nabla)(J^s u) \|_p \le F( \| J^s u \|_p, 
\| \partial u \|_{\operatorname{BMO}}).
\end{align*}
An enlightening feature of our construction is the incorporation of the divergence-free constraint. 

\begin{thm} \label{thm921_1a}
Let $s>0$ and $1<p<\infty$. Fix an integer $l \in \{1,\cdots, d\}$. Then the following hold for
any $f$, $g\in \mathcal S(\mathbb R^d)$:
\begin{align*}
 & \| J^s \partial_l (fg) - f J^s \partial_l g - g J^s \partial_l f - \partial f \cdot J^{\partial}g\\
 & \quad - \partial g \cdot J^{\partial} f
 \|_p \lesssim_{s,p,d}
 \| J^s f \|_p \| \partial g \|_{\bmo} + \| \partial f \|_{\bmo} \| J^s g\|_p,
 \end{align*}
 where
 \begin{align*}
  & \widehat{J^{\partial} g}(\xi) = \widehat{J^{\partial}}(\xi) \hat g(\xi), \\
  & \widehat{J^{\partial}}(\xi) = - i \partial_{\xi} ( (\langle \xi \rangle^s -1) \cdot i \xi_l ).
  \end{align*}
  \end{thm}
  \begin{proof}[Proof of Theorem \ref{thm921_1a}]
We shall only sketch the proof since it is similar to the proof of Theorem \ref{thm921_1}.
Define $B=(J^s-1)\partial_l$. Then
\begin{align}
B(fg) -f Bg & = \sum_j ( B(f_{<j-2} g_j) - f_{<j-2} B g_j) \label{921_1a_e1} \\
& \quad + \sum_j ( B(f_j g_{<j-2}) - f_j Bg_{<j-2} ) \label{921_1a_e2} \\
& \quad + \sum_j (B(f_j \tilde g_j) - f_j B\tilde g_j),\label{921_1a_e3}
\end{align}
where $\tilde g_j=\sum_{a=-2}^2 g_{j+a}$.

\texttt{Estimate of \eqref{921_1a_e3}}.
Clearly
\begin{align*}
 \| \sum_{j\le 0} B(f_j \tilde g_j) \|_p
&\lesssim \sum_{j\le 0} 2^{3j} \|f_j \|_p \| \tilde g_j \|_{\infty}
\lesssim \| J^s f \|_p \| \partial g \|_{\dot B^0_{\infty,\infty}}, \\
 \| \sum_{j>0} P_{\le 0} B (f_j \tilde g_j) \|_p
&\lesssim \sum_{j>0} \|f_j \|_p \| \tilde g_j \|_{\infty}
\lesssim \| J^s f\|_p \| \partial g \|_{\boo};\\
 \| \sum_{j>0} P_{>0} B(f_j \tilde g_j) \|_p
 & \lsm \| (2^{k(1+s)} \tilde P_k ( \sum_{j>k-4} f_j \tilde g_j ) )_{l_k^2(k>0)} \|_p
 \notag \\
 & \lsm \| (2^{j(1+s)} f_j \tilde g_j)_{l_j^2(j>-10)} \|_p \lsm \| J^s f \|_p \| \partial g \|_{\boo}.
 \end{align*}

Similarly
\begin{align*}
\| \sum_{j\le 0} f_j B \tilde g_j \|_p & \lsm \| J^s f \|_p \| \partial g \|_{\boo},\\
\end{align*}
and by Lemma \ref{lemq1},
\begin{align*}
\| \sum_{j>0} f_j B \tilde g_j \|_p & = \|\sum_{j} 2^{-js} \tilde P_j D^s f_{>0} 2^{js}
 \tilde P_j ( (J^s-1) \partial_l D^{-s} g_{>-4} ) \|_p \notag \\
& \lsm \| D^s f \|_p \| D^{-s} (J^s-1) \partial g_{>-4} \|_{\bmo} \notag \\
& \lsm \| J^s f \|_p \| \partial g \|_{\bmo}.
\end{align*}
Thus \eqref{921_1a_e3} is OK for us.

\texttt{Estimate of \eqref{921_1a_e1}}.

Let $K_j = \sum_{a=-10}^{10} BP_{j+a} \delta_0$. Then
\begin{align}
 &B(f_{<j-2} g_j) -f_{<j-2} Bg_j \notag \\
= & \; \int K_j(y) (f_{<j-2}(x-y) -f_{<j-2}(x) -\partial f_{<j-2}(x) \cdot (-y) ) g_j(x-y) dy
\label{921_1a_e6}\\
& \qquad + \int K_j(y) \partial f_{<j-2}(x) \cdot (-y) g_j(x-y) dy.
\label{921_1a_e7}
\end{align}

\underline{Estimate of \eqref{921_1a_e6}}.

By using
\begin{align*}
&| f_{<j-2}(x-y) -f_{<j-2}(x)- \partial f_{<j-2} (x) \cdot (-y)| \notag \\
\lsm & \mathcal M(\partial^2 f_{<j-2})(x) \cdot (1+2^j |y|)^d \cdot |y|^2,
\end{align*}
we get
\begin{align*}
|\eqref{921_1a_e6}| \lsm
\mathcal M(\partial^2 f_{<j-2})(x) \cdot (\mathcal M g_j)(x)
\cdot
\begin{cases}
2^j, \quad \text{if $j\le 0$},\\
2^{j(s-1)}, \quad \text{if $j>0$}.
\end{cases}
\end{align*}

Therefore by frequency localization,
\begin{align*}
& \| \sum_{j} \eqref{921_1a_e6} \|_p \notag \\
\lsm & \; \| ( \mathcal M (\partial^2 f_{<j-2}) \mathcal Mg_j 2^j )_{l_j^2(j\le 0)} \|_p
+ \| ( \mathcal M(\partial^2 f_{<j-2}) \mathcal Mg_j 2^{j(s-1)})_{l_j^2(j>0)} \|_p
\notag \\
\lsm & \; \| \partial f \|_{\boo} \| J^s g\|_p.
\end{align*}

\underline{Estimate of \eqref{921_1a_e7}}.

Easy to check that
\begin{align*}
  & \int K_j(y) (-y) g_j(x-y) dy \notag \\
  = & \; \frac 1 i \mathcal F^{-1} ( \partial_{\xi} ( \widehat{K_j}(\xi) \widehat{g_j}(\xi) )
  = J^{\partial} g_j,
  \end{align*}
  where we recall
  \begin{align*}
  \widehat{J^{\partial}}(\xi)
  =  \frac 1 i \partial_{\xi} ( \hat B(\xi)) = -i
  \partial_{\xi} ( (\langle \xi \rangle^s -1) i \xi_l).
  \end{align*}

Then
\begin{align*}
\sum_j \eqref{921_1a_e7}
& = \sum_{j} \partial f_{<j-2} \cdot J^{\partial} g_j \notag \\
& = J^{\partial} g \cdot \partial f - \sum_{j} J^{\partial} g_{\le j+2} \cdot \partial f_j.
\end{align*}
Clearly by Lemma \ref{lemq1},
\begin{align*}
\| \sum_{j} J^{\partial} g_{\le j+2}\cdot \partial f_j
\|_p
& \lsm \| J^{\partial} g\|_p \| \partial f \|_{\bmo} \notag\\
& \lsm \| J^s f\|_p \| \partial f \|_{\bmo}.
\end{align*}
Thus
\begin{align*}
\sum_j\eqref{921_1a_e7} = J^{\partial} g\cdot \partial f +\operatorname{OK},
\end{align*}
where
\begin{align*}
\| \operatorname{OK}\|_p \lsm \| J^s f\|_p \| \partial g \|_{\bmo}
+ \| J^s g\|_p \| \partial f \|_{\bmo}.
\end{align*}

\underline{Estimate of \eqref{921_1a_e2}}.

First note that by swapping $f$ and $g$, and using the previous estimates, we have
\begin{align*}
\sum_j ( B(f_j g_{<j-2}) -g_{<j-2} Bf_j ) = J^{\partial} f \cdot \partial g +\operatorname{OK}.
\end{align*}
On the other hand,
\begin{align*}
\sum_j g_{<j-2} B f_j &= gBf - \sum_j g_j Bf_{\le j+2}, \\
\| \sum_{j\le 0} g_j B f_{\le j+2} \|_p
& \lsm \sum_{j\le 0} 2^{3j} \| g_j \|_{\infty} \| f\|_p \lsm \| J^s f \|_p \| \partial g \|_{\boo},\\
\| \sum_{j>0} g_j B f_{\le j+2} \|_p
& \lsm \| \partial g \|_{\bmo} \| J^s f\|_p.
\end{align*}
Also similarly
\begin{align*}
\| \sum_j f_j B g_{<j-2} \|_p \lsm \| \partial f \|_{\bmo} \| J^s g\|_p.
\end{align*}
Thus
\begin{align*}
\eqref{921_1a_e2} = J^{\partial} f \cdot \partial g + gBf +\operatorname{OK}.
\end{align*}
The desired estimates then follow from the simple identity that
\begin{align*}
B(fg) - f Bg -g Bf
= J^s \partial_l (fg) - f J^s \partial_l g - g J^s \partial_l f.
\end{align*}

\end{proof}

An immediate corollary of Theorem \ref{thm921_1a} is the following estimate.

\begin{cor}
Let $s>0$ and $1<p<\infty$. Then for any $u \in \mathcal S(\mathbb R^d)$ with $\nabla \cdot u=0$, we have
\begin{align*}
\| J^s(  (u\cdot \nabla) u) - (u\cdot \nabla) J^s u\|_p \lesssim_{s,p,d} \| \partial u \|_{\infty}
\| J^s u\|_p.
\end{align*}
\end{cor}
\begin{proof}
Obvious.
\end{proof}

\begin{prop}\label{prop_1007_1}
Let $1<p<\infty$ and $s>1+\frac dp$. Then for any $M>0$, there exists
$u \in \mathcal S(\mathbb R^d)$ with $\nabla \cdot u=0$, such that
\begin{align*}
\| \partial u \|_{\bmo} \le 1,
\end{align*}
but
\begin{align*}
\| J^s ( (u\cdot \nabla) u) - (u\cdot \nabla) (J^s u) \|_p > M \| J^s u\|_p.
\end{align*}
\end{prop}
\begin{proof}[Proof of Proposition \ref{prop_1007_1}]
We shall choose $u$ in the form
\begin{align*}
u=(u_1,u_2,0,\cdots,0).
\end{align*}
Then we only need to show\footnote{Here we choose $u_2$ for convenience only. One can of course
choose $u_1$ as well.}
\begin{align*}
\| \sum_{l=1}^2 J^s \partial_l (u_l u_2) - \sum_{l=1}^2 u_l \partial_l J^s u_2 \|_p
> M \|J^s u\|_p.
\end{align*}
By Theorem \ref{thm921_1a}, we have
\begin{align*}
\| \sum_{l=1}^2 &J^s \partial_l (u_l u_2) - \sum_{l=1}^2 u_l J^s \partial_l u_2
- \sum_{l=1}^2 (J^s \partial_l u_l) u_2  -\sum_{l=1}^2 \partial u_l
\cdot ( (J^s-1) \partial_l)^{\partial} u_2 \notag \\
& - \sum_{l=1}^2 ( (J^s-1) \partial_l)^{\partial} u_l \cdot \partial u_2\|_p
\lsm \| J^s u \|_p \| \partial u \|_{\bmo},
\end{align*}
where
\begin{align*}
  \mathcal F{\bigl( ( (J^s-1) \partial_l)^{\partial} \bigr)} (\xi)
=& - i \partial_{\xi} ( (\langle \xi \rangle^s-1) i \xi_l) \notag \\
= & \; s \langle \xi \rangle^{s-2} \xi \xi_l
+ (\langle \xi \rangle^s -1) e_l, \qquad e_l=(\underbrace{0,\cdots,0,1}_l,0\cdots,0), \\
 ( (J^s-1) \partial_l)^{\partial} & = -s J^{s-2} \partial \partial_l + e_l (J^s-1).
\end{align*}
Clearly by using $\nabla \cdot u=0$, we get
\begin{align*}
& \sum_{l=1}^2 \partial u_l \cdot ( (J^s-1) \partial_l)^{\partial} u_2
= -s \sum_{l=1}^2 \partial u_l \cdot J^{s-2} \partial \partial_l u_2, \\
& \sum_{l=1}^2 ( (J^s-1) \partial_l)^{\partial} u_l \cdot \partial u_2
= \sum_{l=1}^2 ( (J^s-1) u_l) \partial_l u_2.
\end{align*}
Thus it suffices to show
\begin{align*}
\| s\sum_{l=1}^2 \partial u_l
\cdot J^{s-2} \partial \partial_l u_2
- \sum_{l=1}^2 \partial_l u_2 (J^s-1) u_l \|_p
> M \|J^s u\|_p.
\end{align*}
For convenience of notation, for $f=(f_1,f_2)$, $g=(g_1,g_2)$, denote
\begin{align*}
B(f,g) = s \sum_{l=1}^2 \partial f_l \cdot J^{s-2} \partial \partial_l g_2
-\sum_{l=1}^2 \partial_l f_2 (J^s-1) g_l.
\end{align*}
Now choose $\tilde \phi \in \mathcal S(\mathbb R^d)$, such that
$\operatorname{supp}(\widehat {\tilde \phi}) \subset \{ \xi:\; 2/3 <|\xi| <1\}$, and
\begin{align*}
\int_{\mathbb R^d} \widehat{\tilde \phi} (\xi) \xi_1 \xi_2 d\xi >0.
\end{align*}
The last condition is to ensure that $|(\partial_{12} \tilde \phi)(0)| \ne 0$.
Now define
\begin{align*}
\phi(x) = \sum_{\tilde l=1}^m \tilde \phi(2^{3\tilde l} x) 2^{-6\tilde l}\cdot \frac 1 {\tilde l},
\end{align*}
where $m$ is chosen sufficiently large such that $\sqrt{\log m} \gg M$. Easy to check that
\begin{align*}
& | (\partial_{12} \phi)(0)| \sim \log m, \\
& \sum_{i,j=1}^d\| \partial_i \partial_j \phi\|_{\bmo} \lesssim
\sum_{i,j=1}^d \| D^{\frac d 2} \partial_i \partial_j \phi\|_2 \lsm 1.
\end{align*}
Set $u^o = (u^o_1, u^o_2,0,\cdots, 0)$, and
\begin{align*}
u^o_1 = - \partial_2 \phi, \quad u_2^o = \partial_1 \phi.
\end{align*}
Note that $(\partial_1 u^o_1)(0) = -(\partial_2 u^o_2)(0)  =-(\partial_{12} \phi)(0)\sim \log m$.
Now discuss two cases.

\underline{Case 1}: $\|B(u^o,u^o)\|_p \ge \sqrt{\log m} \| J^s u^o\|_p$.
In this case we set $u=u^o$ and no work is needed.

\underline{Case 2}: $\| B(u^o,u^o)\|_p < \sqrt{\log m} \| J^s u^o \|_p$.
In this case we shall do a further perturbation argument.

First by continuity, we can find $\delta_0 >0$ such that
\begin{align*}
| (\partial_{12} \phi)(x) | > \frac 12 |(\partial_{12} \phi)(0)| \sim \log m, \quad
\forall\, |x|<\delta_0.
\end{align*}

Now choose $b \in C_c^{\infty}(B(0,\frac{\delta_0}2))$ such that
\begin{align*}
\| b\|_p = \frac 1{100} \| J^s u^o\|_p.
\end{align*}

For $k\gg 1$, define
\begin{align*}
 & \phi^n = \frac 1 {k^{1+s}}( \sin kx_1) b(x), \\
 & u^n = (-\partial_2 \phi^n, \partial_1 \phi^n, 0,\cdots, 0), \\
 & u=u^o+u^n.
 \end{align*}
Then by taking $k$ large, it is easy to check that $\|J^s u^n \|_p \lsm \| J^s u\|_p$, and
\begin{align*}
 \| B(u^n,u^n) \|_p & \lsm_{b,d,s,p} \frac 1 {k^{s-1}} \ll 1, \\
 \|B(u^n,u^o)\|_p & = \| s \sum_{l=1}^2 \partial u_l^n \cdot J^{s-2} \partial \partial_l u_2^o
 -\sum_{l=1}^2 \partial_l u_2^n \cdot (J^s-1) u_l^o \|_p \notag \\
 & \lsm_{b,d,s,p,\phi} \frac 1 {k^{s-1}} \ll 1, \\
 B(u^o,u^n) &= s\sum_{l=1}^2 \partial u_l^o \cdot J^{s-2} \partial \partial_l u_2^n
 -\sum_{l=1}^2 \partial_l u_2^o \cdot (J^s-1) u_l^n \notag \\
 & = s \partial_1 u^o_1 J^{s-2} \partial_1 \partial_1 u_2^n -\partial_2 u_2^o J^s u_2^n
 + \operatorname{error}, 
 \end{align*}
 where
 \begin{align*}
 &\| \operatorname{error} \|_p \lesssim_{\tilde \phi, d,s,p} 1.
 \end{align*}
On the other hand,
\begin{align*}
\| J^{s-2} \partial_1 \partial_1 u_2^n - (-1) \cos (kx_1) b(x) \|_p
& \lesssim_{b,d,s,p} \frac 1k, \\
\| J^s u_2^n -\cos (kx_1) b(x) \|_p \lesssim_{b,d,s,p} \frac 1k.
\end{align*}
Thus
\begin{align*}
\| B(u^o,u^n)\|_p \gtrsim_{d,s,p} (s-1) \log m \|b\|_p,
\end{align*}
and it follows easily that
\begin{align*}
\|B(u,u)\|_p \gtrsim_{d,s,p} ((s-1)\log m- \sqrt{\log m} ) \| J^s u^o\|_p >M \|J^s u\|_p.
\end{align*}
\end{proof}

\begin{prop} \label{prop_1007_2}
Let $1<p<\infty$ and $0<s\le 1+\frac dp$. Then for any $M>0$, there exists
$u \in \mathcal S(\mathbb R^d)$ with $\nabla \cdot u=0$ such that
\begin{align*}
\| J^s u\|_p + \| \partial u \|_{\bmo} \le 1,
\end{align*}
but
\begin{align*}
\| J^s ( (u\cdot \nabla ) u) - (u\cdot \nabla ) (J^s u) \|_p >M.
\end{align*}
\end{prop}
\begin{proof}[Proof of Proposition \ref{prop_1007_2}]
We only need to alter slightly the construction in the proof of Proposition
\ref{prop_1007_1}. We use the same notation as therein and define
\begin{align*}
& \phi(x) = \sum_{ l=1}^m \tilde \phi(2^{3 l} x) 2^{-6 l}\cdot \frac 1 { l}, \\
& u^o = (u_1^o,u_2^o,0,\cdots,0), \\
& u_1^o=-\partial_2 \phi, \quad u_2^o = \partial_1 \phi.
\end{align*}
Easy to check that $\| J^{1+\frac dp} u^o\|_p \lesssim 1$.

\underline{Case 1}: $\|B(u^o,u^o)\|_p \ge \sqrt{\log m}$. No work is needed and we can take $u=u^o$.

\underline{Case 2}: $\| B(u^o,u^o)\|_p < \sqrt{\log m}$. In this case we just choose
$b\in C_c^{\infty}(B(0,\frac {\delta_0}2))$ such that $\|b\|_p =1$. The rest of the argument is then similar
to that in the proof of Proposition \ref{prop_1007_1}. We omit details.

\end{proof}

\end{document}